\pdfoutput=1
\documentclass[a4paper]{scrartcl}

\usepackage[utf8]{inputenc}
\usepackage[T1]{fontenc}

\usepackage{amsmath, amsthm, amssymb}
\usepackage{mathtools}
\usepackage{mathrsfs}
\usepackage{bm}
\usepackage{url}
\usepackage{enumitem}
\newlist{hypothenum}{enumerate}{3}
\setlist[hypothenum,1]{label=(\roman*)}
\usepackage[all,cmtip]{xy}
\usepackage{tikz-cd}
\usepackage{tikz}
\usetikzlibrary{calc}
\usetikzlibrary{patterns}
\usetikzlibrary{decorations.pathmorphing}
\usepackage{relsize}
\usepackage{tabularx}
\newcolumntype{C}{>{\centering\arraybackslash}X} 
\usepackage{multirow}
\usepackage{arydshln}
\usepackage[referable]{threeparttablex}
\renewlist{tablenotes}{enumerate}{1}
\makeatletter
\setlist[tablenotes]{label=\tnote{\alph*},ref=\alph*,itemsep=\z@,topsep=\z@skip,partopsep=\z@skip,parsep=\z@,itemindent=\z@,labelindent=\tabcolsep,labelsep=.2em,leftmargin=*,align=left,before={\footnotesize}}
\makeatother
\usepackage{longtable}
\usepackage{needspace}
\newcommand{\innerneedspace}{\vspace{\dimexpr-\baselineskip-\parskip\relax}\needspace}
\usepackage{caption}
\usepackage{subcaption}

\usepackage{afterpage}

\numberwithin{equation}{section}

\theoremstyle{plain}
\newtheorem*{mainthm}{Main Theorem}
\newtheorem{theorem}{Theorem}[section]
\newtheorem{proposition}[theorem]{Proposition}

\newtheorem{lemma}[theorem]{Lemma}
\newtheorem{corollary}[theorem]{Corollary}

\newtheorem{assumption}[theorem]{Assumption}

\newtheorem{schema}[theorem]{Schema}

\theoremstyle{definition}
\newtheorem{definition}[theorem]{Definition}
\newtheorem{example}[theorem]{Example}

\theoremstyle{remark}
\newtheorem{remark}[theorem]{Remark}

\newtagform{starred}{(}{*)}

\newcommand{\eps}{\varepsilon}
\newcommand{\setsuch}[2]{\left\{ #1 \; \middle| \; #2 \right\}}
\newcommand{\restr}[2]{{\left. #1 \right|}_{#2}}
\newcommand{\subl}{{\mathsmaller{<}}}
\newcommand{\subg}{{\mathsmaller{>}}}
\newcommand{\sube}{{\mathsmaller{=}}}
\newcommand{\suble}{{\mathsmaller{\leq}}}
\newcommand{\subge}{{\mathsmaller{\geq}}}
\newcommand{\subenz}{\Bumpeq}
\newcommand{\subll}{{\mathsmaller{\ll}}}
\newcommand{\subgg}{{\mathsmaller{\gg}}}
\newcommand{\subap}{{\mathsmaller{\approx}}}
\newcommand{\sublap}{{\mathsmaller{\lesssim}}}
\newcommand{\subgap}{{\mathsmaller{\gtrsim}}}
\newcommand{\transl}{t}

\newcommand{\extra}{a}
\newcommand{\ext}{\mathsf{\Lambda}}

\newcommand{\bigo}{\mathcal{O}}

\DeclareMathOperator{\Conv}{Conv}
\DeclareMathOperator{\Aff}{Aff}
\DeclareMathOperator{\Id}{Id}
\DeclareMathOperator{\Ad}{Ad}
\DeclareMathOperator{\Stab}{Stab}
\DeclareMathOperator{\GL}{GL}

\DeclareMathOperator{\SL}{SL}
\DeclareMathOperator{\PSL}{PSL}
\DeclareMathOperator{\SO}{SO}
\DeclareMathOperator{\PSO}{PSO}

\DeclareMathOperator{\Orth}{O}
\newcommand{\fundef}[5]{
\entrymodifiers={+!!<0pt,\fontdimen22\textfont2>}
\xymatrix@R=3pt{\llap{$#1$\;\;} {#2} \ar@{->}[r] & {#3} \\ {#4} \ar@{|->}[r] & {#5}}
} 
\newcommand{\noqed}{\renewcommand{\qedsymbol}{}}

\newcommand{\ie}{i.e.\ }
\newcommand{\eg}{e.g.\ }

\newcommand{\longlongrightarrow}{\xrightarrow{\hspace*{1cm}}}
\newcommand{\jordan}{\operatorname{Jd}}
\newcommand{\cartan}{\operatorname{Ct}}

\begin{document}

\title{Proper affine actions: a sufficient criterion}
\author{Ilia Smilga}
              
\maketitle

\begin{abstract}
For a semisimple real Lie group $G$ with a representation $\rho$ on a finite-dimensional real vector space $V$, we give a sufficient criterion on~$\rho$ for existence of a group of affine transformations of $V$ whose linear part is Zariski-dense in $\rho(G)$ and that is free, nonabelian and acts properly discontinuously on~$V$. This new criterion is more general than the one given in the author's previous paper \emph{Proper affine actions in non-swinging representations} (submitted; available at~\url{arXiv:1605.03833}), insofar as it also deals with ``swinging'' representations. When $G$ is split, almost all the irreducible representations of~$G$ that have~$0$ as a weight satisfy this criterion. We conjecture that it is actually a necessary and sufficient criterion.
\end{abstract}


\tableofcontents

\section{Introduction}

\subsection{Background and motivation}
\label{sec:background}

The present paper is part of a larger effort to understand discrete groups $\Gamma$ of affine transformations (subgroups of the affine group $\GL_n(\mathbb{R}) \ltimes \mathbb{R}^n$) acting properly discontinuously on the affine space $\mathbb{R}^n$. The case where $\Gamma$ consists of isometries (in other words, $\Gamma \subset \Orth_n(\mathbb{R}) \ltimes \mathbb{R}^n$) is well-understood: a classical theorem by Bieberbach says that such a group always has an abelian subgroup of finite index.

We say that a group $G$ acts \emph{properly discontinuously} on a topological space $X$ if for every compact $K \subset X$, the set $\setsuch{g \in G}{g K \cap K \neq \emptyset}$ is finite. We define a \emph{crystallographic} group to be a discrete group $\Gamma \subset \GL_n(\mathbb{R}) \ltimes \mathbb{R}^n$ acting properly discontinuously and such that the quotient space $\mathbb{R}^n / \Gamma$ is compact. In \cite{Aus64}, Auslander conjectured that any crystallographic group is virtually solvable, that is, contains a solvable subgroup of finite index. Later, Milnor \cite{Mil77} asked whether this statement is actually true for any affine group acting properly discontinuously. The answer turned out to be negative: Margulis \cite{Mar83, Mar87} gave a nonabelian free group of affine transformations with linear part Zariski-dense in $\SO(2, 1)$, acting properly discontinuously on $\mathbb{R}^3$. On the other hand, Fried and Goldman \cite{FG83} proved the Auslander conjecture in dimension 3 (the cases $n=1$ and $2$ are easy). Recently, Tomanov \cite{Tom16} proved it in dimension $n \leq 5$, and independently Abels, Margulis and Soifer \cite{AMS12} proved it in dimension $n \leq 6$.  See \cite{AbSur} for a survey of already known results.

Margulis's breakthrough was soon followed by the construction of other counterexamples to Milnor's conjecture. Most of these counterexamples have been free groups. (Only recently, Danciger, Gu{\'e}ritaud and Kassel~\cite{DGK20} found examples of affine groups acting properly discontinuously that are neither virtually solvable nor virtually free.) The author focuses on the case of free groups, asking the following question. Consider a semisimple real Lie group~$G$; for every representation~$\rho$ of~$G$ on a finite-dimensional real vector space~$V$, we may consider the affine group~$G \ltimes V$. Does that affine group contain a nonabelian free subgroup with linear part Zariski-dense in~$G$ and acting properly discontinuously on~$V$? (More precisely, for which values of~$G$ and~$V$ is the answer positive?)

Here is a summary of previous work on this question:
\begin{itemize}
\item Margulis's original work~\cite{Mar83, Mar87} gave a positive answer for $G = \SO^+(2,1)$ acting on~$V = \mathbb{R}^3$.
\item Abels, Margulis and Soifer generalized this, giving~\cite{AMS02} a positive answer for $G = \SO^+(2n+2,2n+1)$ acting on~$V = \mathbb{R}^{4n+3}$ (for every integer $n \geq 0$).
\item They showed later~\cite{AMS11} that for all other values of $p$ and~$q$, the answer for $G = \SO^+(p,q)$ acting on~$V = \mathbb{R}^{p+q}$ is negative.
\item The author of this paper gave~\cite{Smi14} a positive answer for any noncompact semisimple real Lie group~$G$ acting (by the adjoint representation) on its Lie algebra~$\mathfrak{g}$.
\item Recently, he gave~\cite{Smi16} a simple algebraic criterion on~$G$ and~$V$ guaranteeing that the answer is positive. However, this criterion included an additional assumption about~$V$ (namely that the representation is ``non-swinging'': see Definition~\ref{simpifying_assumptions_def}.\ref{itm:non-swinging_def}), which is not in fact necessary.
\end{itemize}
This paper gives a better sufficient condition on~$G$ and~$V$ for the answer to be positive, that gets rid of this ``non-swinging'' assumption. The author conjectures that the new condition is in fact necessary, thus giving a complete classification of such counterexamples.

In order to state this condition, we need to introduce a few classical notations.

\subsection{Basic notations}
\label{sec:lie}

For the remainder of the paper, we fix a semisimple real Lie group~$G$; let $\mathfrak{g}$~be its Lie algebra. Let us introduce a few classical objects related to~$\mathfrak{g}$ and~$G$ (defined for instance in Knapp's book \cite{Kna96}, though our terminology and notation differ slightly from his).

We choose in $\mathfrak{g}$:
\begin{itemize}
\item a Cartan involution $\theta$. Then we have the corresponding Cartan decomposition $\mathfrak{g} = \mathfrak{k} \oplus \mathfrak{q}$, where we call $\mathfrak{k}$ the space of fixed points of $\theta$ and $\mathfrak{q}$ the space of fixed points of $-\theta$. We call $K$ the maximal compact subgroup with Lie algebra $\mathfrak{k}$.
\item a \emph{Cartan subspace} $\mathfrak{a}$ compatible with $\theta$ (that is, a maximal abelian subalgebra of $\mathfrak{g}$ among those contained in $\mathfrak{q}$). We set $A := \exp \mathfrak{a}$.
\item a system $\Sigma^+$ of positive restricted roots in $\mathfrak{a}^*$. Recall that a \emph{restricted root} is a nonzero element $\alpha \in \mathfrak{a}^*$ such that the restricted root space
\[\mathfrak{g}^\alpha := \setsuch{Y \in \mathfrak{g}}{\forall X \in \mathfrak{a},\; [X, Y] = \alpha(X)Y}\]
is nontrivial. They form a root system $\Sigma$; a system of positive roots $\Sigma^+$ is a subset of $\Sigma$ contained in a half-space and such that $\Sigma = \Sigma^+ \sqcup -\Sigma^+$. (Note that in contrast to the situation with ordinary roots, the root system $\Sigma$ need not be reduced; so in addition to the usual types, it can also be of type~$BC_n$.)

We call~$\Pi$ the set of simple restricted roots in~$\Sigma^+$. We call
\[\mathfrak{a}^{++} := \setsuch{X \in \mathfrak{a}}{\forall \alpha \in \Sigma^+,\; \alpha(X) > 0}\]
the (open) dominant Weyl chamber of $\mathfrak{a}$ corresponding to~$\Sigma^+$, and
\[\mathfrak{a}^{+} := \setsuch{X \in \mathfrak{a}}{\forall \alpha \in \Sigma^+,\; \alpha(X) \geq 0} = \overline{\mathfrak{a}^{++}}\]
the closed dominant Weyl chamber.
\end{itemize}
Then we call:
\begin{itemize}
\item $M$ the centralizer of $\mathfrak{a}$ in $K$, $\mathfrak{m}$ its Lie algebra.
\item $L$ the centralizer of $\mathfrak{a}$ in $G$, $\mathfrak{l}$ its Lie algebra. It is clear that $\mathfrak{l} = \mathfrak{a} \oplus \mathfrak{m}$, and well known (see \eg \cite{Kna96}, Proposition 7.82a) that $L = MA$.
\item $\mathfrak{n}^+$ (resp. $\mathfrak{n}^-$) the sum of the restricted root spaces of $\Sigma^+$ (resp. of $-\Sigma^+$), and $N^+ := \exp(\mathfrak{n}^+)$ and $N^- := \exp(\mathfrak{n}^-)$ the corresponding Lie groups.
\item $\mathfrak{p}^+ := \mathfrak{l} \oplus \mathfrak{n}^+$ and $\mathfrak{p}^- := \mathfrak{l} \oplus \mathfrak{n}^-$ the corresponding minimal parabolic subalgebras, $P^+ := LN^+$ and $P^- := LN^-$ the corresponding minimal parabolic subgroups.
\item $W := N_G(A)/Z_G(A)$ the restricted Weyl group.
\item $w_0$ the \emph{longest element} of the Weyl group, that is, the unique element such that $w_0(\Sigma^+) = \Sigma^-$. It is clearly an involution.
\end{itemize}

See Examples 2.3 and 2.4 in the author's earlier paper~\cite{Smi14} for working through these definitions in the cases $G = \PSL_n(\mathbb{R})$ and~$G = \PSO^+(n, 1)$.

Finally, if $\rho$~is a representation of~$G$ on a finite-dimensional real vector space~$V$, we call:
\begin{itemize}
\item the \emph{restricted weight space} in~$V$ corresponding to a form~$\lambda \in \mathfrak{a}^*$ the space
\[V^\lambda := \setsuch{v \in V}{\forall X \in \mathfrak{a},\; X \cdot v = \lambda(X)v};\]
\item a \emph{restricted weight} of the representation~$\rho$ any form~$\lambda \in \mathfrak{a}^*$ such that the corresponding weight space is nonzero.
\end{itemize}

\begin{remark}
The reader who is unfamiliar with the theory of noncompact semisimple real Lie groups may focus on the simpler case where $G$~is \emph{split}, \ie its Cartan subspace~$\mathfrak{a}$ is actually a Cartan subalgebra (just a maximal abelian subalgebra, without any additional hypotheses). In that case the restricted roots are just roots, the restricted weights are just weights, and the restricted Weyl group is just the usual Weyl group. Also the algebra~$\mathfrak{m}$ vanishes and $M$~is a discrete group.
\end{remark}

\subsection{Statement of main result}

Let $\rho$~be any representation of~$G$ on a finite-dimensional real vector space~$V$. Without loss of generality, we may assume that $G$~is connected and acts faithfully. We may then identify the abstract group~$G$ with the linear group~$\rho(G) \subset \GL(V)$. Let~$V_{\Aff}$~be the affine space corresponding to~$V$. The group of affine transformations of~$V_{\Aff}$ whose linear part lies in~$G$ may then be written~$G \ltimes V$ (where $V$~stands for the group of translations). Here is the main result of this paper.
\needspace{\baselineskip}
\begin{mainthm}
Suppose that $\rho$ satisfies the following conditions:
\begin{hypothenum}
\item \label{itm:main_condition} there exists a vector $v \in V$ such that:
\begin{enumerate}[label=(\alph*)]
\item \label{itm:fixed_by_l} $\forall l \in L,\; l(v) = v$, and
\item \label{itm:not_fixed_by_w0} $\tilde{w}_0(v) \neq v$, where $\tilde{w}_0$ is any representative in~$G$ of~$w_0 \in N_G(A)/Z_G(A)$;
\end{enumerate}
\end{hypothenum}
Then there exists a subgroup~$\Gamma$ in the affine group~$G \ltimes V$ whose linear part is Zariski-dense in~$G$ and that is free, nonabelian and acts properly discontinuously on the affine space corresponding to~$V$.
\end{mainthm}
(Note that the choice of the representative~$\tilde{w}_0$ in~\ref{itm:main_condition}\ref{itm:not_fixed_by_w0} does not matter, precisely because by~\ref{itm:main_condition}\ref{itm:fixed_by_l} the vector~$v$ is fixed by~$L = Z_G(A)$.)

For almost the whole duration of this paper, we will actually work with a much larger class of representations~$\rho$: the only restriction we will impose on~$\rho$ is to have zero as a restricted weight (Assumption~\ref{zero_is_a_weight}), which is anyway a necessary condition for~\ref{itm:main_condition}\ref{itm:fixed_by_l} above (and is equivalent when $G$ is split). We will not require $\rho$ to satisfy the hypotheses of the Main Theorem until the very last section, where the Main Theorem is proved.

When $G$ is split, the author, together with Bruno Le Floch, has moreover explicitly classified the irreducible representations satisfying this condition, in a separate paper~\cite{LFlSm}:
\begin{theorem}[\cite{LFlSm}]
Assume that $G$ is split. Let $\mathfrak{g} = \mathfrak{g}_1 \oplus \cdots \oplus \mathfrak{g}_n$ be its decomposition into simple factors; let $\rho$~be an irreducible representation of~$G$ with highest weight~$\lambda$. Then:
\begin{hypothenum}
\item $\rho$ satisfies condition~\ref{itm:main_condition}\ref{itm:fixed_by_l} of the Main Theorem if and only if $\lambda$ is a $\mathbb{Z}$-linear combination of roots of~$\mathfrak{g}$;
\item if this is the case, then $\rho$ also satisfies the whole condition~\ref{itm:main_condition}, except if $\lambda$ is of the form
\begin{equation}
\label{eq:w0_purely_positive_characterization}
\lambda = \sum_{j = 1}^n k_j p_{i, j} \varpi_{i, j} \quad\text{with, for all } j:
\begin{cases}
\varpi_{i, j} \text{ some fundamental weight of } \mathfrak{g}_j; \\
0 \leq k_j \leq m_{i, j}; \\
\prod_{j = 1}^n \sigma_{i, j}^{k_j} = +1,
\end{cases}
\end{equation}
where $p_{i, j} \in \mathbb{N}$, $m_{i, j} \in \{0, 1, 2, \infty\}$ and~$\sigma_{i, j} \in \{\pm 1\}$ are the constants corresponding to the weight~$\varpi_i$ of~$\mathfrak{g}_j$ in Table~1 in~\cite{LFlSm}.
\end{hypothenum} 
\end{theorem}
(The interesting part here is~(ii); the statement~(i) is trivial: see Remark~\ref{radical_characterization}.) We are currently working on extending this classification to non-split groups.

\begin{remark}
\label{reduction_to_irreducible}
It is sufficient to prove the Main Theorem in the case where $\rho$~is irreducible. Indeed, we may decompose~$\rho$ into a direct sum of irreducible representations, and then observe that:
\begin{itemize}
\item if some representation $\rho_1 \oplus \cdots \oplus \rho_k$ has a vector $(v_1, \ldots, v_k)$ that satisfies conditions \ref{itm:fixed_by_l} and~\ref{itm:not_fixed_by_w0} from the Main Theorem, then at least one of the vectors~$v_i$ must satisfy conditions~\ref{itm:fixed_by_l} and~\ref{itm:not_fixed_by_w0};
\item if $V = V_1 \oplus V_2$, and a subgroup $\Gamma \subset G \ltimes V_1$ acts properly on~$V_1$, then its image~$i(\Gamma)$ by the canonical inclusion $i: G \ltimes V_1 \to G \ltimes V$ still acts properly on~$V$.
\end{itemize}
\end{remark}

Here are a few examples. Items 1 and~2 show that all the previously-known examples do fall under the scope of this theorem; but actually, these were already covered by~\cite{Smi16}. Item~3 is the simplest previously-unknown example that our paper brings to light:
\begin{example}~
\begin{enumerate}
\item For $G = \SO^+(r+1, r)$, the standard representation (acting on~$V = \mathbb{R}^{2r+1}$) satisfies these conditions when $r$~is odd (see Remark~3.11 and Examples~4.22.1.b and~10.2.1 in~\cite{Smi16} for details). So Theorem~A from~\cite{AMS02} is a particular case of our theorem.

Note that this group is split, so we can also look at this through the prism of~\cite{LFlSm}. Its root system is of type~$B_r$, and its standard representation has highest weight~$\varpi_1$ when $r > 1$ and $2\varpi_1$ when $r = 1$ (in Bourbaki's numbering). Of course it is a $\mathbb{Z}$-linear combination of roots (it is a root itself). We then read off Table~1 in~\cite{LFlSm} that for $\mathfrak{g} = B_r$ and $i = 1$, we have: $p_1 = 1$ when $r > 1$ and $2$ when $r = 1$; $m_1 = \infty$; and $\sigma_1 = (-1)^r$. So indeed $\lambda = \varpi_1$ is of the form~\eqref{eq:w0_purely_positive_characterization} when $r$~is even, but is not when $r$~is odd.
\item If the real semisimple Lie group~$G$ is noncompact, the adjoint representation satisfies these conditions (see Remark~3.11 and Examples~4.22.3 and~10.2.2 in~\cite{Smi16} for details). So the main theorem of~\cite{Smi14} is a particular case of our theorem.
\item Take $G = \SL_3(\mathbb{R})$ acting on~$V = S^3 \mathbb{R}^3$ (see Example~3.9 in~\cite{Smi16} for details). The group is then split, so that $\mathfrak{l} = \mathfrak{a}$ and the set of $L$-invariant vectors is precisely the zero (restricted) weight space~$V^0$. The representation~$V$ has dimension~$10$; the zero weight space has dimension~$1$: it is spanned by the vector~$e_1 \cdot e_2 \cdot e_3$, where $(e_1, e_2, e_3)$ is the canonical basis of~$\mathbb{R}^3$. A representative~$\tilde{w_0} \in \SL_3(\mathbb{R})$ of~$w_0$ is given by
\[\tilde{w_0} = \begin{pmatrix}
0 & 0 & -1 \\
0 & 1 & 0 \\
1 & 0 & 0 \end{pmatrix}.\]
Clearly this element acts nontrivially on the vector~$e_1 \cdot e_2 \cdot e_3$.
\end{enumerate}
\end{example}
\begin{remark}
When $G$~is compact, no representation can satisfy these conditions: indeed in that case $L$ is the whole group~$G$ and condition~\ref{itm:main_condition}\ref{itm:fixed_by_l} fails. So for us, only noncompact groups are interesting. This is not surprising: indeed, any compact group acting on a vector space preserves a positive-definite quadratic form, and so falls under the scope of Bieberbach's theorem.
\end{remark}

\subsection{Strategy of the proof}
A central pillar of this paper consists in the following ``proposition template'':
\begin{schema}
\label{proposition_template}
Let $g$ and~$h$ be two elements of ``the appropriate group'', such that:
\begin{itemize}
\item both $g$ and~$h$ are ``regular'';
\item the pair $(g, h)$ has a ``non-degenerate'' geometry;
\item both $g$ and~$h$ have sufficient ``contraction strength''.
\end{itemize}
\needspace{\baselineskip}
In that case:
\begin{hypothenum}
\item \label{itm:template_regular} the product~$gh$ is still ``regular'';
\item \label{itm:template_attracting} its ``attracting geometry'' is close to that of~$g$;
\item \label{itm:template_repelling} its ``repelling geometry'' is close to that of~$h$; 
\item \label{itm:template_contraction} its ``contraction strength'' is close to the product of those of $g$ and~$h$;
\item \label{itm:template_dynamics} its ``asymptotic dynamics'' is close to the sum of those of $g$ and~$h$.
\end{hypothenum}
\end{schema}
We prove three different versions of this statement (with some slight variations, especially concerning asymptotic dynamics), all of which involve a different set of definitions for the concepts in scare quotes (see Table~\ref{tab:notions}):

\begin{table}
\begin{threeparttable}
\bgroup
\def\arraystretch{1.3}
\newlength{\firstcolumn}
\settowidth{\firstcolumn}{Group where $g$ lives}
\begin{tabularx}{\textwidth}{p{\firstcolumn}|>{\setlength\hsize{0.93\hsize}}CCC}
 & \rule[-\baselineskip]{0pt}{2\baselineskip}{\scshape Proximal case} & \rule[-\baselineskip]{0pt}{2\baselineskip}{\scshape Linear case} & \rule[-\baselineskip]{0pt}{2\baselineskip}{\scshape Affine case} \\ \hline
Group where $g$~lives
  & \rule{0pt}{1.2\baselineskip} $\bm{\GL(E)}$ (for~some~real vector~space~$E$) \rule[-0.5\baselineskip]{0pt}{\baselineskip}
  & $\bm{G}$
  & $\bm{G \ltimes V}$\\ \hline
Parameter
  & (none)
  & some subset $\bm{\Pi_X}$ of simple restr. roots (sec.~\ref{sec:parabolics})
  & \rule{0pt}{1.2\baselineskip} some extreme, generically symmetric $\bm{X_0} \in \mathfrak{a}^+$ (sec.~\ref{sec:choice})\rule[-0.5\baselineskip]{0pt}{\baselineskip}\\ \hline
``Regularity''
  & \rule{0pt}{1.2\baselineskip}{\bfseries Proximality} (Def.~\ref{proximal_definition})\rule[-0.5\baselineskip]{0pt}{\baselineskip}
  & \rule{0pt}{1.2\baselineskip}{\bfseries $\bm{X}$-regularity} (Def.~\ref{X-regular_definition})\rule[-0.5\baselineskip]{0pt}{\baselineskip}
  & \rule{0pt}{1.2\baselineskip}$\bm{\rho}$-{\bfseries regularity} (Def.~\ref{strong_regularity_definition})\rule[-0.5\baselineskip]{0pt}{\baselineskip}\\ \hline
\multicolumn{1}{r|}{attracting}
  & $\bm{E^s_g} \in \mathbb{P}(E)$
  & $\bm{y^{X, +}_g} \in G/P^+_X$
  & \rule{0pt}{1.2\baselineskip}$\bm{A^\subgap_g}$, essentially\tnotex{tnote:essentially}~$\in G \ltimes V / P^+_{X_0} \ltimes V^\subge_0$ \\
``Geometry'':
  & (Def.~\ref{attracting_repelling_space_definition})
  & (Def.~\ref{geometry_definition})
  & (Def.~\ref{ideal_dynamical_spaces_definition})\\
\multicolumn{1}{r|}{repelling}
  & $\bm{E^u_g} \in \mathbb{P}(E^*)$
  & $\bm{y^{X, -}_g} \in G/P^-_X$
  & $\bm{A^\sublap_g}$, essentially\tnotex{tnote:essentially}~$\in G \ltimes V / P^-_{X_0} \ltimes V^\suble_0$\rule[-0.5\baselineskip]{0pt}{\baselineskip} \\ \hline
``Non-degeneracy''
  & {\bfseries Def.~\ref{proximal_C_non_deg}}\rule[-0.5\baselineskip]{0pt}{1.7\baselineskip}
  & {\bfseries Def.~\ref{linear-C-non-deg-definition}}\rule[-0.5\baselineskip]{0pt}{1.7\baselineskip}
  & {\bfseries Def.~\ref{regular_definition}}\rule[-0.5\baselineskip]{0pt}{1.7\baselineskip} \\ \hline
``Contr. strength''
  & $\bm{\tilde{s}(g)}$ (Def.~\ref{s_tilde_definition})\rule[-0.5\baselineskip]{0pt}{1.7\baselineskip}
  & $\bm{\vec{s}_X(g)}$ (Def.~\ref{vec_s_definition})\rule[-0.5\baselineskip]{0pt}{1.7\baselineskip}
  & $\bm{s_{X_0}(g)}$ (Def.~\ref{s_definition})\rule[-0.5\baselineskip]{0pt}{1.7\baselineskip}\\ \hline
``Asymp. dynamics''
  & possibly\tnotex{tnote:possibly} the $\log$ of the spectral radius~$r(g)$ (Def.~\ref{proximal_definition})
  & roughly\tnotex{tnote:roughly} the {\bfseries Jordan projection} $\bm{\mathrm{Jd}(g)}$ (Def.~\ref{Jd_Ct_definition}) 
  & \rule{0pt}{1.2\baselineskip}roughly\tnotex{tnote:roughly} the {\bfseries Jordan projection} $\bm{\mathrm{Jd}(g)}$ together with the {\bfseries Margulis invariant} $\bm{M(g)}$ (Def.~\ref{margulis_invariant})
\end{tabularx}
\vspace{5mm}
\begin{tablenotes}
  \item\label{tnote:essentially} See the last point of Remark~\ref{affine_flag_varieties}.
  \item\label{tnote:possibly} In the proximal case, there are at least two reasonable definitions of asymptotic dynamics. Another one would be the logarithm of the \emph{spectral gap}~$\frac{|\lambda_1|}{|\lambda_2|}$, defined in the same place.
  \item\label{tnote:roughly} Technically, the statement of Schema~\ref{proposition_template} only holds for a further projection of~$\jordan(g)$, given by the subset of coordinates~$\alpha_i(\jordan(g))$ for~$i \not\in \Pi_X$; for the remaining coordinates we only get an inequality. However Proposition~\ref{regular_product_qualitative} allows us to circumvent this problem. Also if $\Pi_{X_0} = \emptyset$, which is often the case (see Example~\ref{simplifying_assumptions_example}.\ref{itm:abundant_classification} and Remark~\ref{X0_examples}.\ref{itm:abundant_implications}), this issue does not arise at all.
\end{tablenotes}
\egroup
\end{threeparttable}
\caption{Possible meanings of notions appearing in Schema~\ref{proposition_template}, with references to the corresponding definitions.}
\label{tab:notions}
\end{table}

\begin{itemize}
\item the ``proximal'' version is Proposition~\ref{proximal_product};
\item the ``linear'' version is Proposition~\ref{intrinsic_regular_product} (for the ``main part'', \ie points \ref{itm:template_regular} to~\ref{itm:template_contraction} in Schema~\ref{proposition_template}) and Proposition~\ref{jordan_additivity} (for the asymptotic dynamics);
\item the ``affine'' version is Proposition~\ref{regular_product} (ii) and~(iii) (for the main part) and Proposition~\ref{invariant_additivity_only} (for the asymptotic dynamics).
\end{itemize}
For the last two versions, the definitions in question also depend on a parameter fixed once and for all (given in the second line of Table~\ref{tab:notions}). To give a first intuition, very roughly, the ``geometry'' of an element has to do with its eigenvectors, its ``contraction strength'' has to do with its singular values, and its ``asymptotic dynamics'' has to do with the moduli of its eigenvalues. (In the ``affine'' case, though, the asymptotic dynamics additionally depend on the translation part of the element.) The ``asymptotic dynamics'' terminology is explained by Gelfand's formula~\eqref{eq:Gelfand}.

Here is the relationship between these three results:
\begin{itemize}
\item The ``proximal'' version is used as a stepping stone to prove the other two versions. This result is not new: very similar lemmas were proved in \cite{Ben96}, \cite{Ben97} and~\cite{AMS02}, and this exact statement was proved in the author's previous papers~\cite{Smi14, Smi16}. We briefly recall it in Section~\ref{sec:proximal_maps_bis}.
\item The ``linear'' version is also used as a stepping stone to prove the affine version, but not in a completely straightforward way. In fact, the main part of the linear version, after being reformulated as Proposition~\ref{regular_product}~(i), is only involved in the proof of the additivity of the Margulis invariant (\ie the affine version for the asymptotic dynamics). On the other hand, the linear version for the asymptotic dynamics is already necessary to prove the main part of the affine version. It is a slight generalization of a result by Benoist~\cite{Ben96}. We cover it in Section~\ref{sec:linear_regular_maps}.
\item The ``affine'' version is the key result of this paper. Just defining the affine concepts keeps us busy for a long time (Sections~\ref{sec:choice}--\ref{sec:rho_regular}), and proving the results also takes a fair amount of work (Sections \ref{sec:quantitative_properties_of_products} and~\ref{sec:additivity}). Once we have it, it takes only five pages to prove our Main Theorem.
\end{itemize}

The proof has a lot in common with the author's previous papers~\cite{Smi14, Smi16}, and ultimately builds upon an idea that was introduced in Margulis's seminal paper~\cite{Mar83}. Let us now present a few highlights of what distinguishes this paper from the previous works.

\begin{itemize}
\item The main difference lies in the treatment of the dynamical spaces. As long as we only worked in ``non-swinging'' representations (see Definition~\ref{simpifying_assumptions_def}.\ref{itm:non-swinging_def}), we could associate to every element of~$G \ltimes V$ that is ``regular'' (in the appropriate sense) a decomposition of~$V$ into three so-called dynamical subspaces
\[V = V^\subg_g \oplus V^\sube_g \oplus V^\subl_g,\]
all stable by~$g$ and on which $g$ acts with eigenvalues respectively of modulus~$> 1$, of modulus~$1$ and of modulus~$< 1$. In the general case, this is no longer possible: we need to enlarge the neutral subspace~$V^\sube_g$ to an ``approximately neutral'' subspace~$V^\subap_g$. The eigenvalues of~$g$ on that subspace are, in some rather weak sense, ``not too far'' from~$1$, but will still grow exponentially (in the group we will eventually construct). The decomposition now becomes
\[V = V^\subgg_g \oplus V^\subap_g \oplus V^\subll_g.\]

This forces us to completely change our point of view. In fact, we now define these ``approximate dynamical spaces'' in a purely algebraic way, by focusing on the dynamics that~$g$ \emph{would} have on them, \emph{if} it were conjugate to (the exponential of) some reference element~$X_0$ of the Weyl chamber~$\mathfrak{a}^+$. For this reason, we actually call them the \emph{ideal dynamical spaces} (see Definition~\ref{ideal_dynamical_spaces_definition}). Only by imposing an additional condition on~$g$ (``asymptotic contraction'', see Definition~\ref{strong_regularity_definition}) can we ensure that the ``ideal'' dynamical spaces become indeed the ``approximate'' dynamical spaces.

\item The linear version of Schema~\ref{proposition_template} never explicitly appeared in the author's previous papers: so far, we have been able to simply present the linear theory as a particular case of the affine theory. (Even Proposition~6.11 in~\cite{Smi16}, which seems almost identical to Proposition~\ref{jordan_additivity} in the current paper, relies in fact on the affine versions of the properties). However this becomes untenable in the case of swinging representations, as the relationship between affine contraction strength and linear contraction strength becomes less straightforward (see Section~\ref{sec:affine_and_linear}). This led us to develop the linear theory on its own; we think that Propositions \ref{intrinsic_regular_product} and~\ref{jordan_additivity} might have some interest independently of the remainder of the paper.

We must however point out that these results are not completely new. The particular case where $\Pi_X = \emptyset$ is due to Benoist~\cite{Ben96, Ben97}; the case where $\Pi_X$~is arbitrary is an easy generalization which also relies on the tools developed by Benoist, and is well-known to experts in the field. It might seem at first sight that the general case is actually mentioned explicitly in~\cite{Ben97}, but this is not so: see Remark~6.16 in~\cite{Smi16}.

\item The central argument of the paper, namely the proof of Proposition~\ref{invariant_additivity_only}, has been completely overhauled. Though still highly technical, it is now much less messy: it is more symmetric, and the organization of the proof better reflects the separation of the ideas it involves. Even if we forget the fact that it works in a more general setting, it is definitely an improvement compared to the proof given in~\cite{Smi16}.
\end{itemize}

\subsection{Plan of the paper}
In Section~\ref{sec:algebraic_prelim}, we give a few algebraic results and introduce some notations related to metric properties and estimates. Most of these results are well-known.

In Section~\ref{sec:proximal_maps_bis}, we recall the definitions of the proximal versions of the key properties and the proximal version of Schema~\ref{proposition_template}. This is also well-known.

In Section~\ref{sec:linear_regular_maps}, we define the linear versions of the key properties, then prove the linear version of Schema~\ref{proposition_template}. This section clarifies and expands the results mentioned in Section~6 of~\cite{Smi16}. Ultimately, most of the definitions and ideas here are due to Benoist~\cite{Ben96, Ben97}.

In Section~\ref{sec:choice}, we choose an element $X_0 \in \mathfrak{a}^+$ that will be used to define the affine versions of the key properties. This generalizes Section~3 in~\cite{Smi16}.

In Section~\ref{sec:X_0_constructions}, we present some preliminary constructions involving~$X_0$, necessary to pave the way for the next section's definitions. We also introduce some (elementary) formalism that expresses affine spaces in terms of vector spaces. Part of the material is borrowed from Sections 4.1 and~4.2 in~\cite{Smi16}.

In Section~\ref{sec:rho_regular}, we define the affine versions of the key properties. This generalizes Sections 4.3--5.2 in~\cite{Smi16}, but introduces several new ideas.

In Section~\ref{sec:quantitative_properties_of_products}, we prove the main part of the affine version of Schema~\ref{proposition_template}. This generalizes Section~7 in~\cite{Smi16}.

Section~\ref{sec:additivity} contains the key part of the proof. We prove the asymptotic dynamics part of the affine version of Schema~\ref{proposition_template} --- in other terms, approximate additivity of the Margulis invariants. This generalizes Section~8 in~\cite{Smi16} and Section~4 in~\cite{Smi14}; the proof is based on the same idea, but has been considerably rewritten.

Section~\ref{sec:induction} uses induction to extend the results of the two previous sections to products of an arbitrary number of elements. It is a straightforward generalization of Section~9 in~\cite{Smi16} and Section~5 in~\cite{Smi14}

Section~\ref{sec:construction} contains the proof of the Main Theorem. It is an almost straightforward generalization of Section~10 in~\cite{Smi16} and Section~6 in~\cite{Smi14}.

\subsection{Acknowledgments}
I am very grateful to my Ph.D. advisor, Yves Benoist, who introduced me to this exciting and fruitful subject, and gave invaluable help and guidance in my initial work on this project.

I would also like to thank Bruno Le Floch for some interesting discussions, which in particular helped me gain more insight about weights of representations.

Thanks to the Yale University (and my colleagues from there), where almost all of the work on this paper has been conducted, for the wonderful working conditions.

\section{Preliminaries}
\label{sec:algebraic_prelim}
In Subsection~\ref{sec:eigenvalues_in_different_representations}, for any element $g \in G$, we give a formula for the eigenvalues and singular values of the linear maps~$\rho(g)$ where $\rho$~is an arbitrary representation of~$G$. This is nothing more than a reminder of Subsection~2.1 in~\cite{Smi16}.

In Subsection~\ref{sec:restricted_weights}, we present some properties of restricted weights of a real finite-dimensional representation of a real semisimple Lie group. This is mostly a reminder of Subsection~2.2 in~\cite{Smi16}.

In Subsection~\ref{sec:parabolics}, we give some basic results from the theory of parabolic subgroups and subalgebras (not necessarily minimal). This is a development on Subsection~3.4 in~\cite{Smi16}.

In Subsection~\ref{sec:metric_conventions}, we give some notation conventions related to metric properties and estimates (mostly borrowed from the author's earlier papers).

\subsection{Eigenvalues in different representations}
\label{sec:eigenvalues_in_different_representations}

In this subsection, we express the eigenvalues and singular values of a given element $g \in G$ acting in a given representation~$\rho$, exclusively in terms of the structure of~$g$ in the abstract group~$G$ (respectively its Jordan decomposition and its Cartan decomposition). This is just a reminder of Subsection~2.1 in~\cite{Smi16}; the result itself was certainly well-known even before that.

\begin{proposition}[Jordan decomposition]
\label{jordan_decomposition}
Let $g \in G$. There exists a unique decomposition of~$g$ as a product $g = g_h g_e g_u$, where:
\begin{itemize}
\item $g_h$~is conjugate in~$G$ to an element of~$A$ (\emph{hyperbolic});
\item $g_e$~is conjugate in~$G$ to an element of~$K$ (\emph{elliptic});
\item $g_u$~is conjugate in~$G$ to an element of~$N^+$ (\emph{unipotent});
\item these three maps commute with each other.
\end{itemize}
\end{proposition}
\begin{proof}
This is well-known, and given for example in \cite{Ebe96}, Theorem~2.19.24. Note however that the latter theorem uses representation-dependent definitions of a hyperbolic, elliptic or unipotent element (applied to the case of the adjoint representation). To state the theorem with the representation-agnostic definitions that we used, we need to apply Proposition~2.19.18 and Theorem~2.19.16 from the same book.
\end{proof}

\begin{proposition}[Cartan decomposition]
Let $g \in G$. Then there exists a decomposition of~$g$ as a product~$g = k_1 a k_2$, with~$k_1, k_2 \in K$ and~$a = \exp(X)$ with~$X \in \mathfrak{a}^{+}$. Moreover, the element~$X$ is uniquely determined by~$g$.
\end{proposition}
\begin{proof}
This is a classical result; see \eg Theorem~7.39 in~\cite{Kna96}.
\end{proof}

\begin{definition}
\label{Jd_Ct_definition}
For every element~$g \in G$, we define:
\begin{itemize}
\item the \emph{Jordan projection} of~$g$ (sometimes also known as the Lyapunov projection), written~$\jordan(g)$, to be the unique element of the closed dominant Weyl chamber~$\mathfrak{a}^{+}$ such that the hyperbolic part $g_h$ (from the Jordan decomposition $g = g_h g_e g_u$ given above) is conjugate to~$\exp(\jordan(g))$;
\item the \emph{Cartan projection} of~$g$, written~$\cartan(g)$, to be the element~$X$ from the Cartan decomposition given above.
\end{itemize}
\end{definition}

To talk about singular values, we need to introduce a Euclidean structure. We are going to use a special one. 
\begin{lemma}
\label{K-invariant}
Let $\rho_*$ be some real representation of~$G$ on some space~$V_*$. There exists a $K$-invariant positive-definite quadratic form~$B_*$ on~$V_*$ such that all the restricted weight spaces are pairwise $B_*$-orthogonal.
\end{lemma}
We want to reserve the plain notation~$\rho$ for the ``default'' representation, to be fixed once and for all at the beginning of Section~\ref{sec:choice}. We use the notation~$\rho_*$ so as to encompass both this representation~$\rho$ and the representations~$\rho_i$ defined in Proposition~\ref{fundamental_real_representation}.

\begin{proof}
See Lemma~5.33.a) in~\cite{BenQui}.
\end{proof}

\begin{example}
If $\rho_* = \Ad$~is the adjoint representation, then $B_*$~is the form~$B_\theta$ given by
\[\forall X, Y \in \mathfrak{g},\quad B_{\theta}(X, Y) = -B(X, \theta Y)\]
(see~(6.13) in~\cite{Kna96}), where $B$~is the Killing form and $\theta$~is the Cartan involution.
\end{example}

Recall that the \emph{singular values} of a map~$g$ in a Euclidean space are defined as the square roots of the eigenvalues of $g^* g$, where $g^*$ is the adjoint map. The largest and smallest singular values of~$g$ then give respectively the operator norm of~$g$ and the reciprocal of the operator norm of~$g^{-1}$.

\begin{proposition}
\label{eigenvalues_and_singular_values_characterization}
Let $\rho_*: G \to \GL(V_*)$ be any representation of~$G$ on some vector space~$V_*$; let $\lambda_*^1, \ldots, \lambda_*^{d_*}$ be the list of all the restricted weights of~$\rho_*$, repeated according to their multiplicities. Let $g \in G$; then:
\begin{hypothenum}
\item The list of the moduli of the eigenvalues of~$\rho_*(g)$ is given by
\[\left( e^{\lambda_*^i(\jordan(g))} \right)_{1 \leq i \leq d_*}.\]
\item The list of the singular values of~$\rho_*(g)$, with respect to a $K$-invariant Euclidean norm~$B_*$ on~$V_*$ that makes the restricted weight spaces of~$V_*$ pairwise orthogonal (such a norm exists by Lemma~\ref{K-invariant}), is given by
\[\left( e^{\lambda_*^i(\cartan(g))} \right)_{1 \leq i \leq d_*}.\]
\end{hypothenum}
\end{proposition}

\begin{proof}
This is also completely straightforward. See Proposition~2.6 in~\cite{Smi16}.
\end{proof}

\subsection{Properties of restricted weights}
\label{sec:restricted_weights}

In this subsection, we introduce a few properties of restricted weights of real finite-dimensional representations. (Proposition~\ref{convex_hull_and_inequalities} is actually a general result about Coxeter groups.) The corresponding theory for ordinary weights is well-known: see for example Chapter~V in~\cite{Kna96}. This is mostly a reminder of Subsection~2.2 from~\cite{Smi16}; the only addition is Lemma~\ref{roots_contained_in_weights}.

Let $\alpha_1, \ldots, \alpha_r$ be an enumeration of the set~$\Pi$ of simple restricted roots generating~$\Sigma^+$. For every~$i$, we set
\begin{equation}
\alpha'_i := \begin{cases}
2\alpha_i &\text{ if $2\alpha_i$ is a restricted root} \\
\alpha_i &\text{ otherwise.}
\end{cases}
\end{equation}
For every index~$i$ such that $1 \leq i \leq r$, we define the $i$-th \emph{fundamental restricted weight}~$\varpi_i$ by the relationship
\begin{equation}
2\frac{\langle \varpi_i, \alpha'_j \rangle}{\| \alpha'_j \|^2} = \delta_{ij}
\end{equation}
for every $j$ such that $1 \leq j \leq r$.

By abuse of notation, we will often allow ourselves to write things such as ``for all~$i$ in some subset $\Pi' \subset \Pi$, $\varpi_i$ satisfies...'' (tacitly identifying the set~$\Pi'$ with the set of indices of the simple restricted roots that are inside).

In the following proposition, for any subset $\Pi' \subset \Pi$, we denote:
\begin{itemize}
\item by~$W_{\Pi'}$ the Weyl subgroup of type~$\Pi'$:
\begin{equation}
W_{\Pi'} := \langle s_\alpha \rangle_{\alpha \in \Pi'};
\end{equation}
\item by~$\mathfrak{a}^{+}_{\Pi'}$ the fundamental domain for the action of~$W_{\Pi'}$ on~$\mathfrak{a}$:
\begin{equation}
\mathfrak{a}^{+}_{\Pi'} := \setsuch{X \in \mathfrak{a}}
{\forall \alpha \in \Pi',\; \alpha(X) \geq 0},
\end{equation}
which is a kind of prism whose base is the dominant Weyl chamber of~$W_{\Pi'}$.
\end{itemize}
\begin{proposition}
\label{convex_hull_and_inequalities}
Take any $\Pi' \subset \Pi$, and let $X$, $Y$ be any two vectors in $\mathfrak{a}^{+}_{\Pi'}$. Then the system of linear inequalities
\[\begin{cases}
\forall i \in \mathrlap{\Pi',}\qquad\qquad \varpi_i(Y) \leq \varpi_i(X) \\
\forall i \in \mathrlap{\Pi \setminus \Pi',}\qquad\qquad \varpi_i(Y) = \varpi_i(X)
\end{cases}\]
is satisfied if and only if the vector~$Y$ lies in the convex hull of the orbit of~$X$ by~$W_{\Pi'}$.
\end{proposition}
\begin{proof}
For the particular case~$\Pi' = \Pi$, this is well known: see \eg \cite{Hall15}, Proposition~8.44.

The general case can easily be reduced to this particular case: see \cite{Smi16}, Proposition~2.7.
\end{proof}

\begin{proposition}
\label{restr_weight_lattice}
Every restricted weight of every representation of $\mathfrak{g}$ is a linear combination of fundamental restricted weights with integer coefficients.
\end{proposition}
\begin{proof}
This is a particular case of Proposition~5.8 in~\cite{BT65}. For a correction concerning the proof, see also Remark~5.2 in~\cite{BT72compl}.
\end{proof}

\begin{proposition}
\label{highest_restr_weight}
If $\rho_*$ is an irreducible representation of $\mathfrak{g}$, there is a unique restricted weight~$\lambda_*$ of~$\rho_*$, called its \emph{highest restricted weight}, such that no element of the form $\lambda_* + \alpha_i$ with $1 \leq i \leq r$ is a restricted weight of~$\rho_*$.
\end{proposition}
\begin{remark}
In contrast to the situation with non-restricted weights, the highest restricted weight is not always of multiplicity~$1$; nor is a representation uniquely determined by its highest restricted weight.
\end{remark}
\begin{proof}
This easily follows from the existence and uniqueness of the ordinary (non-restricted) highest weight, given for example in \cite{Kna96}, Theorem~5.5~(d).
\end{proof}

\begin{proposition}
\label{convex_hull}
Let $\rho_*$ be an irreducible representation of $\mathfrak{g}$; let $\lambda_*$ be its highest restricted weight. Let $\Lambda_{\lambda_*}$ be the restricted root lattice shifted by~$\lambda_*$:
\[\Lambda_{\lambda_*} := \setsuch{\lambda_* + c_1 \alpha_1 + \cdots + c_r \alpha_r}{c_1, \ldots, c_r \in \mathbb{Z}}.\]
Then the set of restricted weights of~$\rho_*$ is exactly the intersection of the lattice $\Lambda_{\lambda_*}$ with the convex hull of the orbit $\setsuch{w(\lambda_*)}{w \in W}$ of~$\lambda_*$ by the restricted Weyl group.
\end{proposition}
\begin{proof}
Once again, this follows from the corresponding result for non restricted weights (see \eg \cite{Hall15}, Theorem~10.1) by passing to the restriction. In the case of restricted weights, one of the inclusions is stated in~\cite{Hel08}, Proposition~4.22.
\end{proof}


\begin{proposition}
\label{fundamental_real_representation}
For every index~$i$ such that $1 \leq i \leq r$, there exists an irreducible representation~$\rho_i$ of~$G$ on a space~$V_i$ whose highest restricted weight is equal to~$n_i \varpi_i$ (for some positive integer~$n_i$) and has multiplicity~$1$.
\end{proposition}
\begin{proof}
This follows from the general Theorem~7.2 in~\cite{Tits71}. This is also stated as Lemma~2.5.1 in~\cite{Ben97}. 
%
%
%
%
\end{proof}
\begin{example}~
\begin{enumerate}
\item If $G = \SL_n(\mathbb{R})$, we may take $V_i = \ext^i \mathbb{R}^n$, so that $\rho_i$~is the $i$-th exterior power of the standard representation of~$G$ on~$\mathbb{R}^n$.
\item More generally, if~$G$ is split, then all restricted weight spaces correspond to ordinary weight spaces, hence have dimension~$1$. So we may simply take every $n_i$ to be~$1$, so that the $\rho_i$'s are precisely the fundamental representations.
\end{enumerate}
\end{example}

\begin{lemma}
\label{fund_repr_other_weights}
Fix an index~$i$ such that $1 \leq i \leq r$. Then all restricted weights of~$\rho_i$ other than $n_i \varpi_i$ have the form
\[n_i \varpi_i - \alpha_i - \sum_{j=1}^r c_j \alpha_j,\]
with $c_j \geq 0$ for every $j$.
\end{lemma}
\begin{proof}
This is the last remark in section~2.5 of~\cite{Ben97}. For a proof, see Lemma~2.13 in~\cite{Smi16}.
\end{proof}
\begin{lemma}
\label{roots_contained_in_weights}
Assume that the Lie algebra~$\mathfrak{g}$ is simple, and that its restricted root system has a simply-laced diagram. Let $\rho_*$ be an irreducible representation of~$G$; let~$\Omega_*$ be the set of its restricted weights. Assume that we have
\[\{0\} \subsetneq \Omega_*\]
(\ie the representation has~$0$ as a restricted weight, and at least one nonzero restricted weight). Then we have $\Sigma \subset \Omega_*$.
\end{lemma}
Note that the only case where $0$ is the \emph{only} restricted weight is when the image of~$G$ by~$\rho$ is a compact group (for simple~$G$, this means that either $G$~itself is compact or $\rho$~is trivial.)
\begin{proof}
First let us show that $\Omega_*$ contains at least one restricted root. By Proposition~\ref{convex_hull}, every restricted weight~$\lambda$ of~$\rho_*$ may be written as a sum
\begin{equation}
\label{eq:weight_as_sum_of_roots}
\lambda = \alpha_{(1)} + \cdots + \alpha_{(l)},
\end{equation}
where $\alpha_{(1)}, \ldots, \alpha_{(l)}$ are some restricted roots. Define the \emph{level} of~$\lambda$ to be the smallest integer~$l$ for which such a decomposition is possible.

Let $\lambda$~be an element of~$\Omega_* \setminus \{0\}$ whose level~$l$ is the smallest possible (it exists since $\Omega_*$ is not reduced to~$\{0\}$), and consider its decomposition~\eqref{eq:weight_as_sum_of_roots}. Then for every~$i \neq l$, we have $\langle \alpha_{(i)}, \alpha_{(l)} \rangle \geq 0$: indeed otherwise, $\alpha_{(i)} + \alpha_{(l)}$ would be a restricted root, so we could combine them together to produce a decomposition of~$\lambda$ of length~$l-1$. Hence we have
\[\frac{2\langle \lambda, \alpha_{(l)} \rangle}{\langle \alpha_{(l)}, \alpha_{(l)} \rangle}
= \sum_{i=1}^l \frac{2\langle \alpha_{(i)}, \alpha_{(l)} \rangle}{\langle \alpha_{(l)}, \alpha_{(l)} \rangle}
\geq \frac{2\langle \alpha_{(l)}, \alpha_{(l)} \rangle}{\langle \alpha_{(l)}, \alpha_{(l)} \rangle}
= 2,\]
which implies that
\[\lambda - \alpha_{(l)} \in \Conv(\{\lambda, s_{\alpha_{(l)}} \lambda\}) \subset \Conv(W \lambda).\]
From Proposition~\ref{convex_hull}, it follows that~$\lambda - \alpha_{(l)}$ is still a restricted weight of~$\rho_*$; but its level is now~$l-1$. Since $l$~is by assumption the smallest nonzero level, necessarily we have $l = 1$, and $\lambda$ is indeed a restricted root.

Now it is well-known (see Problem~2.11 in~\cite{Kna96}) that for a simple Lie algebra, $W$~acts transitively on the set of restricted roots having the same length as~$\lambda$. Since the restricted root system of~$\mathfrak{g}$ has a simply-laced diagram, all of the restricted roots have the same length. Hence the orbit of~$\lambda$ by~$W$ is the whole set~$\Sigma$; we conclude that $\Sigma \subset \Omega_*$. \qedhere
\end{proof}

\subsection{Parabolic subgroups and subalgebras}
\label{sec:parabolics}

In this subsection we recall the well-known theory of parabolic subalgebras and subgroups. We begin by defining them, as well as the Levi subalgebra and subgroup of a given type and the corresponding subset of~$\Pi$ and subgroup of~$W$; so far we follow Subsection~3.4 in~\cite{Smi16}. But now we go further, giving a few propositions relating these different objects; in particular the generalized Bruhat decomposition (Lemma~\ref{generalized_Bruhat}).

A parabolic subgroup (or subalgebra) is usually defined in terms of a subset~$\Pi'$ of the set~$\Pi$ of simple restricted roots. We find it more convenient however to use a slightly different language. To every such subset corresponds a facet of the Weyl chamber, given by intersecting the walls corresponding to elements of~$\Pi'$. We may exemplify this facet by picking some element~$X$ in it that does not belong to any subfacet. Conversely, for every~$X \in \mathfrak{a}^+$, we define the corresponding subset
\begin{equation}
\Pi_{X} := \setsuch{\alpha \in \Pi}{\alpha(X) = 0}.
\end{equation}
The parabolic subalgebras and subgroups of type~$\Pi_{X}$ can then be very conveniently rewritten in terms of~$X$, as follows.
\begin{definition}
For every $X \in \mathfrak{a}^+$, we define:
\begin{itemize}
\item $\mathfrak{p}_{X}^+$ and $\mathfrak{p}_{X}^-$ the parabolic subalgebras of type~$X$, and~$\mathfrak{l}_{X}$ the Levi subalgebra of type~$X$:
\[\mathfrak{p}_{X}^+ := \mathfrak{l} \oplus \bigoplus_{\alpha(X) \geq 0} \mathfrak{g}^\alpha;\]
\[\mathfrak{p}_{X}^- := \mathfrak{l} \oplus \bigoplus_{\alpha(X) \leq 0} \mathfrak{g}^\alpha;\]
\[\mathfrak{l}_{X} := \mathfrak{l} \oplus \bigoplus_{\alpha(X) = 0} \mathfrak{g}^\alpha.\]

\item $P_{X}^+$ and $P_{X}^-$ the corresponding parabolic subgroups, and~$L_{X}$ the Levi subgroup of type~$X$:
\[P_{X}^+ := N_G(\mathfrak{p}_{X}^+);\]
\[P_{X}^- := N_G(\mathfrak{p}_{X}^-);\]
\[L_{X} := P_{X}^+ \cap P_{X}^-.\]
\end{itemize}
\end{definition}

The following statement is well-known:
\begin{proposition}
\label{centralizer_of_x}
We have $L_X = Z_G(X)$.
\end{proposition}
\begin{proof}
First note that by combining Propositions 7.83~(b), (e) and 7.82~(a) in~\cite{Kna96}, we get that $L_X = Z_G(\mathfrak{a}_X)$, where
\[\mathfrak{a}_X := \setsuch{Y \in \mathfrak{a}}{\forall \alpha \in \Pi_X,\; \alpha(Y) = 0}\]
is the intersection of all walls of the Weyl chamber containing~$X$. It remains to show that
\[Z_G(\mathfrak{a}_X) = Z_G(X).\]

Clearly the Lie algebra of both groups is equal to~$\mathfrak{l}_X$; hence their identity components $Z_G(\mathfrak{a}_X)_e$ and~$Z_G(X)_e$ are also equal. But by combining Propositions 7.25 and~7.33 in~\cite{Kna96}, it follows that $Z_G(X) = M Z_G(X)_e$ and similarly for~$Z_G(\mathfrak{a}_X)$. The conclusion follows.
\end{proof}

An object closely related to these parabolic subgroups (see Corollary~\ref{parabolic_Bruhat_decomposition}, the Bruhat decomposition for parabolic subgroups) is the stabilizer of~$X$ in the Weyl group:
\begin{definition}
For any $X \in \mathfrak{a}^{+}$, we set
\[W_{X} := \setsuch{w \in W}{wX = X}.\]
\end{definition}
\begin{remark}
\label{W_vs_Pi}
The group~$W_X$ is also closely related to the set~$\Pi_X$. Indeed, it follows immediately that a simple restricted root~$\alpha$ belongs to~$\Pi_X$ if and only if the corresponding reflection~$s_\alpha$ belongs to~$W_X$. Conversely, it is well-known (Chevalley's lemma, see \eg \cite{Kna96}, Proposition~2.72) that these reflections actually generate the group~$W_X$.

Thus $W_{X}$~is actually the same thing as~$W_{\Pi'}$ (as defined before Proposition~\ref{convex_hull_and_inequalities}) where we substitute~$\Pi' = \Pi_{X}$.
\end{remark}
\begin{example}
To help understand the conventions we are taking, here are the extreme cases:
\begin{enumerate}
\item If $X$ lies in the open Weyl chamber~$\mathfrak{a}^{++}$, then:
\begin{itemize}
\item $P^+_X = P^+$ is the minimal parabolic subgroup; $P^-_X = P^-$; $L_X = L$;
\item $\Pi_X = \emptyset$;
\item $W_X = \{\Id\}$.
\end{itemize}
\item If $X=0$, then:
\begin{itemize}
\item $P^+_X = P^-_X = L_X = G$;
\item $\Pi_X = \Pi$;
\item $W_X = W$.
\end{itemize}
\end{enumerate}
\end{example}

We will finish this subsection by giving a result (Corollary~\ref{parabolic_Bruhat_decomposition}) that links $W_X$ with~$P^\pm_X$. But the more general Proposition~\ref{generalized_Bruhat} will be useful to us in its own right.
\begin{definition}
\label{top_bottom_subset}
Let $\Omega$ be a set of elements of~$\mathfrak{a}^*$. We say that $\Xi$ is:
\begin{itemize}
\item a \emph{top-subset} of~$\Omega$ if $\phantom{\text{bottom}}\quad   \forall \alpha \in \Sigma^+,\quad (\Xi + \alpha) \cap \Omega \;\subset\; \Xi$;
\item a \emph{bottom-subset} of~$\Omega$ if $\phantom{\text{top}}\quad\; \forall \alpha \in \Sigma^+,\quad (\Xi - \alpha) \cap \Omega \;\subset\; \Xi$.
\end{itemize} 
\end{definition}
\begin{example}
In the set of restricted roots~$\Sigma$, the positive restricted roots form a top-subset and the negative restricted roots form a bottom-subset.
\end{example}
\begin{lemma}[Generalized Bruhat decomposition]
\label{generalized_Bruhat}
Let $\rho_*$ be some real representation of~$G$ on some space~$V_*$; let $\Omega_*$ be the set of restricted weights of~$\rho_*$. For every subset $\Xi \subset \Omega_*$, we set
\begin{equation}
V^\Xi_* := \bigoplus_{\lambda \in \Xi} V^\lambda_*.
\end{equation}
Then we have:
\[\begin{cases}
\Stab_G \left( V^\Xi_* \right) = P^+ \Stab_W(\Xi) P^+ &\text{ if $\Xi$ is a top-subset of~$\Omega_*$;} \\
\Stab_G \left( V^\Xi_* \right) = P^- \Stab_W(\Xi) P^- &\text{ if $\Xi$ is a bottom-subset of~$\Omega_*$.}
\end{cases}\]
\end{lemma}
\begin{proof}
Assume that $\Xi$ is a top-subset; the case of a bottom-subset is analogous.

The first step is to show that $\Stab_G(V^\Xi_*)$ contains~$P^+$. Indeed:
\begin{itemize}
\item The group $L$ stabilizes every restricted weight space $V^\lambda_*$. Indeed, take some $\lambda \in \mathfrak{a}^*$, $v \in V^\lambda_*$, $l \in L$, $X \in \mathfrak{a}$; then we have:
\begin{equation}
\label{eq:L_stabilizes_weight_spaces}
X \cdot l(v) = l(\Ad (l^{-1}) (X) \cdot v) = l(X \cdot v) = \lambda(X)l(v).
\end{equation}
\item For every $\alpha \in \Sigma^+$ and every $\lambda \in \Xi$, clearly we have
\[\mathfrak{g}^\alpha \cdot V^\lambda_* \subset V^{\lambda+\alpha}_* \subset V^\Xi_*\]
by definition of a top-subset. Hence $\mathfrak{n}^+$ stabilizes $V^\Xi_*$.
\item The statement follows as $P^+ = L\exp(\mathfrak{n}^+)$.
\end{itemize}
Now take any element $g \in G$. Let us apply the Bruhat decomposition (see \eg \cite{Kna96}, Theorem 7.40): there exists an element~$w$ of the restricted Weyl group $W$ such that we may write
\[g = p_1\tilde{w}p_2\]
where $p_1, p_2$ are some elements of the minimal parabolic subgroup $P^+$ and $\tilde{w} \in N_G(A)$ is some representative of~$w \in W = N_G(A)/Z_G(A)$. From the statement that we just proved, it follows that $g$~stabilizes $V^\Xi_*$ if and only if $\tilde{w}$~does so. On the other hand, it is clear that for every $\lambda \in \mathfrak{a}^*$, we have
\begin{equation}
\tilde{w} V^\lambda = V^{w \lambda}
\end{equation}
(the choice of the representative~$\tilde{w}$ does not matter since as seen above~\eqref{eq:L_stabilizes_weight_spaces}, the kernel~$Z_G(A) = L$ stabilizes~$V^\lambda$). The conclusion follows.
\end{proof}
The following particular case is well-known (see for example \cite{Hum75}, Theorem~29.3):
\begin{corollary}[Bruhat decomposition for parabolic groups]
\label{parabolic_Bruhat_decomposition}
We have the identities $P_X^+ = P^+ W_X P^+$ and~$P_X^- = P^- W_X P^-$.
\end{corollary}
\begin{proof}
Take $\rho_*$ to be the adjoint representation: then $V_* = \mathfrak{g}$ and~$\Omega_* = \Sigma \cup \{0\}$.
Take
\[\Xi^+ = \setsuch{\alpha \in \Sigma \cup \{0\}}{\alpha(X) \geq 0};\]
this is a top-subset. It is easy to show that we have $\Stab_W(\Xi^+) = W_X$, and by definition~$\mathfrak{g}^{\Xi^+} = \mathfrak{p}^+_X$. Applying the lemma, the first identity follows. Applying the lemma to the subset~$\Xi^-$ defined analogously, the second identity follows.
\end{proof}

\subsection{Metric properties and estimates}
\label{sec:metric_conventions}

In this subsection we mostly introduce some notational conventions. They were already introduced in the beginning of Subsection~5.1 in~\cite{Smi16} and the beginning of Subsection~2.6 in~\cite{Smi14}.

For any linear map~$g$ acting on a Euclidean space~$E$, we write $\|g\| := \sup_{x \neq 0} \frac{\|g(x)\|}{\|x\|}$ its operator norm.

Consider a Euclidean space $E$. We introduce on the projective space $\mathbb{P}(E)$ a metric by setting, for every $\overline{x}, \overline{y} \in \mathbb{P}(E)$,
\begin{equation}
\alpha (\overline{x}, \overline{y}) := \arccos \frac{| \langle x, y \rangle |}{\|x\| \|y\|} \in \textstyle [0, \frac{\pi}{2}],
\end{equation}
where $x$ and $y$ are any vectors representing respectively $\overline{x}$ and $\overline{y}$ (obviously, the value does not depend on the choice of $x$ and $y$). This measures the angle between the lines $\overline{x}$ and $\overline{y}$. For shortness' sake, we will usually simply write $\alpha(x, y)$ with $x$ and $y$ some actual vectors in $E \setminus \{0\}$.

For any vector subspace $F \subset E$ and any radius $\eps > 0$, we shall denote the $\eps$-neighborhood of $F$ in $\mathbb{P}(E)$ by:
\begin{equation}
B_{\mathbb{P}}(F, \eps) := \setsuch{x \in \mathbb{P}(E)}{\alpha(x,\mathbb{P}(F)) < \eps}.
\end{equation}
(You may think of it as a kind of ``conical neighborhood''.)

Consider a metric space $(\mathcal{M}, \delta)$; let $X$ and $Y$ be two subsets of $\mathcal{M}$. We shall denote the ordinary, minimum distance between $X$ and $Y$ by
\begin{equation}
\delta(X, Y) := \inf_{x \in X} \inf_{y \in Y} \delta(x, y),
\end{equation}
as opposed to the Hausdorff distance, which we shall denote by
\begin{equation}
\delta^\mathrm{Haus}(X, Y) := \max\left( \sup_{x \in X} \delta\big(\{x\}, Y\big),\; \sup_{y \in Y} \delta\big(\{y\}, X\big) \right).
\end{equation}

Finally, we introduce the following notation. Let $X$ and $Y$ be two positive quantities, and $p_1, \ldots, p_k$ some parameters. Whenever we write
\[X \lesssim_{p_1, \ldots, p_k} Y,\]
we mean that there is a constant $K$, depending on nothing but $p_1, \ldots, p_k$, such that $X \leq KY$. (If we do not write any subscripts, this means of course that $K$ is an ``absolute'' constant --- or at least, that it does not depend on any ``local'' parameters; we consider the ``global'' parameters such as the choice of $G$ and of the Euclidean norms to be fixed once and for all.) Whenever we write
\[X \asymp_{p_1, \ldots, p_k} Y,\]
we mean that $X \lesssim_{p_1, \ldots, p_k} Y$ and $Y \lesssim_{p_1, \ldots, p_k} X$ at the same time.

The following result will often be useful:
\begin{lemma}[\cite{Smi14}]
\label{bounded_norm_is_bilipschitz}
Let $C \geq 1$. Then any map $\phi \in \GL(E)$ such that $\|\phi^{\pm 1}\| \leq C$ induces a $C^2$-Lipschitz continuous map on $\mathbb{P}(E)$.
\end{lemma}
\begin{proof}
See~\cite{Smi14}, Lemma~2.20.
\end{proof}

\section{Proximal maps}
\label{sec:proximal_maps_bis}

In this section, we give the definitions of the proximal versions of the concepts from Table~\ref{tab:notions} and state the proximal version of Schema~\ref{proposition_template}, namely Proposition~\ref{proximal_product}. It contains no new results.

Let $E$ be a Euclidean space.

\begin{definition}[Proximal version of regularity]
\label{proximal_definition}
Let $\gamma \in \GL(E)$; let $\lambda_1, \ldots, \lambda_n$ be its eigenvalues repeated according to multiplicity and ordered by nonincreasing modulus. We define the \emph{spectral radius} of~$\gamma$ as
\begin{equation}
r(\gamma) := |\lambda_1|
\end{equation}
(we do not use ``$\rho(\gamma)$'' as it could be confused with the representation~$\rho$). We say that $\gamma$~is \emph{proximal} if we have
\[|\lambda_2| < |\lambda_1| =: r(\gamma),\]
\ie if $\lambda_1$ is the only eigenvalue with modulus~$r(\gamma)$ and has multiplicity~$1$. Equivalently, $\gamma$~is proximal if and only if its \emph{spectral gap} $\frac{|\lambda_1|}{|\lambda_2|}$ is greater than~$1$.
\end{definition}

\begin{definition}[Proximal version of geometry]
\label{attracting_repelling_space_definition}
For every proximal map~$\gamma$, we may then decompose $E$ into a direct sum of a line $E^s_\gamma$, called the \emph{attracting space} of~$\gamma$, and a hyperplane $E^u_\gamma$, called the \emph{repelling space} of~$\gamma$, both stable by $\gamma$ and such that:
\[\begin{cases}
\restr{\gamma}{E^s_\gamma} = \lambda_1 \Id; \\
\text{for every eigenvalue } \lambda \text{ of } \restr{\gamma}{E^u_\gamma},\; |\lambda| < |\lambda_1|.
\end{cases}\]
\end{definition}

\begin{definition}[Proximal version of non-degeneracy]
\label{proximal_C_non_deg}
Consider a line $E^s$ and a hyperplane $E^u$ of $E$, transverse to each other. An \emph{optimal canonizing map} for the pair $(E^s, E^u)$ is a map $\phi \in GL(E)$ satisfying
\[\phi(E^s) \perp \phi(E^u)\]
and minimizing the quantity $\max \left( \|\phi\|, \|\phi^{-1}\| \right)$.

We define an \emph{optimal canonizing map} for a proximal map $\gamma \in \GL(E)$ to be an optimal canonizing map for the pair $(E^s_\gamma, E^u_\gamma)$.

Let $C \geq 1$. We say that the pair formed by a line and a hyperplane $(E^s, E^u)$ (resp. that a proximal map $\gamma$) is \emph{$C$-non-degenerate} if it has an optimal canonizing map $\phi$ such that $\left \|\phi^{\pm 1} \right\| \leq C$. (This is equivalent to the angle between $E^s$ and $E^u$ being bounded below by a constant that depends only on~$C$.)

Now take $\gamma_1, \gamma_2$ two proximal maps in $\GL(E)$. We say that the pair $(\gamma_1, \gamma_2)$ is \emph{$C$-non-degenerate} if every one of the four possible pairs $(E^s_{\gamma_i}, E^u_{\gamma_j})$ is $C$-non-degenerate.
\end{definition}

\begin{definition}[Proximal version of contraction strength]
\label{s_tilde_definition}
Let $\gamma \in \GL(E)$ be a proximal map. We define the \emph{proximal contraction strength} of $\gamma$ by
\[\tilde{s}(\gamma)
:= \frac{\left\| \restr{\gamma}{E^u_\gamma} \right\|}{\left\| \restr{\gamma}{E^s_\gamma} \right\|}
= \frac{\left\| \restr{\gamma}{E^u_\gamma} \right\|}{r(\gamma)}\]
(where $r(\gamma)$ is the spectral radius of~$\gamma$, equal to~$|\lambda_1|$ in the notations of the previous definition). We say that $\gamma$ is \emph{$\tilde{s}$-contracting} if $\tilde{s}(\gamma) \leq \tilde{s}$.
\end{definition}

\begin{proposition}
\label{proximal_product}
For every $C \geq 1$, there is a positive constant $\tilde{s}_{\ref{proximal_product}}(C)$ with the following property. Take a $C$-non-degenerate pair of proximal maps $\gamma_1, \gamma_2$ in $\GL(E)$, and suppose that both $\gamma_1$ and $\gamma_2$ are $\tilde{s}_{\ref{proximal_product}}(C)$-contracting. Then $\gamma_1 \gamma_2$ is proximal, and we have:
\begin{hypothenum}
\item $\alpha \left(E^s_{\gamma_1 \gamma_2},\; E^s_{\gamma_1} \right) \lesssim_C \tilde{s}(\gamma_1)$;
\item $\tilde{s}(\gamma_1 \gamma_2) \lesssim_C \tilde{s}(\gamma_1)\tilde{s}(\gamma_2)$;
\item $r(\gamma_1 \gamma_2) \asymp_C \|\gamma_1\| \|\gamma_2\|$.
\end{hypothenum}
\end{proposition}
(The constant~$\tilde{s}_{\ref{proximal_product}}(C)$ is indexed by the number of the proposition, a scheme that we will stick to throughout the paper.)

Similar results have appeared in the literature for a long time, see \eg Lemma~5.7 in~\cite{AMS02}, Proposition~6.4 in~\cite{Ben96} or Lemma~2.2.2 in~\cite{Ben97}.

\begin{proof}
See Proposition~3.4 in~\cite{Smi14} for the proof of (i) and~(ii), and Proposition~5.12 in~\cite{Smi16} for the proof of~(iii).
\end{proof}

\begin{remark}
\label{comparison_proximal_with_schema}
If we wanted to literally follow Schema~\ref{proposition_template} (taking ``asymptotic dynamics'' to mean the logarithm of the spectral radius), we would need:
\begin{itemize}
\item to add a point (i'): $\alpha \left(E^u_{\gamma_1 \gamma_2},\; E^u_{\gamma_2} \right) \lesssim_C \tilde{s}(\gamma_2)$;
\item to replace (iii) by (iii'): $r(\gamma_1 \gamma_2) \asymp_C r(\gamma_1) r(\gamma_2)$.
\end{itemize}
However:
\begin{itemize}
\item The estimate~(i') will not be used in the sequel. It is nevertheless true: it follows by considering the action of~$\GL(E)$ on the dual space~$E^*$ and applying~(i).
\item The estimate~(iii') is on the contrary not strong enough for the applications we need. It is also true; it follows by plugging the identity
\[r(\gamma) \asymp_C \|\gamma\|\]
(valid for proximal, $C$-non-degenerate, $\tilde{s}_{\ref{proximal_product}}(C)$-contracting~$\gamma$; obtained from~(iii) by setting $\gamma_1 = \gamma_2 = \gamma$) into (iii) itself.
\end{itemize}
\end{remark}

\section{$X$-regular linear maps}
\label{sec:linear_regular_maps}

In this section, we define the linear versions of the properties from Table~\ref{tab:notions} and state the linear version of Schema~\ref{proposition_template}. Most of the basic ideas, and several definitions, come from~\cite{Ben97}; however, we use a slightly different point of view. Benoist relies most of the time only on the proximal versions of the properties from Table~\ref{tab:notions}, using the representations~$\rho_i$ as a proxy. We, on the other hand, clearly separate the linear versions from the proximal versions, and establish the correspondences between linear and proximal versions as theorems.

In Subsection~\ref{sec:X-regular}, we give the definitions of these linear properties. All of them are parametrized by some vector~$X \in \mathfrak{a}^+$ (or equivalently by some subset $\Pi_X \subset \Pi$; see the discussion in the beginning of Subsection~\ref{sec:parabolics}). When we will apply the linear case to the affine case, this vector will be set to the vector~$X_0$ chosen in Section~\ref{sec:choice}.

In Subsection~\ref{sec:group_inverse}, we examine what happens to these properties when we replace~$g$ by its inverse~$g^{-1}$.

In Subsection~\ref{sec:linear_regular_product}, we relate the linear properties with the proximal properties, and then prove Propositions \ref{intrinsic_regular_product} and~\ref{jordan_additivity}, which together comprise the linear version of Schema~\ref{proposition_template}.

\subsection{Definitions}
\label{sec:X-regular}

Let us fix some $X \in \mathfrak{a}^+$.

\begin{definition}[Linear version of regularity]
\label{X-regular_definition}
We say that an element $g \in G$ is \emph{$X$-regular} if every root which does not vanish on~$X$ does not vanish on~$\jordan(g)$ either:
\[\forall \alpha \in \Pi \setminus \Pi_X,\quad \alpha(\jordan(g)) > 0.\]
\end{definition}
Benoist calls such elements ``elements of type~$\theta$'', where his $\theta$ is our set~$\Pi \setminus \Pi_X$: see Definition~3.2.2 in~\cite{Ben97}.
\begin{example}~
\begin{enumerate}
\item If $X \in \mathfrak{a}^{++}$, then $g \in G$ is $X$-regular if and only if $\jordan(g) \in \mathfrak{a}^{++}$.
\item If $X = 0$, then the condition is vacuous: all elements of~$G$ are $X$-regular.
\end{enumerate}
\end{example}
Elements $g \in G$ such that $\jordan(g) \in \mathfrak{a}^{++}$ are often called \emph{$\mathbb{R}$-regular} or \emph{loxodromic}. So informally, being $X$-regular should be understood as being ``partially $\mathbb{R}$-regular''. (In fact, technically we should probably say ``$X$-$\mathbb{R}$-regular'' instead of ``$X$-regular''.)

\begin{definition}[Linear version of geometry]
\label{geometry_definition}
Let $g \in G$ be an $X$-regular element. Let $g = g_h g_e g_u$ be its Jordan decomposition, and let $\phi$ be any element of~$G$ realizing the conjugacy
\[\phi g_h \phi^{-1} = \exp(\jordan(g))\]
(called a \emph{canonizing map} for~$g$). Then we define:
\begin{itemize}
\item the \emph{attracting $X$-flag} of~$g$, denoted by~$y^{X, +}_g$, to be the class of~$\phi^{-1}$ in the flag variety $G/P_X^+$:
\[y^{X, +}_g := \phi^{-1} P_X^+ \in G/P_X^+;\]
\item the \emph{repelling $X$-flag} of~$g$, denoted by~$y^{X, -}_g$, to be the class of~$\phi^{-1}$ in the flag variety $G/P_X^-$:
\[y^{X, -}_g := \phi^{-1} P_X^- \in G/P_X^-;\]
\item the \emph{$X$-geometry} of~$g$ to be the data of its attracting and repelling $X$-flags, \ie the pair $\left( y^{X, +}_g, y^{X, -}_g \right)$.
\end{itemize}
\end{definition}
Benoist defines the attracting and repelling flags in the last sentence of~3.4 in~\cite{Ben97}.
\begin{remark}
Depending on context, sometimes it is the map~$\phi$ itself that is more relevant to consider, and sometimes it is its inverse. Indeed while $\phi$ is the map that ``brings $g$ to the canonical position'', its inverse $\phi^{-1}$ is the map that ``defines the geometry of~$g$'', starting from the canonical position. This is why the formulas above involve~$\phi^{-1}$. 
\end{remark}
We need to check that those definitions do not depend on the choice of~$\phi$. Indeed, $\phi^{-1}$ is unique up to multiplication on the right by an element of the centralizer of~$\jordan(g)$. By Proposition~\ref{centralizer_of_x}, the latter is equal to~$L_{\jordan(g)}$; since $g$ is $X$-regular, it is contained in~$L_X$, which in turn is contained both in~$P^+_X$ and in~$P^-_X$.

\begin{definition}
\label{transverse_purely_linear}
We say that a pair $(y^+, y^-) \in G/P^+_X \times G/P^-_X$ is \emph{transverse} if the intersection of $y^+$ and~$y^-$ (seen as cosets in~$G$) is nonempty, \ie if there exists an element $\phi \in G$ such that
\[\begin{cases}
y^+ = \phi P^+_X; \\
y^- = \phi P^-_X.
\end{cases}\]
\end{definition}
In particular, the pair of flags giving the geometry of any $X$-regular element is transverse.
(Compare this with the definition given in~\cite{Ben97} in~3.3, p.~14, third line from the end.)
\begin{proposition}
\label{GL_GPGP}
The map
\[\fundef{\Phi_X:}{G/L_X}{G/P^+_X \times G/P^-_X}
{\phi L_X}{\left( \phi P^+_X, \phi P^-_X \right)}\]
gives a canonical diffeomorphism between~$G/L_X$ and the subset of~$G/P^+_X \times G/P^-_X$ formed by transverse pairs.
\end{proposition}
From now on, we shall tacitly identify $G/L_X$ with the set of transverse pairs (which is also known as the open $G$-orbit in $G/P^+_X \times G/P^-_X$).
\begin{proof}
The group~$G$ acts smoothly on the manifold~$G/P^+_X \times G/P^-_X$. The orbit of the point~$(P^+_X, P^-_X)$ is precisely the set of transverse pairs, and its stabilizer is~$P^+_X \cap P^-_X = L_X$.
\end{proof}

Our map~$\Phi_X$ corresponds to Benoist's ``injection $Z_\theta \to Y_\theta \times Y_\theta^-$'' introduced near the end of~3.3 in~\cite{Ben97}.

\begin{definition}
On every flag variety~$G/P^+_X$ and on every flag variety~$G/P^-_X$, we now fix, once and for all, a distance coming from some Riemannian metric. All these distances shall be denoted by~$\delta$.
\end{definition}

\begin{remark}
Note that every flag variety~$G/P^+_X$ (and~$G/P^-_X$, which is isomorphic) is compact. Indeed by the Iwasawa decomposition (see \cite{Kna96}, Theorem 6.46), the maximal compact subgroup~$K$ acts transitively on it. This means that any two Riemannian metrics on a given flag variety are always Lipschitz-equivalent. It turns out that we will only be interested in properties that are true up to a multiplicative constant; so the choice of a Riemannian metric does not influence anything in the sequel.
\end{remark}

We now introduce the notion of $C$-non-degeneracy, which is basically a quantitative measure of transversality. Every transverse pair of flags is $C$-non-degenerate for \emph{some} constant~$C$; but the smaller the constant gets, the ``more strongly'' the flags are transverse.

In~\cite{Ben97}, this notion appears bundled together with contraction strength in the concept of ``$(\theta, \eps)$-proximality'' (for a single element) or ``$(\theta, \eps)$-Schottky-ness'' (for a pair or, more generally, a family of elements).

\begin{definition}
\label{nu_definition}
We fix, once and for all, a continuous proper map
\[\nu: G \to [0, +\infty).\]

A typical example of such a map is given by
\begin{equation}
\label{eq:nu_example}
\nu(g) = \max \left( \|\rho_*(g)\|,\; \|\rho_*(g)^{-1}\| \right),
\end{equation}
where $\rho_*$ can be any faithful representation of~$G$ and $\| \bullet \|$ can be any Euclidean norm on the representation space. The specific choice of~$\nu$ is not really important: see the remark below. In practice, we will indeed find it convenient to use a very specific map~$\nu$ of this form: see \eqref{eq:nu_compatible_definition}.

The important property is that then the family of the preimages~$\nu^{-1}([0,C])$, indexed by~$C \in [0, +\infty)$, is a nested family of compact sets whose union exhausts the set $G$.
\end{definition}
\begin{definition}[Linear version of non-degeneracy]
\label{linear-C-non-deg-definition}
Note that the last statement also holds for the projections of these preimages onto~$G/L_X$. We may call these projections $\nu_X^{-1}([0,C])$, where we set
\[\fundef{\nu_X:}{G/L_X}{[0, +\infty)}
{\phi L_X}{\displaystyle \min_{\phi' \in \phi L_X} \nu(\phi').}\]
To justify that this is well-defined, notice that for every $\phi \in G$, the intersection of the coset~$\phi L_X$ with~$\nu^{-1}([0,\nu(\phi)])$ is compact (and nonempty), so the continuous map~$\nu$ reaches a minimum on it. Also, the map~$\nu_X$ is still proper.

We say that an element $\phi \in G$ is an \emph{optimal representative} of the coset~$\phi L_X \in G/L_X$ if $\nu$ reaches its minimum at~$\phi$, \ie if we have $\nu(\phi) = \nu_X(\phi L_X)$.
\begin{itemize}
\item We say that a transverse pair $(y^+, y^-) \in G/P^+_X \times G/P^-_X$ is \emph{$C$-non-degenerate} if $\nu_X(y^+, y^-) \leq C$, or in other terms
\[(y^+, y^-) \in \nu_X^{-1}([0,C])\]
(where we identify $(y^+, y^-)$ with a coset of~$G/L_X$ by the map~$\Phi_X$ defined above).
\item We say that an $X$-regular element $g \in G$ is \emph{$C$-non-degenerate} if its $X$-geometry is $C$-non-degenerate, \ie if
\[\left( y^{X, +}_g,\; y^{X, -}_g \right) \in \nu_X^{-1}([0,C]).\]
\item We say that a pair $(g_1, g_2)$ of $X$-regular elements of~$G$ is \emph{$C$-non-degenerate} if we have
\[\left( y^{X, +}_{g_i},\; y^{X, -}_{g_j} \right) \in \nu_X^{-1}([0,C])\]
for all four possible pairs $(i, j) \in \{1, 2\} \times \{1, 2\}$.
\end{itemize}
\end{definition}
\begin{remark}
\label{choice_of_nu_does_not_matter}
Let us now explain why the choice of the function~$\nu$ is not really important. Indeed, suppose we replace the function~$\nu$ by another function~$\nu'$ having the same property. Then we will simply need to replace every constant~$C$ by some~$C'$ that depends only on~$C$.
\end{remark}
\begin{example}
Take $G = \SO^+(n,1)$. This group has real rank~$1$; so the closed Weyl chamber~$\mathfrak{a}^+$ is a half-line, with only two facets: $\{0\}$ and its interior~$\mathfrak{a}^{++}$. Taking $X = 0$ makes everything trivial; so assume that $X \in \mathfrak{a}^{++}$. Let us identify~$G$ with the group of isometries of the hyperbolic space~$\mathbb{H}^n$. In this case:
\begin{itemize}
\item An element $g \in G$ is then $X$-regular if and only if it is \emph{loxodromic}, \ie fixes exactly two points of the ideal boundary $\partial^\infty \mathbb{H}^n$.
\item The flag variety~$G/P_X^+$ canonically identifies with this ideal boundary~$\partial^\infty \mathbb{H}^n$, and so does the opposite flag variety~$G/P_X^-$.
\item The attracting flag~$y_g^{X, +}$ (resp. repelling flag~$y_g^{X, -}$) of an $X$-regular element~$g$ corresponds to the attracting (resp. repelling) fixed point at infinity of the loxodromic isometry~$g$.
\item Two flags $(y^+, y^-) \in G/P_X^+ \times G/P_X^-$ are transverse if and only if the corresponding points of~$\partial^\infty \mathbb{H}^n$ are distinct.
\item One possible choice of the ``non-degeneracy'' function~$\nu_X$ is as follows. Choose any reference point~$x_0 \in \mathbb{H}^n$, and let $\delta^+$ and~$\delta^-$ be half-lines starting at~$x_0$ and reaching the ideal boundary at points corresponding to~$y^+$ and~$y^-$ respectively. Then we may let $\nu_X(y^+, y^-)$ be the reciprocal of the angle between $\delta^+$ and~$\delta^-$.
\item In that case, a pair $(y^+, y^-)$ is $C$-non-degenerate if and only if the corresponding points of the ideal boundary are separated by an angle of at least $\frac{1}{C}$ (when looking from~$x_0$).
\end{itemize}
\end{example}

We finish this subsection by introducing the following notion.

\begin{definition}[Linear version of contraction strength]
\label{vec_s_definition}
We define the \emph{linear $X$\nobreakdash-\hspace{0pt}contraction strength} of an $X$-regular element~$g \in G$ to be the quantity
\[\vec{s}_X(g) := \exp \left( - \min_{\alpha \in \Pi \setminus \Pi_X} \alpha(\cartan(g)) \right).\]
It measures how far the Cartan projection of~$g$ is from the walls of the Weyl chamber, except those containing~$X$. The arrow serves to distinguish it from the affine contraction strength $s_{X_0}$ introduced in Definition~\ref{s_definition}.
\end{definition}

\subsection{Impact of the group inverse}
\label{sec:group_inverse}

In this section, we examine what happens to the properties we just introduced when we pass from an element~$g \in G$ to its inverse~$g^{-1}$. Though slightly technical, the proof is completely straightforward.

We start by observing that for every $g \in G$, we have
\begin{equation}
\label{eq:inverse_jordan_projection}
\jordan(g^{-1}) = -w_0(\jordan(g))
\end{equation}
and
\begin{equation}
\label{eq:inverse_cartan_projection}
\cartan(g^{-1}) = -w_0(\cartan(g)).
\end{equation}
(The map $-w_0$ is Benoist's ``opposition involution''~$\imath$: compare these formulas with~\cite{Ben97}, Section~2.4, p.~8, l.~10.)

The first identity immediately follows from the definitions of the Jordan projection and of~$w_0$. The second identity also follows from the definitions, using the well-known fact that every element of the restricted Weyl group~$W$ (in particular~$w_0$) has a representative in the maximal compact subgroup~$K$ (see \cite{Kna96}, formulas (7.84a) and~(7.84b)).

\begin{proposition}~
\label{intrinsic_inverse}
\begin{hypothenum}
\item \label{itm:w0_X-reg} An element $g \in G$ is $X$-regular if and only if its inverse~$g^{-1}$ is $-w_0(X)$-regular;
\item \label{itm:w0_flags} For every $X$-regular element $g \in G$, we have
\[\iota^-_X \left( y^{X, -}_g \right) = y^{-w_0(X), +}_{g^{-1}},\]
where $\iota^\pm_X$ are the diffeomorphisms given by
\[\fundef{\iota^\pm_X:}{G/P^\pm_X}{G/P^\mp_{-w_0(X)}}
{\phi P^\pm_X}{\phi w_0 P^\mp_{-w_0(X)}.}\]
\item \label{itm:w0_C-non-deg} If a pair
\[(y^+, y^-) \in G/P^+_X \times G/P^-_X\]
is $C$-non-degenerate, then the pair
\[(\iota^-_X(y^-),\; \iota^+_X(y^+)) \in G/P^+_{-w_0(X)} \times G/P^-_{-w_0(X)}\]
is $C'$-non-degenerate for some constant~$C'$ that depends only on~$C$.
\item \label{itm:w0_contraction} For every $X$-regular element $g \in G$, we have
\[\vec{s}_X(g) = \vec{s}_{-w_0(X)}(g^{-1}).\]
\end{hypothenum}
\end{proposition}
\begin{remark}
Starting from Section~\ref{sec:choice}, we will only consider situations where $X$ will be \emph{symmetric}, \ie $-w_0(X) = X$ (which simplifies the above formulas).
\end{remark}
\begin{proof}~
\begin{hypothenum}
\item This is an immediate consequence of~\eqref{eq:inverse_jordan_projection}.
\item To show that the map~$\iota^-_X$ is well-defined, note that we have $w_0^{-1} P^-_X w_0 = P^+_{-w_0(X)}$ (this easily follows from the definitions). It is obviously smooth. Hence the map~$\iota^+_{-w_0(X)}$ is also smooth, and is clearly equal to the inverse of~$\iota^-_X$. To show the desired identity, note that (by~\eqref{eq:inverse_jordan_projection}) if $\phi$~is a canonizing map for~$g$, then $w_0 \phi$~is a canonizing map for~$g^{-1}$. (Pay attention to the convention of $\phi$ versus~$\phi^{-1}$.)
\item The map $\phi \mapsto \phi w_0$ descends to a diffeomorphism~$\iota^L_X$ that makes the diagram
\begin{equation}\begin{tikzcd}[column sep = huge]
G/L_X \arrow{r}{\Phi_X} \arrow{dd}{\iota^L_X}
&G/P^+_X \times G/P^-_X
  \arrow[white]{dd}{{\color{black} \hspace{1ex} \sigma \circ (\iota^+_X \times \iota^-_X)}}
  \arrow[rounded corners, to path =
       {([xshift=-1ex]\tikztostart.south)
    -- ([xshift=-1ex, yshift=1ex]$(\tikztostart.south)!.5!(\tikztotarget.north)$)
    -- ([xshift=1ex, yshift=-1ex]$(\tikztostart.south)!.5!(\tikztotarget.north)$)
    -- ([xshift=1ex]\tikztotarget.north)}]{dd}
  \arrow[rounded corners, to path =
       {([xshift=1ex]\tikztostart.south)
    -- ([xshift=1ex, yshift=1ex]$(\tikztostart.south)!.5!(\tikztotarget.north)$)
    -- ([xshift=-1ex, yshift=-1ex]$(\tikztostart.south)!.5!(\tikztotarget.north)$)
    -- ([xshift=-1ex]\tikztotarget.north)}]{dd}\\
& \\
G/L_{-w_0(X)} \arrow{r}{\Phi_{-w_0(X)}}
&G/P^+_{-w_0(X)} \times G/P^-_{-w_0(X)}
\end{tikzcd}\end{equation}
commutative. Here $\Phi_X$~is the embedding from Proposition~\ref{GL_GPGP}, and $\sigma$ denotes the pair-switching map: $(a, b) \mapsto (b, a)$ (the crossing-over double arrow is meant to suggest this graphically). So clearly the map $\sigma \circ (\iota^+_X \times \iota^-_X)$ preserves transversality of pairs. Now for every~$C$, the map~$\iota^L_X$ maps the preimage $\nu_X^{-1}([0,C])$ to some compact subset of~$G/L_{-w_0(X)}$, which is in particular contained in the preimage $\nu_{-w_0(X)}^{-1}([0, C'])$ for some $C'$.
\item This is an immediate consequence of \eqref{eq:inverse_cartan_projection}. \qedhere
\end{hypothenum}
\end{proof}
\begin{remark}~
\label{iota_isometry}
\begin{itemize}
\item In point~\ref{itm:w0_X-reg}, since the choice of the metrics on the flag varieties was arbitrary, we lose no generality in assuming that the diffeomorphisms~$\iota^\pm_X$ defined above are actually isometries.
\item In point~\ref{itm:w0_C-non-deg}, if $\nu$ is chosen in a sufficiently natural way (for example as defined by~\eqref{eq:nu_compatible_definition}), we may actually let $C' = C$.
\end{itemize}
\end{remark}

\subsection{Products of $X$-regular maps}
\label{sec:linear_regular_product}

In this subsection, we start by proving a few results that link linear properties to proximal properties. As a key tool for this, we will use the representations~$\rho_i$ (on spaces called~$V_i$) defined in Proposition~\ref{fundamental_real_representation}. More precisely:
\begin{itemize}
\item Proposition~\ref{type_X0_to_proximality}.\ref{X_regular_to_proximal} links the linear and proximal versions of regularity;
\item Lemma~\ref{PS_to_PVi_and_PVistar} links the linear and proximal versions of geometry;
\item Proposition~\ref{C-non-deg-in-Vi} links the linear and proximal versions of non-degeneracy;
\item Proposition~\ref{type_X0_to_proximality}.\ref{X_to_proximal_contraction_strength} links the linear and proximal versions of contraction strength.
\end{itemize}
Their proofs are adapted from Section~6 in~\cite{Smi16}.

We then use these results to deduce the linear version of Schema~\ref{proposition_template} from the proximal version. The linear version comprises two parts:
\begin{itemize}
\item Proposition~\ref{intrinsic_regular_product} is the ``main'' part. This result did not appear in~\cite{Smi16}.
\item Proposition~\ref{jordan_additivity} is the ``asymptotic dynamics'' part. It gives the same conclusion as Proposition~6.11 in~\cite{Smi16}, but uses in its hypotheses the linear versions of the properties (which are the most natural here) rather than the affine versions.
\end{itemize}

\begin{proposition}~
\label{type_X0_to_proximality}
\begin{hypothenum}
\item \label{X_regular_to_proximal}An element $g \in G$ is $X$-regular if and only if for every~$i \in \Pi \setminus \Pi_X$, the map~$\rho_i(g)$ is proximal. 
\item \label{X_to_proximal_contraction_strength} For every~$C \geq 1$, there is a constant $\vec{s}_{\ref{type_X0_to_proximality}}(C)$ with the following property. Let $g \in G$ be a $C$-non-degenerate $X$-regular map such that $\vec{s}_X(g) \leq \vec{s}_{\ref{type_X0_to_proximality}}(C)$. Then for every~$i \in \Pi \setminus \Pi_X$, we have
\[\tilde{s}(\rho_i(g)) \lesssim_C \vec{s}_X(g).\]
\end{hypothenum}
\end{proposition}

These two statements essentially correspond to the respective left halves of~(6.12) and~(6.13) in~\cite{Smi16}. The proof is also essentially the same.

Part~\ref{X_regular_to_proximal} is given in Definition~3.2.2 in~\cite{Ben97}.

\begin{remark}~
\begin{itemize}
\item Note that since all Euclidean norms on a finite-dimensional vector space are equivalent, this estimate makes sense even though we did not specify any norm on~$V_i$. In the course of the proof, we shall choose one that is convenient for us.
\item Recall that ``$i \in \Pi \setminus \Pi_X$'' is a notation shortcut for ``$i$ such that $\alpha_i \in \Pi \setminus \Pi_X$''.
\item $\Pi_X$ should be thought of as a kind of ``exceptional set''. In practice, it will often be empty (see Remark~\ref{X0_examples} below).
\end{itemize}
\end{remark}

\needspace{\baselineskip}
\begin{proof}[Proof of Proposition~\ref{type_X0_to_proximality}]~
\begin{hypothenum}
\item For each $i \in \Pi \setminus \Pi_X$, we will estimate (in \eqref{eq:spectral_gap_estimate}) the spectral gap of~$\rho_i(g)$ in terms of~$\jordan(g)$. Recalling the respective definitions of proximality and of $X$-regularity, the conclusion will then follow immediately.

We fix some $i \in \Pi \setminus \Pi_X$. By Proposition~\ref{eigenvalues_and_singular_values_characterization}~(i), the list of the moduli of the eigenvalues of~$\rho_i(g)$ is precisely
\[\left( e^{\lambda_i^j(\jordan(g))} \right)_{1 \leq j \leq d_i},\]
where $d_i$ is the dimension of~$V_i$ and~$(\lambda_i^j)_{1 \leq j \leq d_i}$ is the list of restricted weights of~$\rho_i$ listed with multiplicity. It remains to determine which are the largest two of them (see \eqref{eq:highest_eigenvalue}).

Up to reordering that list, we may suppose that~$\lambda_i^1 = n_i \varpi_i$ is the highest restricted weight of~$\rho_i$. We may also suppose that~$\lambda_i^2 = n_i \varpi_i - \alpha_i$. Indeed we have
\begin{align}
s_{\alpha_i}(n_i \varpi_i)
  &= s_{\alpha'_i}(n_i \varpi_i) \nonumber \\
  &= n_i \varpi_i - 2 n_i \frac{\langle \varpi_i, \alpha'_i \rangle}{\langle \alpha'_i, \alpha'_i \rangle}\alpha'_i \nonumber \\
  &= n_i \varpi_i - n_i \alpha'_i
\end{align}
(recall that~$\alpha'_i$ is equal to~$2\alpha_i$ if~$2\alpha_i$ is a restricted root and to~$\alpha_i$ otherwise). But by Proposition~\ref{convex_hull}, $s_{\alpha_i}(n_i \varpi_i)$~is a restricted weight of~$\rho_i$ (because it is the image of a restricted weight of~$\rho_i$ by an element of the Weyl group) and then $n_i \varpi_i - \alpha_i$ is also a restricted weight of~$\rho_i$ (as a convex combination of two restricted weights of~$\rho_i$, that belongs to the restricted root lattice shifted by~$n_i \varpi_i$).

Take any~$j > 2$. Since by hypothesis, the restricted weight~$n_i \varpi_i$ has multiplicity~$1$, we have $\lambda_i^j \neq \lambda_i^1$. By Lemma~\ref{fund_repr_other_weights}, it follows that this restricted weight has the form
\innerneedspace{\baselineskip}
\[\lambda_i^j = n_i \varpi_i - \alpha_i - \sum_{i' = 1}^r c_{i'} \alpha_{i'},\]
with $c_{i'} \geq 0$ for every index~$i'$.

Finally, since by definition $\jordan(g) \in \mathfrak{a}^{+}$, for every index~$i'$ we have~$\alpha_{i'}(\jordan(g)) \geq 0$. It follows that for every~$j > 2$, we have
\begin{equation}
\label{eq:highest_eigenvalue}
\lambda_i^1(\jordan(g)) \geq \lambda_i^2(\jordan(g)) \geq \lambda_i^j(\jordan(g)).
\end{equation}
In other words, among the moduli of the eigenvalues of~$\rho_i(g)$, the largest is
\[\exp \Big( \lambda_i^1(\jordan(g)) \Big) = \exp \Big( n_i \varpi_i(\jordan(g)) \Big),\]
and the second largest is
\[\exp \Big( \lambda_i^2(\jordan(g)) \Big) = \exp \Big( n_i \varpi_i(\jordan(g)) - \alpha_i(\jordan(g)) \Big).\]
It follows that the spectral gap of~$\rho_i(g)$ is equal to
\begin{equation}
\label{eq:spectral_gap_estimate}
\frac{\exp \Big( \lambda_i^1(\jordan(g)) \Big)}
     {\exp \Big( \lambda_i^2(\jordan(g)) \Big)} = \exp(\alpha_i(\jordan(g)))
\end{equation}
as desired.

\item Let~$\vec{s}_{\ref{type_X0_to_proximality}}(C)$ be a constant small enough to satisfy all the constraints that will appear in the course of the proof.

To prove the desired inequality, we will prove the following four intermediate steps, the main step being~\eqref{eq:prox_contr_strength_expression} (with a proof closely analogous to the proof of~\eqref{eq:spectral_gap_estimate}):
\begin{align}
\tilde{s}(\rho_i(g)) &\asymp_C \tilde{s}(\rho_i(g')) \label{eq:srhoig_srhoig'} \\
                     &=_{\phantom{C}} \exp(\alpha_i(\cartan(g')))^{-1} \label{eq:prox_contr_strength_expression} \\
                     &\asymp_C \exp(\alpha_i(\cartan(g)))^{-1} \label{eq:cartan_canonization} \\
                     &\leq_{\phantom{C}} \vec{s}_X(g), \label{eq:lin_contr_strength_straightforward}
\end{align}
where $g'$ is some conjugate of~$g$ to be chosen suitably. Also we assume here that we use, on the space~$V_i$ where the representation~$\rho_i$ acts, a $K$-invariant Euclidean form~$B_i$ such that all the restricted weight spaces for~$\rho_i$ are pairwise $B_i$\nobreakdash-\hspace{0pt}orthogonal (which exists by Lemma~\ref{K-invariant} applied to~$\rho_i$). If we worked with a different norm, the final result would of course stay the same (since all norms are equivalent), but some of the intermediate equalities would have to be replaced by ``$\asymp_C$'' estimates.

\begin{itemize}
\item Let $g \in G$ be a map satisfying the hypotheses, and let us fix $i \in \Pi \setminus \Pi_X$. To obtain~\eqref{eq:srhoig_srhoig'}, let us define~$g' := \phi g \phi^{-1}$, where $\phi$ is an optimal representative of the coset in $G/L_X$ giving the geometry of~$g$. Then \eqref{eq:srhoig_srhoig'} follows from the $C$-non-degeneracy of~$g$, by the following observation: for every~$C \geq 1$, the continuous map
\begin{equation}
\label{eq:rho_i_norm_map}
\phi \mapsto \max \Big( \left\|\rho_i(\phi\vphantom{^{-1}})\right\|,\; \left\|\rho_i(\phi^{-1})\right\| \Big)
\end{equation}
is bounded above on the compact set $\nu^{-1}([0,C])$, by some constant~$C'_i$ that depends only on~$C$ (and on the choice of a norm on~$V_i$).

\item The $C$-non-degeneracy of~$g$ also ensures~\eqref{eq:cartan_canonization}, as a corollary of the following important fact: the difference between $\cartan(g')$ and $\cartan(g)$ is bounded in norm by a constant that depends only on~$C$; or more concisely
\begin{equation}
\label{eq:cartan_g_g'}
\|\cartan(g') - \cartan(g)\| \lesssim_C 1.
\end{equation}
This can easily be proved coordinate by coordinate, using the $(\varpi_i)_{i \in \Pi}$ coordinate system. Indeed, note that we have
\begin{equation}
\forall i \in \Pi,\quad \varpi_i(\cartan(g)) = \frac{1}{n_i} \log \|\rho_i(g)\|,
\end{equation}
and similarly for~$g'$ (this is a consequence of Proposition~\ref{eigenvalues_and_singular_values_characterization}~(ii)). So it suffices to prove that $\|\rho_i(g)\| \asymp_C \|\rho_i(g')\|$ for every~$i$; but this is just another application of the remark about the boundedness of the map~\eqref{eq:rho_i_norm_map} on compacts.

\item The inequality~\eqref{eq:lin_contr_strength_straightforward} is just a straightforward consequence of the definition of~$\vec{s}_X$.

\item It remains to prove the main part, namely~\eqref{eq:prox_contr_strength_expression}. Note that, thanks to our choice of~$g'$ (and of the Euclidean norm~$B_i$), $\tilde{s}(\rho_i(g'))$ is simply the quotient of the two largest singular values of~$\rho_i(g')$. So the equality~\eqref{eq:prox_contr_strength_expression} can be proved in the same way as~\eqref{eq:spectral_gap_estimate}; the only thing to change is to replace all uses of Proposition~\ref{eigenvalues_and_singular_values_characterization}~(i) by uses of Proposition~\ref{eigenvalues_and_singular_values_characterization}~(ii) (that gives the singular values of an element of~$G$ in a given representation). \qedhere
\end{itemize}
\end{hypothenum}
\end{proof}

\begin{proposition}
\label{C-non-deg-in-Vi}
Let $(g_1, g_2)$ be a $C$-non-degenerate pair of $X$-regular elements of~$G$. Then for every~$i \in \Pi \setminus \Pi_X$, the pair~$(\rho_i(g_1), \rho_i(g_2))$ is a $C'$-non-degenerate pair of proximal maps in~$\GL(V_i)$, where $C'$~is some constant that depends only on~$C$.
\end{proposition}

This is a straightforward generalization of Proposition~6.7 in~\cite{Smi16}. Its proof relies on the following two lemmas, which are analogous to Lemmas~6.8 and~6.10 in~\cite{Smi16}. However, the two lemmas that follow are now formulated more generally; the stronger statements will be useful in order to prove Proposition~\ref{intrinsic_regular_product}.

\begin{lemma}
\label{stabilizer_Es_Eu}
We have:
\begin{hypothenum}
\item $\displaystyle \bigcap_{i \in \Pi \setminus \Pi_X} \Stab_G \Bigg( V^{n_i \varpi_i}_i \Bigg) = P_X^+$;
\item $\displaystyle \bigcap_{i \in \Pi \setminus \Pi_X} \Stab_G \Bigg( \bigoplus_{\lambda \neq n_i \varpi_i} V^\lambda_i \Bigg) = P_X^-$.
\end{hypothenum}
\end{lemma}
\begin{proof}
For every~$i \in \Pi$, let~$\Omega_i$ be the set of restricted weights of the representation~$\rho_i$.
\begin{hypothenum}
\item We begin by noting that for every $i$, we have
\[\Stab_G \Bigg( V^{n_i \varpi_i}_i \Bigg) = P^+ \Stab_W(n_i \varpi_i) P^+ = P^+ W_{\varpi_i} P^+.\]
This follows from Lemma~\ref{generalized_Bruhat}: indeed the singleton $\{n_i \varpi_i\}$ is clearly a top-subset in~$\Omega_i$. Corollary~\ref{parabolic_Bruhat_decomposition} then gives us that
\begin{equation}
\Stab_G \Bigg( V^{n_i \varpi_i}_i \Bigg) = P^+_{\varpi_i}.
\end{equation}
Now it is an easy exercise to show that for any $X, Y, Z \in \mathfrak{a}^+$, we have
\begin{align}
P^+_X \cap P^+_Y = P^+_Z &\iff \mathfrak{p}^+_X \cap \mathfrak{p}^+_Y = \mathfrak{p}^+_Z \nonumber \\
                         &\iff W_X \cap W_Y = W_Z \nonumber \\
                         &\iff \Pi_X \cap \Pi_Y = \Pi_Z.
\end{align}
Now obviously $\Pi_{\varpi_i} = \Pi \setminus \{\alpha_i\}$, and if we remove from~$\Pi$ all elements that are not in~$\Pi_X$, we are left with precisely~$\Pi_X$. Thus
\begin{equation}
\bigcap_{i \in \Pi \setminus \Pi_X} P^+_{\varpi_i} = P^+_X,
\end{equation}
and the conclusion follows.

\item Note that for every~$i$, the complement $\Omega_i \setminus \{n_i \varpi_i\}$ is a bottom-subset of~$\Omega_i$, and has the same stabilizer in~$W$ as~$\{n_i \varpi_i\}$. We may then follow the same line of reasoning. \qedhere
\end{hypothenum}
\end{proof}

In the following lemma, we identify the projective space~$\mathbb{P}(V_i)$ with the set of vector lines in~$V_i$ and the projective space~$\mathbb{P}(V^*_i)$ with the set of vector hyperplanes of~$V_i$.
\begin{lemma}
\label{PS_to_PVi_and_PVistar}
There exist two smooth embeddings
\[\Phi_X^s: G/P^+_X \to \prod_{i \in \Pi \setminus \Pi_X} \mathbb{P}(V_i)
\quad\text{ and }\quad
\Phi_X^u: G/P^-_X \to \prod_{i \in \Pi \setminus \Pi_X} \mathbb{P}(V^*_i)\]
with the following properties:
\begin{hypothenum}
\item For every $X$-regular map~$g \in G$, we have
\[\begin{cases}
\Phi_X^s(y^{X, +}_g) = \Big( E^s_{\rho_i(g)} \Big)_{i \in \Pi \setminus \Pi_X}; \\
\Phi_X^u(y^{X, -}_g) = \Big( E^u_{\rho_i(g)} \Big)_{i \in \Pi \setminus \Pi_X}.
\end{cases}\]
\item If $(y^+, y^-) \in G/P^+_X \times G/P^-_X$ is a transverse pair, then for every $i \in \Pi \setminus \Pi_X$, we have $\Phi_X^s(y^+)_i \not\in \Phi_X^u(y^-)_i.$
\end{hypothenum}
\end{lemma}
Our map~$\Phi_X^s$ is the map~$j_\theta$ defined in the beginning of~3.4 in~\cite{Ben97}. Statement~(ii) appears a little further in the same text (together with its converse which is also true).
\begin{proof}
We define the maps $\Phi_X^s$ and~$\Phi_X^u$ in the following way. For every $\phi \in G$, we set
\begin{equation}
\begin{cases}
\Phi_X^s(\phi P^+_X) = \Big( \rho_i(\phi) \left( V^{n_i \varpi_i}_i \right) \Big)_{i \in \Pi \setminus \Pi_X}; \\
\Phi_X^u(\phi P^-_X) = \Big( \rho_i(\phi) \left( \bigoplus_{\lambda \neq n_i \varpi_i} V^\lambda_i \right) \Big)_{i \in \Pi \setminus \Pi_X}.
\end{cases}
\end{equation}

These maps are well-defined and injective by Lemma~\ref{stabilizer_Es_Eu}, and obviously continuous. They are by construction $G$-equivariant; hence to prove that they are smooth embeddings, it is sufficient to prove that both of them have injective differential at the identity coset. This also follows from Lemma~\ref{stabilizer_Es_Eu} (by differentiating it).

To show property~(i), we essentially use:
\begin{itemize}
\item the identities
\begin{equation}
\begin{cases}
E^s_{\rho_i(\exp(\jordan(g)))} = V^{n_i \varpi_i}_i; \\
E^u_{\rho_i(\exp(\jordan(g)))} = \bigoplus_{\lambda \neq n_i \varpi_i} V^\lambda_i,
\end{cases}
\end{equation}
which follow from the inequality~\eqref{eq:highest_eigenvalue} ranking the values of different restricted weights of~$\rho_i$ evaluated at~$\jordan(g)$;
\item and the simple observation that any eigenspace of~$\phi g \phi^{-1}$ is the image by~$\phi$ of the eigenspace of~$g$ with the same eigenvalue.
\end{itemize}
Property~(ii) is obvious from the definitions.
\end{proof}

\begin{proof}[Proof of Proposition~\ref{C-non-deg-in-Vi}]
Let $C \geq 1$. Then the set of $C$-non-degenerate transverse pairs in $G/P^+_X \times G/P^-_X$ is compact. On the other hand, the function
\[(y^+, y^-) \mapsto \max_{i \in \Pi \setminus \Pi_X} \alpha(\Phi_X^s(y^+)_i, \Phi_X^u(y^-)_i)\]
is continuous, and (by Lemma~\ref{PS_to_PVi_and_PVistar}~(ii)) takes positive values on that set. Hence it is bounded below. So there is a constant~$C' \geq 1$, depending only on~$C$, such that whenever a transverse pair $(y^+, y^-)$ is $C$-non-degenerate, all pairs $(\Phi_X^s(y^+)_i, \Phi_X^u(y^-)_i)$ are $C'$-non-degenerate.

The conclusion then follows by Lemma~\ref{PS_to_PVi_and_PVistar}~(i).
\end{proof}

\begin{proposition}
\label{intrinsic_regular_product}
For every $C \geq 1$, there is a positive constant $\vec{s}_{\ref{intrinsic_regular_product}}(C) < 1$ with the following property. Take any $C$-non-degenerate pair $(g, h)$ of $X$-regular maps; suppose that we have $\vec{s}_X(g) \leq \vec{s}_{\ref{intrinsic_regular_product}}(C)$ and $\vec{s}_X(h) \leq \vec{s}_{\ref{intrinsic_regular_product}}(C)$. Then $gh$ is still $X$-regular, and we have:
\[\begin{cases}
\delta \left(y^{X, +}_{gh},\; y^{X, +}_{g} \right) \lesssim_C \vec{s}_X(g); \vspace{1mm} \\
\delta \left(y^{X, -}_{gh},\; y^{X, -}_{h} \right) \lesssim_C \vec{s}_X(h).
\end{cases}\]
\end{proposition}
\begin{remark}
We did not include the conclusion about contraction strength, namely that
\begin{equation}
\label{eq:linear_contraction_strength_additive}
\vec{s}_X(gh) \asymp_C \vec{s}_X(g) \vec{s}_X(h).
\end{equation}
This statement \emph{is} true (as we shall see in a moment), but will never be directly useful to us.
\end{remark}
\begin{proof}
Let us fix some constant $\vec{s}_{\ref{intrinsic_regular_product}}(C)$, small enough to satisfy all the constraints that will appear in the course of the proof. Let $(g, h)$ be a pair of maps satisfying the hypotheses.

Our strategy is to apply Proposition~\ref{proximal_product} to every pair $(\rho_i(g), \rho_i(h))$ with $i \in \Pi \setminus \Pi_X$. Let us fix such an~$i$, and check the hypotheses of Proposition~\ref{proximal_product}:
\begin{itemize}
\item By Proposition~\ref{type_X0_to_proximality}~\ref{X_regular_to_proximal}, the maps $\rho_i(g)$ and~$\rho_i(h)$ are proximal.
\item By Proposition~\ref{C-non-deg-in-Vi}, the pair $(\rho_i(g)), (\rho_i(h))$ is $C'$-non-degenerate (where $C'$ depends only on~$C$).
\item If we choose $\vec{s}_{\ref{intrinsic_regular_product}}(C) \leq \vec{s}_{\ref{type_X0_to_proximality}}(C)$, it follows by Proposition~\ref{type_X0_to_proximality}~\ref{X_to_proximal_contraction_strength} that $\tilde{s}(\rho_i(g)) \lesssim_C \vec{s}_X(g)$ and~$\tilde{s}(\rho_i(h)) \lesssim_C \vec{s}_X(h)$. If we choose $\vec{s}_{\ref{intrinsic_regular_product}}(C)$ sufficiently small, then all the maps $\rho_i(g)$ and $\rho_i(h)$ are $\tilde{s}_{\ref{proximal_product}}(C')$-contracting, \ie sufficiently contracting to apply Proposition~\ref{proximal_product}.
\end{itemize}
Now Proposition~\ref{proximal_product} gives us:
\begin{itemize}
\item that for every~$i \in \Pi \setminus \Pi_X$, the map $\rho_i(gh)$ is proximal. Hence by Proposition~\ref{type_X0_to_proximality}~\ref{X_regular_to_proximal}, the element~$gh$ is $X$-regular.
\item that moreover for every~$i \in \Pi \setminus \Pi_X$, we have
\begin{equation}
\label{eq:esrhoigh}
\alpha(E^s_{\rho_i(gh)}, E^s_{\rho_i(g)}) \lesssim_C \tilde{s}(\rho_i(g)) \lesssim_C \vec{s}_X(g).
\end{equation}
Now let us endow the product~$\prod_{i \in \Pi \setminus \Pi_X} \mathbb{P}(V_i)$ with the product distance~$\delta$ given by
\begin{equation}
\delta(y, y') := \max_{i \in \Pi \setminus \Pi_X} \alpha(y_i, y'_i);
\end{equation}
then all the inequalities~\eqref{eq:esrhoigh} may be combined together to yield
\begin{equation}
\delta \left( \Phi_X^s \left( y^{X, +}_{gh} \right),\; \Phi_X^s \left( y^{X, +}_{g} \right) \right) \lesssim_C \vec{s}_X(g),
\end{equation}
where~$\Phi_X^s$ is the map introduced in Lemma~\ref{PS_to_PVi_and_PVistar}. Since~$\Phi_X^s$ is a smooth embedding of the compact manifold~$G/P^+_X$, it is in particular a bilipschitz map onto its image. Hence we also have
\[\delta \left( y^{X, +}_{gh}, y^{X, +}_g \right) \lesssim_C \vec{s}_X(g).\]
\end{itemize}
This yields the first line of the conclusion. To get the second line, simply note that by Proposition~\ref{intrinsic_inverse}, if we replace~$X$ by~$-w_0(X)$, the pair $(h^{-1}, g^{-1})$ satisfies the same hypotheses as the pair~$(g, h)$. Applying what we did above to this pair, we get
\[\delta \left( \iota^-_X(y^{X, -}_{gh}),\; \iota^-_X(y^{X, -}_{h}) \right) \lesssim_C \vec{s}_X(h);\]
by Remark~\ref{iota_isometry}, the conclusion follows.
\end{proof}

\begin{proposition}
\label{jordan_additivity}
For every $C \geq 1$, there are positive constants $\vec{s}_{\ref{jordan_additivity}}(C)$ and $k_{\ref{jordan_additivity}}(C)$ with the following property. Take any $C$-non-degenerate pair $(g, h)$ of $X$-regular elements of~$G$ such that $\vec{s}_X(g) \leq \vec{s}_{\ref{jordan_additivity}}(C)$ and $\vec{s}_X(h) \leq \vec{s}_{\ref{jordan_additivity}}(C)$. Then we have:
\begin{hypothenum}
\item $\forall i \in \mathrlap{\Pi,}\qquad\qquad\quad
\varpi_i \left( \jordan(gh) - \cartan(g) - \cartan(h) \right) \leq 0$;
\item $\forall i \in \mathrlap{\Pi \setminus \Pi_{X_0},}\qquad\qquad\quad
\varpi_i \left( \jordan(gh) - \cartan(g) - \cartan(h) \right) \geq - k_{\ref{jordan_additivity}}(C)$.
\end{hypothenum}
\end{proposition}
\begin{proof}
The proof is completely analogous to the proof of Proposition~6.11 in~\cite{Smi16}.
\end{proof}
See Figure~\ref{fig:trapezoid_picture} for a picture (borrowed from~\cite{Smi16}) explaining both this proposition and the corollary below.
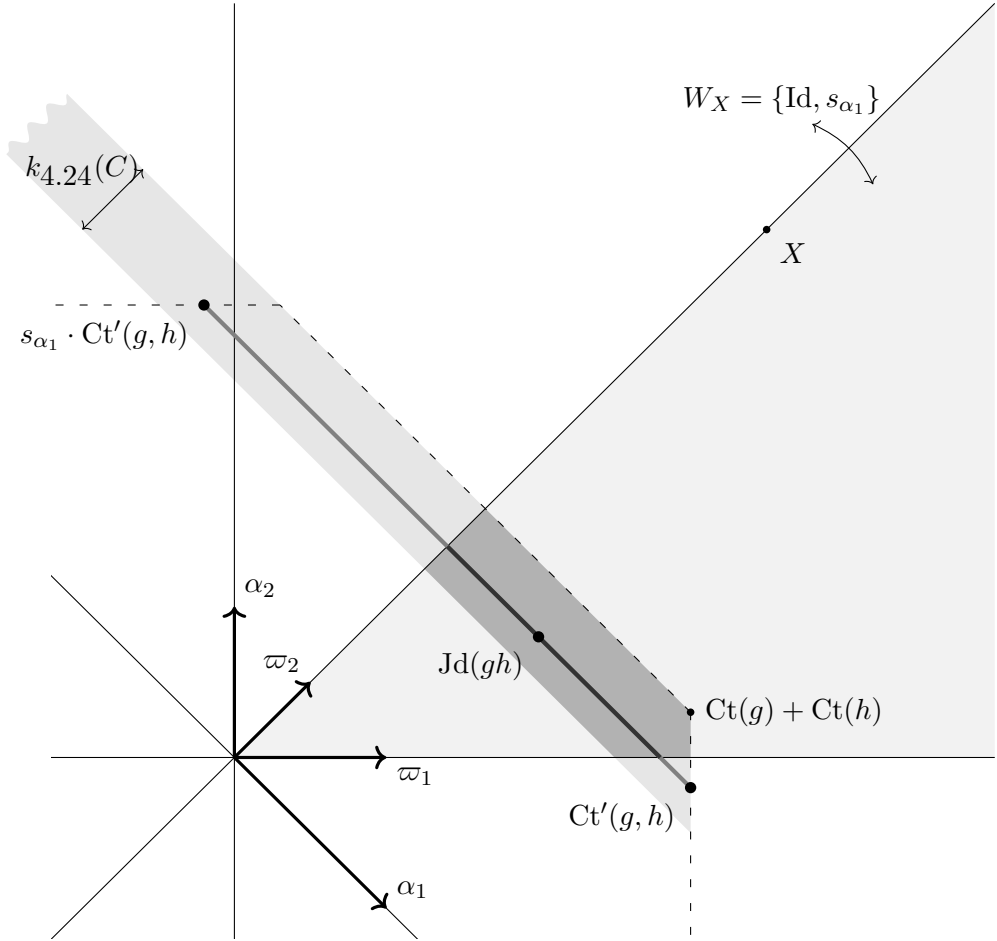
\begin{figure}
\[\begin{tikzpicture}[scale=2]
\fill[black!5!white] (0,0) -- (5,0) -- (5,5) -- (0,0);
\fill[black!10!white] (3,0.3) -- (-1.1,4.4) decorate[decoration=snake] {-- (-1.5,4)} -- (3,-0.5) -- (3,0.3);
\fill[black!30!white] (3,0.3) -- (1.65,1.65) -- (1.25,1.25) -- (2.5,0) -- (3,0) -- (3,0.3);
\draw[<->] (-0.6,3.9) -- (-1,3.5);
\node at (-1,3.9) {$k_{\ref{jordan_additivity}}(C)$};
\begin{scope}
  \clip (-1.2, -1.2) rectangle (5, 5);
  \draw (-5, 0) -- (5, 0);
  \draw (-5, -5) -- (5, 5);
  \draw (0, -5) -- (0, 5);
  \draw (5, -5) -- (-5, 5);
  \draw[loosely dashed] (3, 0.3) -- (3, -5);
  \draw[loosely dashed] (0.3, 3) -- (-5, 3);
\end{scope}
\draw[loosely dashed] (3, 0.3) -- (0.3, 3);
\draw[ultra thick, black!50!white] (3, -0.2) -- (-0.2, 3);
\draw[ultra thick, black!80!white] (2.8, 0) -- (1.4, 1.4);
\node[circle, fill, inner sep=1pt, label=right:$\cartan(g) + \cartan(h)$] at (3,0.3) {};
\node[circle, fill, inner sep=1.5pt, label=below left:${\cartan'(g,h)}$] at (3,-0.2) {};
\node[circle, fill, inner sep=1.5pt, label=below left:$\jordan(gh)$] at (2,0.8) {};
\node[circle, fill, inner sep=1.5pt, label=below left:${s_{\alpha_1} \cdot \cartan'(g,h)}$] at (-0.2,3) {};
\node[circle, fill, inner sep=1pt, label=below right:$X$] at (3.5,3.5) {};
\draw[<->] (3.8,4.2) to[bend left = 20] (4.2,3.8);
\node[above] at (3.6,4.2) {$W_X = \{\Id, s_{\alpha_1}\}$};
\draw[very thick, ->] (0,0) -- (0,1);
\node[above right] at (0,1) {$\alpha_2$};
\draw[very thick, ->] (0,0) -- (0.5,0.5);
\node[above left] at (0.5,0.5) {$\varpi_2$};
\draw[very thick, ->] (0,0) -- (1,0);
\node[below right] at (1,0) {$\varpi_1$};
\draw[very thick, ->] (0,0) -- (1,-1);
\node[above right] at (1,-1) {$\alpha_1$};
\end{tikzpicture}\]
\caption{This picture represents the situation for~$G = \SO^+(3,2)$ acting on~$\mathbb{R}^5$, and $X$~chosen such that $\Pi_X = \{\alpha_1\}$ (or~$\{1\}$ with the usual abuse of notations). This choice of~$X$ is not random: it satisfies the conditions that will be required starting from Section~\ref{sec:choice}. (This also corresponds to Example~3.7.3 in~\cite{Smi16}.) \\
\hspace*{15pt}The group~$W_X$ is then generated by the single reflection~$s_{\alpha_1}$. Proposition~\ref{jordan_additivity} states that $\jordan(gh)$ lies in the shaded trapezoid. Corollary~\ref{jordan_additivity_reformulation} states that it lies on the thick line segment. In any case it lies by definition in the dominant open Weyl chamber~$\mathfrak{a}^{++}$ (the shaded sector).
}
\label{fig:trapezoid_picture}
\end{figure}

\needspace{\baselineskip}
Let us also give a more palatable (though slightly weaker) reformulation.
\begin{corollary}
\label{jordan_additivity_reformulation}
For every $C \geq 1$, there exists a positive constant $k_{\ref{jordan_additivity_reformulation}}(C)$ with the following property. For any pair $(g, h)$ satisfying the hypotheses of the Proposition, we have
\begin{equation}
\label{eq:cartan_in_jordan_convex_hull}
\jordan(gh) \in \Conv \Big( W_X \cdot \cartan'(g, h) \Big),
\end{equation}
where $\Conv$ denotes the convex hull and $\cartan'(g, h)$~is some vector in~$\mathfrak{a}$ satisfying
\begin{equation}
\label{eq:ct'gh_ctg_cth}
\|\cartan'(g, h) - \cartan(g) - \cartan(h)\| \leq k_{\ref{jordan_additivity_reformulation}}(C).
\end{equation}
\end{corollary}
\begin{proof}
The proof is the same as the proof of Corollary~6.13 in~\cite{Smi16}.
\end{proof}
Note that we do \emph{not} require that the vector~$\cartan'(g, h)$ lie in the closed dominant Weyl chamber~$\mathfrak{a}^{+}$ (even though in practice, it is very close to the vector~$\cartan(g) + \cartan(h)$, which does).
\begin{remark}
At first sight, one might think that by putting together Lemma 4.1 and~4.5.2 in~\cite{Ben97}, we recover this result and even something stronger. However this is not the case: in fact we recover only the particular case when $\Pi_X = \emptyset$. See Remark~6.16 in~\cite{Smi16} for an explanation.
\end{remark}
\begin{remark}~
\begin{itemize}
\item Though we shall not use it, an interesting particular case is $g = h$. We then obviously have~$\jordan(gh) = 2\jordan(g)$ and~$\cartan(g) + \cartan(h) = 2\cartan(g)$, so the statement simply reduces to a relationship between the Jordan and Cartan projections on~$g$.
\item The Proposition, and its Corollary, also hold if we replace $\jordan(gh)$ by~$\cartan(gh)$. This immediately implies~\eqref{eq:linear_contraction_strength_additive}.
\end{itemize}
(Compare this to Remark~\ref{comparison_proximal_with_schema}, which involved a similar construction in the proximal case.)
\end{remark}

\section{Choice of a reference Jordan projection~$X_0$}
\label{sec:choice}

For the remainder of the paper, we fix $\rho$ an irreducible representation of~$G$ on a finite-dimensional real vector space~$V$; we also require~$\rho$ to satisfy the following condition:
\begin{assumption}
\label{zero_is_a_weight}
The form~$0$ is a restricted weight of~$\rho$:
\[\dim V^0 > 0.\]
\end{assumption}
However we do \emph{not} require, for the moment, the representation~$\rho$ to satisfy all of the hypotheses of the Main Theorem. We will only need them at the very end of the paper.
\begin{remark}
\label{radical_characterization}
By Proposition~\ref{convex_hull}, this is the case if and only if the highest restricted weight of~$\rho$ is a $\mathbb{Z}$-linear combination of restricted roots.
\end{remark}
\begin{remark}
We lose no generality in assuming this property, because it comes as a consequence of condition~\ref{itm:main_condition}\ref{itm:fixed_by_l} of the Main Theorem: indeed, any nonzero vector fixed by~$L$ is in particular fixed by~$A \subset L$, which means that it belongs to the zero restricted weight space. Conversely, without this assumption, the conclusion of the Main Theorem is certain to fail (see Remark~3.5 in~\cite{Smi16}).
\end{remark}

We call $\Omega$ the set of restricted weights of~$\rho$. For any $X \in \mathfrak{a}$, we call~$\Omega^\subg_X$ (resp. $\Omega^\subl_X$, $\Omega^\sube_X$, $\Omega^\subge_X$, $\Omega^\suble_X$) the set of all restricted weights of~$\rho$ that take a positive (resp. negative, zero, nonnegative, nonpositive) value on~$X$:
\begin{align*}
\Omega^\subge_X &:= \setsuch{\lambda \in \Omega}{\lambda(X) \geq 0}
   & \Omega^\subg_X &:= \setsuch{\lambda \in \Omega}{\lambda(X) > 0} \\
\Omega^\suble_X &:= \setsuch{\lambda \in \Omega}{\lambda(X) \leq 0}
   & \Omega^\subl_X &:= \setsuch{\lambda \in \Omega}{\lambda(X) < 0} \\
\Omega^\sube_X &:= \setsuch{\lambda \in \Omega}{\lambda(X) = 0}. & &
\end{align*}
The goal of this section is to study these sets, and to choose a vector~$X_0 \in \mathfrak{a}^+$ for which the corresponding sets have some nice properties. This generalizes Section~3 in~\cite{Smi16}.

In fact, these sets are the only property of~$X_0$ that matters for us; in other terms, what we really care about is the class of~$X_0$ with respect to the following equivalence relation:

\begin{definition}
\label{type_definition}
We say that $X$ and~$Y$ have \emph{the same type} if
\begin{equation}
\label{eq:same_type_def}
\begin{cases}
\Omega^\subg_Y = \Omega^\subg_X \\
\Omega^\subl_Y = \Omega^\subl_X.
\end{cases}
\end{equation}
Obviously this implies that the spaces $\Omega^\subge$, $\Omega^\suble$ and~$\Omega^\sube$ coincide as well for~$X$ and~$Y$. This is an equivalence relation, which partitions~$\mathfrak{a}$ into finitely many equivalence classes.
\end{definition}

%
%

\begin{remark}
Every such equivalence class is obviously a convex cone; taken together, the equivalence classes decompose~$\mathfrak{a}$ into a cell complex.
\end{remark}

\begin{example}
If $\rho$ is the adjoint representation, two \emph{dominant} vectors $X, Y \in \mathfrak{a}^+$ have the same type if and only if $\Pi_X = \Pi_Y$. There is only one generic type, corresponding to~$\mathfrak{a}^{++}$.

See Example~3.7.3 in~\cite{Smi16} for more details, and two other examples with pictures.
\end{example}

The motivation for the study of these five sets~$\Omega^?_X$ is that they allow us to introduce some ``reference dynamical spaces'' (see Subsection~\ref{sec:reference_dynamical_spaces}).

In Subsections \ref{sec:generically_symmetric} and~\ref{sec:extreme}, we define two properties that we want~$X_0$ to satisfy. Subsection~\ref{sec:simplifying_assumptions} basically consists of examples, and may be safely skipped.

\subsection{Generically symmetric vectors}
\label{sec:generically_symmetric}
We start by defining the property of being ``generically symmetric'', which generalizes generic and symmetric vectors as defined in Subsections~3.1 and~3.3 of~\cite{Smi16}. One of the goals is to ensure that the set~$\Omega^\sube_X$ is as small as possible.

Here is the first attempt: we say that an element $X \in \mathfrak{a}$ is \emph{generic} if
\[\Omega^\sube_X = \{0\}.\]
\begin{remark}
This is indeed the generic case: given Assumption~\ref{zero_is_a_weight}, this happens as soon as~$X$ avoids a finite collection of hyperplanes, namely the kernels of all nonzero restricted weights of~$\rho$. In fact, a vector is generic if and only if its equivalence class is open. If $X$~is generic, its equivalence class is just the connected component containing~$X$ in the set of generic vectors; otherwise, its equivalence class is always contained in some proper vector subspace of~$\mathfrak{a}$.
\end{remark}

Since we want to construct a \emph{group}~$\Gamma$ with certain properties, it must in particular be stable by inverse. The identity~\eqref{eq:inverse_jordan_projection} encourages us to examine the action of~$-w_0$ on~$\mathfrak{a}$.
\begin{definition}
We say that an element $X \in \mathfrak{a}$ is \emph{symmetric} if it is invariant by $-w_0$:
\[-w_0(X) = X.\]
\end{definition}
Ideally, we would like our ``reference'' vector~$X_0$ to be both symmetric and generic. This is always possible in the non-swinging case, so this is what we required in~\cite{Smi16}. However in general, this is not always possible: indeed, every restricted weight that happens to be invariant by~$w_0$ necessarily vanishes on every symmetric vector. Let
\[\Omega^{w_0} := \setsuch{\lambda \in \Omega}{w_0 \lambda = \lambda}\]
be the set of those restricted weights.
\begin{definition}
We say that an element $X \in \mathfrak{a}$ is \emph{generically symmetric} if it is symmetric and we have
\[\Omega^\sube_X = \Omega^{w_0}.\]
\end{definition}
In other terms, an element is generically symmetric if it is ``as generic as possible'' while still being symmetric.

\subsection{Extreme vectors}
\label{sec:extreme}
Besides $W_X$, we are also interested in the group
\begin{equation}
W_{\rho, X} := \setsuch{w \in W}{w X \text{ has the same type as } X},
\end{equation}
which is the stabilizer of~$X$ ``up to type''. It obviously contains~$W_X$. The goal of this subsection is to show that in every equivalence class, we can actually choose~$X$ in such a way that both groups coincide. This generalizes Subsection~3.5 in~\cite{Smi16}.
\begin{example}
In Example~3.7.3 in~\cite{Smi16} ($G = \SO^+(3,2)$ acting on~$V = \mathbb{R}^5$), the group $W_{\rho, X}$ corresponding to any generic~$X$ is a two-element group. If we take~$X$ to be generic not only with respect to~$\rho$ but also with respect to the adjoint representation (in other terms if $X$~is in an open Weyl chamber), then the group~$W_X$ is trivial. If however we take as~$X$ any element of the diagonal wall of the Weyl chamber, we have indeed $W_{X} = W_{\rho, X}$.
\end{example}
\begin{definition}
We call an element $X \in \mathfrak{a}^{+}$ \emph{extreme} if $W_X = W_{\rho, X}$, \ie if it satisfies the following property:
\[\forall w \in W,\quad w X \text{ has the same type as } X \iff w X = X.\]
\end{definition}
\begin{remark}
Here is an equivalent definition which is possibly more enlightening. It is possible to show that a vector~$X \in \mathfrak{a}^+$ is extreme if and only if it lies in every wall of the Weyl chamber that contains at least one vector of the same type as~$X$. In other words, a vector is extreme if it is in the ``furthest possible corner'' of its equivalence class in the Weyl chamber~$\mathfrak{a}^+$.

As this last statement will never be used in this paper, we have left out its proof.
\end{remark}
\begin{proposition}
\label{change_of_X_0}
For every generically symmetric $X \in \mathfrak{a}^{+}$, there exists a generically symmetric $X' \in \mathfrak{a}^{+}$ that has the same type as~$X$ and that is extreme.
\end{proposition}
This is a straightforward generalization of Proposition~3.19 in~\cite{Smi16}. The proof is similar.
\begin{proof}
To construct an element that has the same type as~$X$ but has the whole group~$W_{\rho, X}$ as stabilizer, we simply average over the action of this group: we set
\begin{equation}
X' = \sum_{w \in W_{\rho, X}} w X.
\end{equation}
(As multiplication by positive scalars does not change anything, we have written it as a sum rather than an average for ease of manipulation.) Let us check that this~$X'$ has the required properties.
\begin{itemize}
\item Let us show that $X'$~is still symmetric. Since $w_0$~belongs to the Weyl group, it induces a permutation on~$\Omega$, hence we have:
\begin{equation}
w_0 \Omega^\subg_X = \Omega^\subg_{w_0 X} = \Omega^\subg_{-X} = \Omega^\subl_X,
\end{equation}
so that $w_0$ swaps the sets $\Omega^\subg_X$ and~$\Omega^\subl_X$. Now by definition we have
\begin{equation}
W_{\rho, X} = \Stab_W(\Omega^\subg_X) \cap \Stab_W(\Omega^\subl_X),
\end{equation}
hence $w_0$~normalizes~$W_{\rho, X}$. Obviously the map~$X \mapsto -X$ commutes with everything, so $-w_0$ also normalizes~$W_{\rho, X}$. We conclude that
\begin{align}
-w_0(X') &= \sum_{w \in W_{\rho, X}} -w_0 (w (X)) \nonumber \\
         &= \sum_{w' \in W_{\rho, X}} w' (-w_0 (X)) \nonumber \\
         &= X'.
\end{align}
\item By definition every~$w X$ for~$w \in W_{\rho, X}$ has the same type as~$X$; since the equivalence class of~$X$ is a convex cone, their sum~$X'$ also has the same type as~$X$.
\item In particular we have $\Omega^\sube_{X'} = \Omega^\sube_X = \Omega^{w_0}$, so $X'$~is still generically symmetric.
\item By construction whenever $w X$~has the same type as~$X$, we have~$w X' = X'$; conversely if $w$~fixes $X'$, then $w X$~has the same type as~$w X' = X'$ which has the same type as~$X$. So $X'$~is extreme.
\item It remains to show that $X' \in \mathfrak{a}^+$, \ie that for every $\alpha \in \Pi$, we have $\alpha(X') \geq 0$:
\begin{itemize}
\item If $s_\alpha X' = X'$, then obviously $\alpha(X') = 0$.
\item Otherwise, since~$X'$ is extreme, it follows that $s_\alpha X'$ does not even have the same type as~$X'$. This means that there exists a restricted weight~$\lambda$ of~$\rho$ such that
\begin{equation}
\label{eq:transgressing_restr_weight_raw}
\lambda(X') \geq 0 \quad\text{ and }\quad s_\alpha(\lambda)(X') \leq 0,
\end{equation}
and (at least) one of the two inequalities is strict. In particular we have
\begin{equation}
\label{eq:transgressing_restr_weight}
\Big(\lambda - s_\alpha(\lambda)\Big)(X') \neq 0.
\end{equation}
Since $\lambda - s_\alpha(\lambda)$~is by definition a multiple of~$\alpha$, it follows that
\[\alpha(X') \neq 0.\]
Now the same reasoning applies to any vector of the same type as~$X$; hence $\alpha$ never vanishes on the equivalence class of~$X$. Since by hypothesis $\alpha(X) \geq 0$ and since the equivalence class is connected, we conclude that
\[\alpha(X') > 0. \qedhere\]
\end{itemize}
\end{itemize}
\end{proof}

Note that in practice, the set~$\Pi_X$ for an extreme, generically symmetric~$X$ can take only a very limited number of values: see Remark~3.21 in~\cite{Smi16}.

\subsection{Simplifying assumptions}
\label{sec:simplifying_assumptions}

In this subsection, we discuss how the constructions of this paper may simplify in some particular cases. None of the following results are logically necessary for the proof of the Main Theorem, which is why we do not provide their proofs. However, they can be helpful for a reader who is only interested in one particular representation (which is likely to fit at least one of the cases outlined below).

\begin{definition}
\label{simpifying_assumptions_def}
We say that the representation~$\rho$ is:
\begin{enumerate}
\item \label{itm:limited_def} \emph{limited} if every restricted weight is a multiple of a restricted root:
\[\Omega \subset \mathbb{Z} \cdot \Sigma := \setsuch{n \alpha}{n \in \mathbb{Z}, \alpha \in \Sigma};\]
\item \label{itm:abundant_def} \emph{abundant} if every restricted root is a restricted weight:
\[\Omega \supset \Sigma \cup \{0\};\]
\item \label{itm:awkward_def} \emph{awkward} if it is neither limited nor abundant;
\item \label{itm:non-swinging_def} \emph{non-swinging} if $\Omega^{w_0} = \{0\}$.
\end{enumerate}
\end{definition}

\begin{example}~
\label{simplifying_assumptions_example}
\begin{enumerate}
\item \label{itm:adjoint_limited_abundant} The adjoint representation is both limited and abundant, and always non-swinging.
\item \label{itm:standard_limited} The standard representation of~$\SO(p+1,p)$ on~$\mathbb{R}^{2p+1}$ is limited, and non-swinging.
\item \label{itm:abundant_classification} If $G$ is simple, it seems that all but finitely many representations are abundant. If additionally its restricted root system has a simply-laced diagram, then Lemma~\ref{roots_contained_in_weights} says that all representations except the trivial one are abundant.
\item \label{itm:awkward_classification} Among simple groups, it seems that awkward representations occur only when the restricted root system is of type~$C_n$ or~$BC_n$, and only for $n$~at least equal to~$4$. For non-simple groups, this phenomenon is more common. See Example~\ref{awkward} for specific examples.
\item \label{itm:swinging_classification} Swinging representations occur only when $-w_0$ is a nontrivial automorphism of the Dynkin diagram. Among simple groups, this happens only if the restricted root system is of type $A_n$ ($n \geq 2$), $D_{2n+1}$ or~$E_6$. The bad news is that for these groups, \emph{most} representations (all but finitely many) are swinging. The simplest example is $SL_3(\mathbb{R})$ acting on~$S^3 \mathbb{R}^3$ (see Example~3.9 in~\cite{Smi16}).
\item \label{itm:swinging_awkward} Thus no representation of a \emph{simple} group is swinging and awkward at the same time. (However this \emph{may} happen for a non-simple group $G_1 \times G_2$: simply take the tensor product of a swinging representation $\rho_1$ of~$G_1$ and an awkward representation $\rho_2$ of~$G_2$.)
\end{enumerate}
\end{example}

All the subsequent constructions rely on the choice of a generically symmetric vector~$X_0 \in \mathfrak{a}^+$ that is extreme; actually, only the type of~$X_0$ matters. Here is what we can say about the choice of~$X_0$ (up to type) in these particular cases:

\begin{remark}~
\label{X0_examples}
\begin{hypothenum}
\item \label{itm:limited_implications} In a limited representation, there is only one type of generically symmetric vector. So we can ignore the dependence on~$X_0$.
\item \label{itm:abundant_implications} In an abundant representation, every generically symmetric vector lies in particular in~$\mathfrak{a}^{++}$. In other terms we have $\Pi_{X_0} = \emptyset$, and we do not need the theory of non-minimal parabolic subgroups (as developed in section~\ref{sec:parabolics}).
\item \label{itm:non_awkward_implications} In both cases, we get that the type of~$X_0$ \emph{with respect to the adjoint representation} (\ie the subset~$\Pi_{X_0}$) does not depend on the choice of~$X_0$. In fact, in every non-awkward representation, for every generically symmetric and extreme~$X_0 \in \mathfrak{a}^+$, we have the identity
\begin{equation}
\label{eq:pi_X0_non-awkward}
\Pi_{X_0} = \setsuch{\alpha \in \Pi}{\alpha \not\in \Omega}.
\end{equation}
\item \label{itm:awkward_implications} Note that one of the inclusions, namely 
\[\Pi_{X_0} \subset \setsuch{\alpha \in \Pi}{\alpha \not\in \Omega},\]
is obvious and holds in every representation. The other inclusion however may fail in awkward representations, and the value of~$\Pi_{X_0}$ may then depend on the choice of~$X_0$: see Example~\ref{awkward}.
\item \label{itm:non_swinging_implications} In a non-swinging representation, clearly the vector $X_0$~is generically symmetric if and only if it is generic and also symmetric.
\end{hypothenum}
\end{remark}

\section{Constructions related to~$X_0$}
\label{sec:X_0_constructions}

\begin{definition}
For the remainder of the paper, we fix some vector~$X_0$ in the closed dominant Weyl chamber~$\mathfrak{a}^{+}$ that is generically symmetric and extreme.
\end{definition}

In this section, we introduce some preliminary constructions associated to this vector~$X_0$.

In Subsection~\ref{sec:reference_dynamical_spaces}, we introduce the ``reference dynamical spaces'' associated to~$X_0$, and find their stabilizers in~$G$. This generalizes Subsection~4.1 in~\cite{Smi16}.

Subsection~\ref{sec:affine} consists entirely of definitions. We introduce some elementary formalism that expresses affine spaces in terms of vector spaces, and use it to define the affine reference dynamical spaces. We basically repeat Subsection~4.2 from~\cite{Smi16}.

In Subsection~\ref{sec:regularity}, we try to understand what ``regularity'' means in the affine context. We introduce \emph{three} different notions of regularity, and establish the relationships between them. In \cite{Smi16}, we used a simpler notion (see Definition~4.14 in~\cite{Smi16}), to which two of the three new notions reduce in the non-swinging case (see Example~\ref{jordan_properties_discussion}.\ref{itm:jord_prop_non-swing} below).

Subsection~\ref{sec:counterexamples} consists entirely of examples. It contains counterexamples to help understand why some statements of the previous subsection cannot be made stronger.

\subsection{Reference dynamical spaces}
\label{sec:reference_dynamical_spaces}

\begin{definition}
\label{reference_dynamical_spaces}
We define the following subspaces of~$V$:
\begin{itemize}
\item $V^\subg_0 := \bigoplus_{\lambda(X_0) > 0} V^\lambda$,\quad the \emph{reference expanding space};
\item $V^\subl_0 := \bigoplus_{\lambda(X_0) < 0} V^\lambda$,\quad the \emph{reference contracting space};
\item $V^\sube_0 := \bigoplus_{\lambda(X_0) = 0} V^\lambda$,\quad the \emph{reference neutral space};
\item $V^\subge_0 := \bigoplus_{\lambda(X_0) \geq 0} V^\lambda$,\quad the \emph{reference noncontracting space};
\item $V^\suble_0 := \bigoplus_{\lambda(X_0) \leq 0} V^\lambda$,\quad the \emph{reference nonexpanding space}.
\end{itemize}
In other terms, $V^\subge_0$~is the direct sum of all restricted weight spaces corresponding to weights in~$\Omega^\subge_{X_0}$, and similarly for the other spaces.
\end{definition}
These are precisely the dynamical spaces associated to the map~$\exp(X_0)$ (acting on~$V$ by~$\rho$), as defined in section~4.3 of~\cite{Smi16}. See Example~4.3 in~\cite{Smi16} for the case of the adjoint representation and of $G = \SO^+(p,q)$ acting on the standard representation.
\begin{remark}
Note that by Assumption~\ref{zero_is_a_weight}, zero \emph{is} a restricted weight, so the space $V^\sube_0$ is always nontrivial.
\end{remark}
Let us now determine the stabilizers in~$G$ of these subspaces.
\begin{proposition}
\label{stabg=stabge} We have:
\begin{hypothenum}
\item $\Stab_G(V^\subge_0) = \Stab_G(V^\subg_0) = P_{X_0}^+$.
\item $\Stab_G(V^\suble_0) = \Stab_G(V^\subl_0) = P_{X_0}^-$.
\end{hypothenum}
\end{proposition}

This generalizes Proposition~4.4 in~\cite{Smi16}; but now the proof is somewhat more involved.

\begin{proof}~
By Lemma~\ref{generalized_Bruhat} and Corollary~\ref{parabolic_Bruhat_decomposition}, it is enough to show that
\begin{equation}
\Stab_W(\Omega^\subge_{X_0}) = \Stab_W(\Omega^\subg_{X_0}) = \Stab_W(\Omega^\suble_{X_0}) = \Stab_W(\Omega^\subl_{X_0}) = W_{X_0}.
\end{equation}
Indeed since $X_0 \in \mathfrak{a}^+$, clearly $\Omega^\subge_{X_0}$ and~$\Omega^\subg_{X_0}$ are top-subsets and $\Omega^\suble_{X_0}$ and~$\Omega^\subl_{X_0}$ are bottom-subsets of~$\Omega$.

Now obviously any subset of~$\Omega$ and its complement always have the same stabilizer by~$W$ (since $\Omega$ is stable by~$W$). On the other hand, we have
\begin{align*}
W_{X_0} &= W_{\rho, X_0}
  &\text{since $X_0$ is extreme} \\
        &= \Stab_W(\Omega^\subg_{X_0}) \cap \Stab_W(\Omega^\subl_{X_0})
  &\text{by definition.}
\end{align*}
So it is sufficient to show that $\Omega^\subg_{X_0}$ and $\Omega^\subl_{X_0}$, or equivalently $\Omega^\subg_{X_0}$ and~$\Omega^\subge_{X_0}$, have the same stabilizer in~$W$. This is a consequence of Lemma~\ref{Weyl_semi_stabilizer} below.
\end{proof}

\begin{lemma}
\label{Weyl_semi_stabilizer}
Every element $w$~of the restricted Weyl group~$W$ such that
\begin{equation}
\label{eq:weyl_semi_stabilizer}
w\Omega^\subg_{X_0} \subset \Omega^\subge_{X_0}
\end{equation}
stabilizes every set of restricted weights~$\Xi$ such that~$\Omega^\subg_{X_0} \subset \Xi \subset \Omega^\subge_{X_0}$.
\end{lemma}
In particular all such sets~$\Xi$ have the same stabilizer in~$W$: indeed every~$w \in W$ that stabilizes~$\Omega^\subg_{X_0}$, and every~$w \in W$ that stabilizes~$\Omega^\subge_{X_0}$, satisfies in particular~\eqref{eq:weyl_semi_stabilizer}.
\begin{proof}
The basic idea is as follows:
\begin{itemize}
\item If $\mathfrak{g}$ is simple and has a simply-laced Dynkin diagram, then $\Omega$ contains all the roots of~$\mathfrak{g}$ (unless the representation is trivial); so the only $w$ that satisfies this is the identity.
\item Otherwise (for simple~$\mathfrak{g}$), we always have $w_0 = -\Id$; so swinging cannot happen, and the sets $\Omega^\subg_{X_0}$ and $\Omega^\subge_{X_0}$ differ only by the zero weight.
\end{itemize}

More precisely, let us decompose~$\mathfrak{g}$ as a direct sum of three pieces $\mathfrak{g} = \mathfrak{g}_1 \oplus \mathfrak{g}_2 \oplus \mathfrak{g}_3$, with respective Cartan subspaces $\mathfrak{a}_1$, $\mathfrak{a}_2$ and~$\mathfrak{a}_3$, such that:
\begin{itemize}
\item every simple summand of~$\mathfrak{g}_1$ has a restricted root system with a nonempty, simply-laced diagram, and acts nontrivially on~$V$;
\item the restriction of~$w_0$ to~$\mathfrak{a}_2$ is~$-\Id$;
\item $\mathfrak{g}_3$~acts trivially on~$V$.
\end{itemize}
(Of course this decomposition is not unique: \eg simple summands of types $A_1$, $D_{2n}$, $E_7$ or~$E_8$ can fit into either $\mathfrak{g}_1$ or~$\mathfrak{g}_2$.)

Then we also have the decompositions:
\begin{itemize}
\item $\Sigma = \Sigma_1 \cup \Sigma_2 \cup \Sigma_3$, where $\Sigma_i$ is the restricted root system of~$\mathfrak{g}_i$;
\item $\Sigma^+ = \Sigma_1^+ \cup \Sigma_2^+ \cup \Sigma_3^+$, where every $\Sigma_i^+ := \Sigma_i \cap \Sigma^+$ is a system of positive restricted roots for~$\mathfrak{g}_i$;
\item $W = W_1 \times W_2 \times W_3$, where $W_i$ is the restricted Weyl group of~$\mathfrak{g}_i$.
\end{itemize}

We now prove the two contrasting statements \eqref{eq:stabwomegaint_23} and~\eqref{eq:omegawo_1}.
\begin{itemize}
\item On the one hand, we claim that for any $w \in W$ such that $w\Omega^\subg_{X_0} \subset \Omega^\subge_{X_0}$, we have
\begin{equation}
\label{eq:stabwomegaint_23}
w \in W_2 \times W_3.
\end{equation}
Indeed, let us decompose~$w = w_1 w_2 w_3$ with $w_i \in W_i$.

Applying Lemma~\ref{roots_contained_in_weights} to every simple summand of~$\mathfrak{g}_1$, we find that $\Sigma_1 \subset \Omega$. Since $X_0$~is generically symmetric, the set~$\Omega^\sube_{X_0} = \Omega^{w_0}$ does not intersect~$\Sigma_1$ (it does not even intersect~$\Sigma$, as by definition $w_0$ does not leave any restricted root invariant). Since $X_0$~is dominant, we deduce that
\[\Omega^\subge_{X_0} \cap \Sigma_1 = \Omega^\subg_{X_0} \cap \Sigma_1 = \Sigma_1^+.\]
It follows that
\begin{align*}
w(\Sigma_1^+) &= w(\Omega^\subg_{X_0} \cap \Sigma_1)
    & \\
              &= w(\Omega^\subg_{X_0}) \cap \Sigma_1
    &\text{since $W$ stabilizes $\Sigma_1$} \\
              &\subset \Omega^\subge_{X_0} \cap \Sigma_1
    &\text{by assumption} \\
              &= \Sigma_1^+.
    & 
\end{align*}
Thus $w$ fixes~$\Sigma_1^+$, which precisely means that $w_1 = \Id$ as required.
\item On the other hand, we claim that
\begin{equation}
\label{eq:omegawo_1}
\Omega^\sube_{X_0} = \Omega^{w_0} \subset \mathfrak{a}_1^*
\end{equation}
(the equality holds since $X_0$~is generically symmetric). Indeed, take any $\lambda \in \Omega^{w_0}$, and let us decompose it as $\lambda = \lambda_1 + \lambda_2 + \lambda_3$ with~$\lambda_i \in \mathfrak{a}_i^*$. The component~$\lambda_3$ vanishes by definition of~$\mathfrak{g}_3$. As for~$\lambda_2$, we have by definition of~$\mathfrak{g}_2$ that $w_0(\lambda_2) = -\lambda_2$, so the component~$\lambda_2$ also vanishes.
\end{itemize}
Combining \eqref{eq:stabwomegaint_23} with~\eqref{eq:omegawo_1}, we conclude that whenever an element~$w \in W$ satisfies $w\Omega^\subg_{X_0} \subset \Omega^\subge_{X_0}$, it actually fixes every element of~$\Omega^\sube_{X_0}$. Now take such a~$w$, and take any element $\xi \in \Xi$. Since $\Xi \subset \Omega^\subge_{X_0}$, we may distinguish two cases:
\begin{itemize}
\item Either $\xi \in \Omega^\sube_{X_0}$. Then it follows from the previous statement that $w(\xi) = \xi \in \Xi$.
\item Or $\xi \in \Omega^\subg_{X_0}$. Then it follows from the previous statement that $w(\xi) \not\in \Omega^\sube_{X_0}$. On the other hand we know that $w(\xi) \in w(\Omega^\subg_{X_0}) \subset \Omega^\subge_{X_0}$. Thus $w(\xi) \in \Omega^\subg_{X_0} \subset \Xi$. \qedhere
\end{itemize}
\end{proof}

\subsection{Extended affine space}
\label{sec:affine}
Let $V_{\Aff}$ be an affine space whose underlying vector space is~$V$.

\begin{definition}[Extended affine space] We choose once and for all a point~$p_0$ in~$V_{\Aff}$ which we take as an origin; we call $\mathbb{R} p_0$ the one-dimensional vector space formally generated by this point, and we set $A := V \oplus \mathbb{R} p_0$ the \emph{extended affine space} corresponding to $V$. (We hope that $A$,~the extended affine space, and $A$,~the~group corresponding to the Cartan space, occur in sufficiently different contexts that the reader will not confuse them.) Then $V_{\Aff}$ is the affine hyperplane ``at height~1'' of this space, and $V$ is the corresponding vector hyperplane:
\[
V        = V \times \{0\} \subset V \times \mathbb{R} p_0; \qquad
V_{\Aff} = V \times \{1\} \subset V \times \mathbb{R} p_0.
\]
\end{definition}
\begin{definition}[Linear and affine group]
Any affine map $g$ with linear part $\ell(g)$ and translation vector $v$, defined on $V_{\Aff}$ by
\[g: x \mapsto \ell(g)(x)+v,\]
can be extended in a unique way to a linear map defined on $A$, given by the matrix
\[\begin{pmatrix} \ell(g) & v \\ 0 & 1 \end{pmatrix}.\]

From now on, we identify the abstract group~$G$ with the group $\rho(G) \subset \GL(V)$, and the corresponding affine group~$G \ltimes V$ with a subgroup of~$\GL(A)$.
\end{definition}
\begin{definition}[Affine subspaces]
We define an \emph{extended affine subspace} of $A$ to be a vector subspace of $A$ not contained in $V$. For every $k$, there is a one-to-one correspondence between $k+1$-dimensional extended affine subspaces of $A$ and $k$-dimensional affine subspaces of $V_{\Aff}$. For any extended affine subspace of~$A$ denoted by~$A_1$ (or~$A_2$, $A'$ and so on), we denote by~$V_1$ (or~$V_2$, $V'$ and so on) the space~$A \cap V$ (which is the linear part of the corresponding affine space $A \cap V_{\Aff}$).
\end{definition}
\begin{definition}[Translations]
By abuse of terminology, elements of the normal subgroup $V \lhd G \ltimes V$ will still be called \emph{translations}, even though we shall see them mostly as endomorphisms of $A$ (so that they are formally transvections). For any vector $v \in V$, we denote by~$\tau_{v}$ the corresponding translation.
\end{definition}
\begin{definition}[Reference affine dynamical spaces]
We now give a name for (the vector extensions of) the affine subspaces of $V_{\Aff}$ parallel respectively to $V^\subge_0$, $V^\suble_0$ and $V^\sube_0$ and passing through the origin: we set
\begin{align*}
A^\subge_0 &:= V^\subge_0 \oplus \mathbb{R} p_0,\quad \text{ the \emph{reference affine noncontracting space};} \\
A^\suble_0 &:= V^\suble_0 \oplus \mathbb{R} p_0,\quad \text{ the \emph{reference affine nonexpanding space};} \\
A^\sube_0 &:= V^\sube_0 \oplus \mathbb{R} p_0,\quad \text{ the \emph{reference affine neutral space}.}
\end{align*}
These are obviously the affine dynamical spaces (in the sense of~\cite{Smi16}) corresponding to the map~$\exp(X_0)$, seen as an element of~$G \ltimes V$ by identifying~$G$ with the stabilizer of~$p_0$ in~$G \ltimes V$.

We then have the decomposition
\begin{equation}
\label{eq:affine_reference_spaces_decomposition}
A =
\lefteqn{\overbrace{\phantom{
  V^\subg_0 \oplus A^\sube_0
}}^{\displaystyle A^\subge_0}}
  V^\subg_0 \oplus \underbrace{
  A^\sube_0 \oplus V^\subl_0.
}_{\displaystyle A^\suble_0}
\end{equation}
This gives a hint for why we do not introduce the spaces ``$A^\subg_0$'' or~``$A^\subl_0$''; see Remark~4.13 in~\cite{Smi16} for a detailed explanation.
\end{definition}

\begin{definition}[Affine Jordan projection]
Finally, we extend the notion of Jordan projection to the whole group~$G \ltimes V$, by setting
\[\forall g \in G \ltimes V,\quad \jordan(g) := \jordan(\ell(g)).\]
\end{definition}

\subsection{Conditions on the Jordan projection}
\label{sec:regularity}
In this subsection, we explain what ``regularity'' means in the affine case. In fact we introduce not just one, but \emph{three} different notions of regularity of an element~$g \in G \ltimes V$ (given as conditions on its Jordan projection); we then discuss the significance of each one of these. We then prove some implications between these properties. Examples are given only at the end of this section (Example~\ref{jordan_properties_discussion}), but the reader can (and maybe even should) study them before working through the propositions. 

\begin{definition}[Affine version of regularity]
\label{strong_regularity_definition}
We say that a vector $Y \in \mathfrak{a}$ is:
\begin{hypothenum}
\item \emph{$\rho$-regular} if it is $X_0$-regular and we have $\Omega^\sube_{Y} \subset \Omega^\sube_{X_0}$, or in other terms:
\begin{equation}
\label{eq:strongly_regular_def}
\begin{cases}
\forall \alpha \in \Pi \setminus \Pi_{X_0},\quad \alpha(Y) \neq 0 \\
\forall \lambda \in \Omega \setminus \Omega^\sube_{X_0},\quad \lambda(Y) \neq 0;
\end{cases}
\end{equation}
\item \emph{asymptotically contracting along~$X_0$} if
\begin{equation}
\label{eq:asymp_contr_def}
\forall \lambda^\subl \in \Omega^\subl_{X_0},\; \forall \lambda^\subge \in \Omega^\subge_{X_0},\quad \lambda^\subl(Y) < \lambda^\subge(Y);
\end{equation}
\item \emph{compatible with~$X_0$} if\footnote{The published version of this article uses an erroneous version of the definition of compatibility with~$X_0$: instead of (6.8*), it uses the weaker property
\begin{equation}
\label{eq:compat_def}
\forall \lambda^\subl \in \Omega^\subl_{X_0},\;
\forall \lambda^\subg \in \Omega^\subg_{X_0},\quad
\lambda^\subl(Y) < 0 < \lambda^\subg(Y).
\end{equation}
and erroneously claims that it is equivalent to~(6.8*). In the current version, we have corrected this mistake.}
\begin{equation}
\tag{6.8*}
\forall \lambda^\subl \in \Omega^\subl_{X_0},\;
\forall \lambda^\sube \in \Omega^\sube_{X_0},\;
\forall \lambda^\subg \in \Omega^\subg_{X_0},\quad
\lambda^\subl(Y) < \lambda^\sube(Y) < \lambda^\subg(Y).
\end{equation}
\addtocounter{equation}{1}
\end{hypothenum}
We say that an element $g \in G \ltimes V$ has one of these three properties if $\jordan(g)$~has the corresponding property.
\end{definition}

Some logical relationships between these three properties will be established soon: see Propositions \ref{asymp_cont_vs_compat} and~\ref{asymptotically_contracting_to_regular}.
\begin{remark}
Rigorously we should talk about ``$(\rho, X_0)$-regularity'', as the definition depends on the choice of~$X_0$. However the author feels that this dependence is not significant enough to be constantly mentioned in this way. (See in particular Example~\ref{jordan_properties_discussion}.\ref{itm:jord_prop_well-beh}).
\end{remark}
All three of these new properties are ``affine analogs'' of $X_0$-regularity. However, they are useful in somewhat different contexts: 
\begin{hypothenum}
\item The purpose of assuming that an element $g \in G \ltimes V$ is $\rho$-regular is just to ensure that the affine ideal dynamical subspaces (introduced in section~\ref{sec:algebraic}) are well-defined. This is a relatively weak property: $\jordan(g)$ is merely requested to avoid a finite collection of hyperplanes.

This is the property we will use the most often, and in particular the one that makes the affine version of Schema~\ref{proposition_template} work.
\item The purpose of assuming that an element $g \in G \ltimes V$ is asymptotically contracting along~$X_0$ is roughly to ensure that $g$~acts ``with the correct dynamics'' on its ideal dynamical subspaces. For more details, see the discussion following the definition of the latter (Definition~\ref{ideal_dynamical_spaces_definition}); for a motivation of the ``asymptotically contracting'' terminology, see Proposition~\ref{contraction_strength_limit}.\ref{itm:asymptotic_explanation}.

This property is not explicitly part of the hypotheses in Schema~\ref{proposition_template}, since it is implied by contraction strength (see Proposition~\ref{contraction_strength_limit}.\ref{itm:one_to_asymp_contr}) However, we will often need it as an extra assumption in intermediate results.

\item Since we work with groups and not just semigroups, we will actually often need to assume that \emph{both $g$ and~$g^{-1}$} are asymptotically contracting, \ie that $g$ is compatible with~$X_0$.

Also note that this property is verified as soon as the Jordan projection of~$g$ points in a direction that is ``close enough'' to that of~$X_0$ (see also Proposition~\ref{paper_non_vacuous}).

\end{hypothenum}
\begin{remark}
(Removed\footnote{Remark~\ref{removed_remark} applies only to the wrong definition of compatibility with $X_0$ used in the previous version, and no longer has its place in the corrected version. We left an empty shell in order to preserve environment numbering.}.)
\label{removed_remark}
\end{remark}
\begin{proposition}
\label{asymp_cont_vs_compat}
A vector $Y \in \mathfrak{a}$ (resp. an element~$g \in G \ltimes V$) is compatible with~$X_0$ if and only if\footnote{With the wrong definition of compatibility, the ``only if'' part of this proposition was false.} both $Y$ and~$-w_0(Y)$ (resp. both $g$ and~$g^{-1}$) are asymptotically contracting along~$X_0$.
\end{proposition}
\begin{proof}
This immediately follows from the fact that $X_0$ is symmetric, from the $W$-invariance of~$\Omega$ and from \eqref{eq:inverse_jordan_projection}.
\end{proof}


\begin{proposition}
\label{asymptotically_contracting_to_regular}
Let $g \in G \ltimes V$. Then:
\begin{hypothenum}
\item if $g$~is asymptotically contracting along~$X_0$, then $g$~is $X_0$-regular;
\item if $g$~is compatible with $X_0$, then $g$~is $\rho$-regular.
\end{hypothenum}
\end{proposition}
\begin{remark}
\label{asymp_vs_regular}
As we shall see in Example~\ref{jordan_properties_discussion}.\ref{itm:jord_prop_non-swing}, if $\rho$~is non-swinging, then asymptotic contraction alone suffices to guarantee $\rho$-regularity. But in general it does not: see Example~\ref{fancy_swinging_example} for a counterexample.
\end{remark}

The proof relies on the following lemma. For the moment we only need the weak version (which generalizes Lemma~6.4 from~\cite{Smi16}); the strong version (which is new) will be useful later.
\begin{lemma}
\label{de_part_et_d_autre}
For every $i \in \Pi \setminus \Pi_{X_0}$, there exists a pair of restricted weights~$(\lambda^\subge_i, \lambda^\subl_i)$ such that $\lambda^\subge_i - \lambda^\subl_i = \alpha_i$, with:
\begin{hypothenum}
\item \label{itm:weak} (weak version) $\lambda^\subl_i \in \Omega^\subl_{X_0}$ and~$\lambda^\subge_i \in \Omega^\subge_{X_0}$;
\item \label{itm:strong} (strong version) $\lambda^\subl_i \in \Omega^\subl_{X_0}$ and~$\lambda^\subge_i \in \Omega^\subg_{X_0} \cup \{0\}$.
\end{hypothenum}
\end{lemma}
\begin{remark}
As long as $\rho$~is non-awkward, \eqref{eq:pi_X0_non-awkward} provides a ``super-strong version'' of this lemma: namely we can actually take $\lambda^\subge_i = 0$. In general this is not true: see Example~\ref{awkward}.
\end{remark}
\begin{proof}~
\begin{hypothenum}
\item Let us reintroduce the decomposition $\mathfrak{g} = \mathfrak{g}_1 \oplus \mathfrak{g}_2 \oplus \mathfrak{g}_3$ from the proof of Lemma~\ref{Weyl_semi_stabilizer}, together with the other notations that went along with it. We then have
\[\Pi = \Pi_1 \sqcup \Pi_2 \sqcup \Pi_3,\]
where $\Pi_{1, 2, 3} := \Pi \cap \Sigma_{1, 2, 3}$. Let~$i \in \Pi \setminus \Pi_{X_0}$. We distinguish three cases:
\begin{itemize}
\item The case $i \in \Pi_3$ never occurs: indeed for $i \in \Pi_3$, the symmetry~$s_{\alpha_i}$ fixes~$\Omega$ pointwise, so we have $i \in \Pi_{X_0}$.
\item If $i \in \Pi_1$, applying Lemma~\ref{roots_contained_in_weights} to the simple summand of~$\mathfrak{g}_1$ containing~$\alpha_i$, we find that $\alpha_i \in \Omega$. Then we may simply take $\lambda^\subge_i = 0$ and $\lambda^\subl_i = -\alpha_i$.
\item Finally suppose that $i \in \Pi_2$. Then by definition, we have~$s_{\alpha_i} \not\in W_{X_0}$; by Proposition~\ref{stabg=stabge}, this means that $s_{\alpha_i}$ does not stabilize~$\Omega^\subge_{X_0}$. In other terms, there exists a restricted weight~$\lambda$ of~$\rho$ such that
\begin{equation}
\lambda(X_0) \geq 0 \quad\text{ and }\quad s_{\alpha_i}(\lambda)(X_0) < 0
\end{equation}
(compare this with~\eqref{eq:transgressing_restr_weight_raw}, which was slightly weaker). Since $\lambda$ is a restricted weight, by Proposition~\ref{restr_weight_lattice}, the number
\begin{equation}
n_\lambda := \frac{\langle \lambda, \alpha_i \rangle}{2\langle \alpha_i, \alpha_i \rangle}
\end{equation}
is an integer. We have, on the one hand:
\begin{equation}
\label{eq:nalpha_positive}
n_\lambda \alpha_i (X_0) = \left( \lambda - s_{\alpha_i}(\lambda) \right) (X_0) > 0;
\end{equation}
on the other hand, $\alpha_i(X_0) \geq 0$ (because $X_0 \in \mathfrak{a}^{+}$); hence $n_\lambda$~is positive.

By Proposition~\ref{convex_hull}, every element of the sequence
\[\lambda,\; \lambda - \alpha_i,\; \ldots,\; \lambda - n_\lambda \alpha_i = s_{\alpha_i}(\lambda)\]
is a restricted weight of~$\rho$. Let~$\lambda_{\mathrm{last}}$ be the last element of this sequence that still lies in~$\Omega^\subge_{X_0}$; then by construction $\lambda_{\mathrm{last}} - \alpha_i \in \Omega^\subl_{X_0}$. Taking $\lambda^\subge_i = \lambda_{\mathrm{last}}$ and~$\lambda^\subl_i = \lambda_{\mathrm{last}} - \alpha_i$, we get the weak version.
\end{itemize}

\item The same proof works for the strong version, except in the last case ($i \in \Pi_2$) if we happen to have $\lambda_{\mathrm{last}} \not\in \Omega^\subg_{X_0}$. But then we have $\lambda_{\mathrm{last}} \in \Omega^\sube_{X_0} = \Omega^{w_0}$ since~$X_0$ is generically symmetric. In particular, by~\eqref{eq:omegawo_1}, we have $\lambda_{\mathrm{last}} \in \mathfrak{a}_1^*$.

Since $\lambda_{\mathrm{last}} - \alpha_i$ is a restricted weight of~$\rho$, by Proposition~\ref{convex_hull}, the form~$-\alpha_i$ is also a restricted weight of~$\rho$ (the latter is the average of the $W_1$-orbit of the former). By~\eqref{eq:nalpha_positive}, we actually have $-\alpha_i \in \Omega^\subl_{X_0}$. Thus we may take $\lambda^\subge_i = 0$ and~$\lambda^\subl_i = -\alpha_i$. \qedhere
\end{hypothenum}
\end{proof}

\begin{proof}[Proof of Proposition~\ref{asymptotically_contracting_to_regular}]~
\begin{hypothenum}
\item Let $i \in \Pi \setminus \Pi_{X_0}$, and let $\lambda^\subge_i$ and~$\lambda^\subl_i$ be the restricted weights constructed in Lemma~\ref{de_part_et_d_autre} (the weak version suffices). Since $g$~is asymptotically contracting along~$X_0$, we then have $\lambda^\subge_i(\jordan(g)) > \lambda^\subl_i(\jordan(g))$, hence $\alpha_i(\jordan(g)) > 0$.
\item Suppose that both $g$ and~$g^{-1}$ are asymptotically contracting along~$X_0$. By the previous point, we already know that $g$~is $X_0$-regular. It remains to check that for any $\lambda \not\in \Omega^\sube_{X_0}$, we have $\lambda(\jordan(g)) \neq 0$. We distinguish two cases:
\begin{itemize}
\item if $\lambda \in \Omega^\subl_{X_0}$, since $g$~is asymptotically contracting along~$X_0$, we have
\[\lambda(\jordan(g)) < \min_{\lambda^\subge \in \Omega^\subge_{X_0}} \lambda^\subge(\jordan(g)) \leq 0;\]
\item if $\lambda \in \Omega^\subg_{X_0}$, since $g^{-1}$~is asymptotically contracting along~$X_0$, we have
\[\lambda(\jordan(g)) > \max_{\lambda^\suble \in \Omega^\suble_{X_0}} \lambda^\suble(\jordan(g)) \geq 0.\]
\end{itemize}
Hence $g$~is $\rho$-regular. \qedhere
\end{hypothenum}
\end{proof}

Now let us prove a somewhat technical (but easy) proposition, that will be useful to us only in the final proof of the Main Theorem.

\begin{definition}
We set\footnote{Note that in this version, the definition of compatibility --- namely Definition~\ref{strong_regularity_definition}.(iii) --- has been changed.}
\begin{equation}
\label{eq:relevant_cone_def}
\tag{6.13*}
\mathfrak{a}'_{\rho, X_0}
:= \setsuch{Y \in \mathfrak{a}}{\textnormal{$Y$ is compatible with~$X_0$}}.
\end{equation}
\addtocounter{equation}{1}
\end{definition}
This set will play in this paper a role analogous to the role played in~\cite{Smi16} by the set~$\mathfrak{a}_{\rho, X_0}$ of vectors \emph{of the same type} as~$X_0$ (defined in~(3.2) in~\cite{Smi16}). In fact, in the particular cases studied in that earlier paper, both sets actually coincided (see Example~\ref{jordan_properties_discussion}.\ref{itm:jord_prop_non-swing}).
\begin{proposition}
\label{paper_non_vacuous}
The intersection $\mathfrak{a}'_{\rho, X_0} \cap \mathfrak{a}^{++}$ is a nonempty open convex cone.
\end{proposition}
\begin{proof}~
By construction $\mathfrak{a}'_{\rho, X_0}$ is an intersection of finitely many open vector half-spaces, so it is an open convex cone. It contains~$X_0$, hence it meets the closed cone~$\mathfrak{a}^+$; hence it also meets the open cone~$\mathfrak{a}^{++}$.
\end{proof}

Finally, as promised, we give some examples.

\begin{example}~
\label{jordan_properties_discussion}
\begin{enumerate}
\item \label{itm:jord_prop_limited} If the representation~$\rho$ is limited, all four properties of being $X_0$-regular, $\rho$-regular, asymptotically contracting along~$X_0$ and compatible with~$X_0$ become equivalent. This includes the standard representation of~$G = \SO(p+1, p)$ on~$\mathbb{R}^{2p+1}$: see Example 4.15.1 in~\cite{Smi16}.
\item \label{itm:jord_prop_adjoint} If the representation~$\rho$ is limited and abundant at the same time (\eg the adjoint representation), all four properties actually reduce to ordinary $\mathbb{R}$-regularity (since $X_0$ generic then means $X_0 \in \mathfrak{a}^{++}$).
\item \label{itm:jord_prop_well-beh} If the representation~$\rho$ is non-awkward (either limited or abundant), the notion of $\rho$-regularity does \emph{not} depend on the choice of~$X_0$: indeed $\Pi_{X_0}$ is then uniquely determined by \eqref{eq:pi_X0_non-awkward}, and $\Omega_{X_0}^\sube = \Omega^{w_0}$ since $X_0$~is by assumption generically symmetric.
\item \label{itm:jord_prop_ill-beh} In general (see Example~\ref{awkward}), the notion of $\rho$-regularity \emph{does} depend on the choice of~$X_0$ (so technically, it should be called ``$(\rho, X_0)$-regularity'').
\item \label{itm:jord_prop_non-swing} If the representation~$\rho$ is non-swinging, it is possible to show that its set of restricted weights~$\Omega$ is centrally symmetric (or, equivalently, invariant by~$-w_0$). (The author must however admit that he knows no better proof of this fact than by complete enumeration of non-swinging representations.) It then follows that whenever $g$ is asymptotically contracting along~$X_0$, so is $g^{-1}$. Then the conditions of asymptotic contraction along~$X_0$, compatibility with~$X_0$ and having the same type as~$X_0$ (in the sense of Definition~\ref{type_definition}) all become equivalent. This is why in~\cite{Smi16}, we worked only with the latter condition.
\item \label{itm:jord_prop_swing} In general, asymptotic contraction along~$X_0$ is a stronger condition than being of type~$X_0$. For a simple counterexample, we encourage the reader to study the case of $G = \SL_3(\mathbb{R})$ acting on~$S^3 \mathbb{R}^3$ (some pictures are given in \cite{Smi16}, Example~3.9). For a fancier counterexample, see Example~\ref{fancy_swinging_example}.
\end{enumerate}
\end{example}

\subsection{Counterexamples}
\label{sec:counterexamples}
Here we give three examples of ``pathological'' behavior, to explain why some constructions of this paper cannot be simplified in general. All three of them are fairly far-fetched. For the reader who wishes to focus on ``non-pathological'' behavior, it is probably safe to skip this subsection.

\begin{example}
\label{awkward}
Here are two examples of awkward (\ie neither limited nor abundant) representations, with an explanation of how they provide a counterexample to~\eqref{eq:pi_X0_non-awkward} and to the ``super-strong version'' of Lemma~\ref{de_part_et_d_autre}. (The first one is a development on Example~3.7.4 in~\cite{Smi16}.)
\begin{enumerate}
\item Take $G = \operatorname{PSp}_4(\mathbb{R})$ (which is a split group). Hence its restricted root system coincides with its ordinary root system, and is of type $C_4$. In the notations of \cite{Kna96}, its simple (restricted) roots are $e_1 - e_2$, $e_2 - e_3$, $e_3 - e_4$ and~$2e_4$.

Take $\rho$ to be the representation with highest weight $e_1 + e_2 + e_3 + e_4$. It has:
\begin{itemize}
\item the $16$~(restricted) weights of the form $\pm e_1 \pm e_2 \pm e_3 \pm e_4$, with multiplicity~$1$;
\item the $24$~(restricted) weights of the form $\pm e_i \pm e_j$, with multiplicity~$1$;
\item the zero (restricted) weight with multiplicity~$2$,
\end{itemize}
for a total dimension of~$42$.

For $1 \leq i \leq 3$, we have $\alpha_i = e_i - e_{i+1} \in \Omega$ and in particular~$-\alpha_i \in \Omega^\subl_{X_0}$ for any generic~$X_0$, so obviously we may take $\lambda^\subge_i = 0$. For $\alpha_4 = 2e_4$ on the other hand, we have $-\alpha_4 \not\in \Omega$ so this is no longer the case; the outcome depends on the type of~$X_0$. There are three different types of generic elements; extreme representatives of each type are given by:
\begin{enumerate}[label=(\alph*)]
\item $X_0 = (4, 2, 1, 0)$, with $\Pi_{X_0} = \{\alpha_4\}$;
\item $X_0 = (5, 3, 2, 1)$, with $\Pi_{X_0} = \emptyset$;
\item $X_0 = (4, 3, 2, 0)$, with $\Pi_{X_0} = \{\alpha_4\}$.
\end{enumerate}
In cases (a) and~(c), the removal of~$\Pi_{X_0}$ excludes~$\alpha_4$ from consideration. In case~(b) however, we have to deal with it. To wit, we then have:
\begin{align}
\Omega^\subge_{X_0} = \{&e_1 + e_2 + e_3 + e_4, \nonumber \\
                        &e_1 + e_2 + e_3 - e_4, \nonumber \\
                        &e_1 + e_2 - e_3 + e_4, \nonumber \\
                        &e_1 + e_2 - e_3 - e_4, \nonumber \\
                        &e_1 - e_2 + e_3 + e_4, \nonumber \\
                        &e_1 - e_2 + e_3 - e_4, \nonumber \\
                        &e_1 - e_2 - e_3 + e_4, \nonumber \\
                        &{-e_1} + e_2 + e_3 + e_4, \nonumber \\
                        &e_i \pm e_j \text{ for all } i < j, \nonumber \\
                        &0\},
\end{align}
and the \emph{only} possible choice is
\begin{equation}\begin{cases}
\lambda^\subge_4 = e_1 - e_2 - e_3 + e_4;\\
\lambda^\subl_4 = e_1 - e_2 - e_3 - e_4.
\end{cases}\end{equation}
\item Take $G = \SO^+(3,2) \times \SO^+(3,2)$. Its root system (both ordinary and restricted) is then of type $B_2 \times B_2$. Let us call
\begin{itemize}
\item $\pm e_1,\; \pm e_2,\; \pm e_1 \pm e_2$ the (restricted) roots of the first factor;
\item $\pm f_1,\; \pm f_2,\; \pm f_1 \pm f_2$ the (restricted) roots of the second factor.
\end{itemize}
Let us order the (restricted) roots by the lexicographical order on each factor. This gives a unique ordering of the combined root system. The simple (restricted) roots of~$G$ are then $e_1 - e_2$, $e_2$, $f_1 - f_2$ and~$f_2$.

Take $\rho$ to be the representation of~$G$ with highest weight~$e_1 + f_1$ (which corresponds to the standard action of~$G$ on~$\mathbb{R}^5 \otimes \mathbb{R}^5$). The set of its (restricted) weights is then
\begin{equation}
\Omega = \{-e_1, -e_2, 0, e_2, e_1\} + \{-f_1, -f_2, 0, f_2, f_1\}
\end{equation}
(where by $A + B$ we mean $\setsuch{a + b}{a \in A, b \in B}$, \ie the set~$\Omega$ is of cardinal~$25$). The set~$\Omega$ contains the negative simple (restricted) roots $-e_2$ and~$-f_2$, but not $-e_1 + e_2$ nor~$-f_1 + f_2$.

There are six different types of generic elements. Let us examine three of these types; the three others are obtained by exchanging~$e$ and~$f$. Using on~$\mathfrak{a}$ the coordinate system $(e_1, e_2, f_1, f_2)$, extreme representatives of these three types are given by:
\begin{enumerate}[label=(\alph*)]
\item $X_0 = (2, 2, 1, 1)$. Then $\Pi_{X_0} = \{e_1 - e_2, f_1 - f_2\}$: this is a ``nice'' case.
\item $X_0 = (3, 1, 2, 2)$. Then $\Pi_{X_0} = \{f_1 - f_2\}$, and we need to deal with~$e_1 - e_2$. The two only possibilities are to take
\begin{equation}\begin{cases}
\lambda^\subge_{e_1 - e_2} = e_1 - f_i;\\
\lambda^\subl_{e_1 - e_2} = e_2 - f_i
\end{cases}\end{equation}
for $i = 1$ or~$2$.
\item $X_0 = (4, 2, 3, 1)$. Then $\Pi_{X_0} = \emptyset$, and we need to deal with both. We have to take
\begin{equation}\begin{cases}
\lambda^\subge_{e_1 - e_2} = e_1 - f_1;\\
\lambda^\subl_{e_1 - e_2} = e_2 - f_1
\end{cases}\end{equation}
and
\begin{equation}\begin{cases}
\lambda^\subge_{f_1 - f_2} = -e_2 + f_1;\\
\lambda^\subl_{f_1 - f_2} = -e_2 + f_2.
\end{cases}\end{equation}
\end{enumerate}
\end{enumerate}
\end{example}

\begin{example}
\label{fancy_swinging_example}
Here is a counterexample to the statement of Remark~\ref{asymp_vs_regular}: a representation~$\rho$ (necessarily swinging) and a choice of~$X_0$ such that asymptotic contraction along~$X_0$ does not, on its own, imply $\rho$-regularity. (We let the reader check the details.)

Take $G = \PSL_4(\mathbb{R})$; its root system (both ordinary and restricted) is of type~$A_3$, and the involution $w_0$ maps $(e_1, e_2, e_3, e_4)$ to~$(e_4, e_3, e_2, e_1)$ (in the notations of~\cite{Kna96}, appendix~C).

Let $\rho$ be a representation of~$G$ with highest weight
\[\lambda = 5\varpi_2 + \varpi_3 = 4e_1 - e_2 - e_3 - 2e_2;\]
this is a representation of dimension~$189$, with $119$~distinct (restricted) weights. Take
\[X_0 = (10, 1, -1, -10);\]
this vector is generically symmetric and extreme with respect to~$\rho$ --- in fact, we have
\[\Omega^\sube_{X_0} = \Omega^{w_0} = \{0,\; \pm(e_1 - e_2 - e_3 + e_4)\}.\]

Then the vector
\[Y = (16, 2, -3, -15)\]
is asymptotically contracting along~$X_0$, \ie
\[\forall \lambda^\subl \in \Omega^\subl_{X_0},\; \forall \lambda^\subge \in \Omega^\subge_{X_0},\quad \lambda^\subl(Y) < \lambda^\subge(Y).\]
However the (restricted) weight~$-e_1-e_2+4e_3-2e_4$ vanishes on~$Y$, but not on~$X_0$; so $Y$ is not $\rho$-regular.
\end{example}

\section{Properties of $\rho$-regular affine maps}
\label{sec:rho_regular}

The goal of this section is to define the affine versions of the remaining properties. It generalizes Subsections 4.3--5.2 in~\cite{Smi16}, but several constructions now become considerably more complex.

In Subsection~\ref{sec:algebraic}, we define the ``ideal dynamical spaces'' associated to a $\rho$-regular map~$g \in G \ltimes V$. The data of two of them is the affine version of the ``geometry'' of~$g$, and the remaining ones can be deduced from these two. This generalizes at the same time Subsections 4.3 and~4.4 from~\cite{Smi16}, but using a different approach.

In Subsection~\ref{sec:quasi-translations}, we study the action of a $\rho$-regular map on its affine ideally neutral space~$A^\subap_g$; this action turns out to be a so-called \emph{quasi-translation}. This generalizes Subsection~4.5 from~\cite{Smi16} and Subsection~2.4 from~\cite{Smi14}, with one small difference: see Remark~\ref{quasi-translation-generalization}.

In Subsection~\ref{sec:canonical}, we introduce a groupoid of ``canonical identifications'' between all the possible affine ideally neutral spaces. We then use them to define the ``translation part'' of the asymptotic dynamics of~$g$ (the Margulis invariant). This is an almost straightforward generalization of Subsection~4.6 from~\cite{Smi16} and Subsection~2.5 from~\cite{Smi14}.

In Subsection~\ref{sec:regular_metric}, we define and study the affine versions of non-degeneracy and contraction strength. We mostly follow the second half of Subsection~5.1 from~\cite{Smi16} or of Subsection~2.6 from~\cite{Smi14}.

In Subsection~\ref{sec:affine_and_linear}, we study the relationships between affine and linear properties. This generalizes Subsection~5.2 from~\cite{Smi16} or Subsection~2.7 from~\cite{Smi14}, but now the statement is weaker and more complicated.

\subsection{Ideal dynamical spaces}
\label{sec:algebraic}

The goal of this subsection is to define the \emph{ideal dynamical spaces} associated to a $\rho$-regular (see Definition~\ref{strong_regularity_definition}) map $g \in G \ltimes V$. We start with the following particular case:


\begin{definition}
Take any element~$g \in G \ltimes V$. We may then write it as
\begin{equation}
\label{eq:sorta_affine_Jordan_decomposition}
g = \tau_v g_h g_e g_u,
\end{equation}
where $\tau_v$ is a translation by some vector~$v \in V$ and $g_h g_e g_u$ is the Jordan decomposition (see Proposition~\ref{jordan_decomposition}) of the linear part of~$g$. We say that $g$ is \emph{in canonical form} if we have:
\begin{hypothenum}
\item $g_h \in \exp(\mathfrak{a}^+)$ (\ie $g_h = \exp(\jordan(g))$);
\item $v \in V^\sube_0$.
\end{hypothenum}
\end{definition}

If $g$ is in canonical form, then we define its ideal dynamical spaces to be the reference dynamical spaces introduced above. This is especially useful when $g$ is $X_0$-regular, as shown by the following property:

\begin{proposition}
\label{reference_spaces_invariant}
If a map $g \in G \ltimes V$ is in canonical form and is $X_0$-regular, then it stabilizes all eight reference dynamical spaces, namely $A^\subge_0$, $A^\suble_0$, $A^\sube_0$, $V^\subge_0$, $V^\suble_0$, $V^\sube_0$, $V^\subg_0$ and~$V^\subl_0$.
\end{proposition}
\begin{proof}
First of all note that $\ell(g)$ commutes (by definition) with its hyperbolic part, which is (because $g$ is in canonical form) equal to $\exp(\jordan(g))$. Hence $\ell(g)$ belongs to the centralizer of~$\jordan(g)$, which is~$L_{\jordan(g)}$. Now since $g$ is $X_0$-regular, we have $L_{\jordan(g)} \subset L_{X_0}$. Finally from Proposition~\ref{stabg=stabge}, it follows that the group~$L_{X_0}$ (hence in particular $\ell(g)$) stabilizes:
\begin{itemize}
\item the spaces $V^\subge_0$ and $V^\subg_0$, as a subgroup of $P^+_{X_0}$;
\item the spaces $V^\suble_0$ and $V^\subl_0$, as a subgroup of $P^-_{X_0}$;
\item the space $V^\sube_0 = V^\subge_0 \cap V^\suble_0$, as a subgroup of both.
\end{itemize}
Now the action of the affine map~$g$ on the subspace $V \subset V \oplus \mathbb{R} p_0$ coincides with the action of its linear part~$\ell(g)$; so $g$ also stabilizes these five subspaces. Finally we know (since $g$ is in canonical form) that the vector $v$ (defined by~\eqref{eq:sorta_affine_Jordan_decomposition}) is contained in~$V^\sube_0$, hence in~$V^\subge_0$ and in~$V^\suble_0$; hence $g$ also stabilizes $A^\sube_0$, $A^\subge_0$ and $A^\suble_0$.
\end{proof}

For a general $g \in G \ltimes V$, we define the ideal dynamical spaces to be the inverse images of the reference dynamical spaces by a ``canonizing'' map~$\phi$ such that the conjugate $\phi g \phi^{-1}$ is in canonical form. However, to ensure that the \emph{translation} part of~$\phi$ is well-defined, we need $g$ to be regular \emph{with respect to~$\rho$}. More precisely:

\begin{proposition}
\label{canonizing}
Let $g \in G \ltimes V$ be a $\rho$-regular map. Then:
\begin{hypothenum}
\item there exists a map $\phi \in G \ltimes V$, called a \emph{canonizing map} for~$g$, such that $\phi g \phi^{-1}$ is in canonical form;
\item any two such maps differ by left-multiplication by an element of~$L_{X_0} \ltimes V^\sube_0$.
\end{hypothenum}
\end{proposition}

The key point of the proof is the following lemma:

\begin{lemma}
\label{canonical_form_property}
If $g \in G \ltimes V$ is in canonical form and is $\rho$-regular, then the linear map $\ell(g) - \Id$ induces an invertible linear map on the quotient space $V/V^\sube_0$.
\end{lemma}
\begin{proof}
The fact that the quotient map is well-defined follows from Proposition~\ref{reference_spaces_invariant}: indeed since $g$ is $\rho$-regular, it is in particular $X_0$-regular; and then $\ell(g)$, hence $\ell(g) - \Id$, stabilizes the subspace $V^\sube_0$.

Let us now show that the quotient map is invertible. All eigenvalues of the restriction of the hyperbolic part~$\exp(\jordan(g))$ to the subspace $V^\subg_0 \oplus V^\subl_0$ are real and positive, and (since $g$ is $\rho$-regular) different from~$1$. Since the elliptic and unipotent parts of~$\ell(g)$ commute with the hyperbolic part of~$\ell(g)$ and have all eigenvalues of modulus~$1$, it follows that all eigenvalues of the restriction of~$\ell(g)$ to~$V^\subg_0 \oplus V^\subl_0$ (this subspace is stable by~$\ell(g)$ by Proposition~\ref{reference_spaces_invariant}) are different from~$1$. In particular the restriction of $\ell(g)-\Id$ to~$V^\subg_0 \oplus V^\subl_0$ is invertible. The conclusion follows.
\end{proof}

\begin{proof}[Proof of Proposition~\ref{canonizing}]~
\begin{hypothenum}
\item Let $\phi_\ell \in G$ be a canonizing map for~$\ell(g)$, which exists by Proposition~\ref{jordan_decomposition}. We then have
\[\phi_\ell g \phi_\ell^{-1} = \tau_v g',\]
where $g' \in G$ is in canonical form and $v \in V$.

We now claim that for a suitable choice of~$w \in V$, the map $\phi := \tau_w \phi_\ell$ is a canonizing map for~$g$. Indeed we then have:
\[\phi g \phi^{-1} = \tau_w \tau_v g' \tau_{-w} = \tau_{v+w-g'(w)} g'.\]
We already know that $g' \in G$~is in canonical form. On the other hand, by Lemma~\ref{canonical_form_property} (here we need surjectivity of the quotient map), we may choose~$w \in V$ in such a way that
\[v + w - g'(w) \in V^\sube_0,\]
which finishes the proof.
\item Assume that $g \in G \ltimes V$~is already in canonical form, so that
\[g = \tau_v \ell(g)\]
with $v \in V^\sube_0$ and~$\ell(g) \in G$. It is enough to show that any~$\phi \in G \ltimes V$ such that $\phi g \phi^{-1}$ is still in canonical form is an element of~$L_{X_0} \ltimes V^\sube_0$. Indeed, let $\phi$ be such a map; let us write
\[\phi = \tau_w \ell(\phi),\]
where $w \in V$ is its translation part and $\ell(\phi) \in G$ is its linear part. By Proposition~\ref{centralizer_of_x}, the fact that $\ell(\phi)$ commutes with~$\jordan(g)$ implies that $\ell(\phi) \in L_{\jordan(g)} \subset L_{X_0}$. As for the translation part, if we have $w - \ell(g)(w) \in V^\sube_0$, by Lemma~\ref{canonical_form_property} (here we need injectivity of the quotient map), we have $w \in V^\sube_0$.
\qedhere
\end{hypothenum}
\end{proof}

\begin{definition}[Affine version of geometry]
\label{ideal_dynamical_spaces_definition}
For any $\rho$-regular map~$g \in G \ltimes V$, we introduce the following eight spaces, called \emph{ideal dynamical spaces} of~$g$:
\begin{itemize}
\item $V^\subgg_g := \phi^{-1}(V^\subg_0)$, the \emph{ideally expanding space} associated to~$g$;
\item $V^\subll_g := \phi^{-1}(V^\subl_0)$, the \emph{ideally contracting space} associated to~$g$;
\item $V^\subap_g := \phi^{-1}(V^\sube_0)$, the \emph{ideally neutral space} associated to~$g$;
\item $V^\subgap_g := \phi^{-1}(V^\subge_0)$, the \emph{ideally noncontracting space} associated to~$g$;
\item $V^\sublap_g := \phi^{-1}(V^\suble_0)$, the \emph{ideally nonexpanding space} associated to~$g$;
\item $A^\subgap_g := \phi^{-1}(A^\subge_0)$, the \emph{affine ideally noncontracting space} associated to~$g$;
\item $A^\sublap_g := \phi^{-1}(A^\suble_0)$, the \emph{affine ideally nonexpanding space} associated to~$g$;
\item $A^\subap_g := \phi^{-1}(A^\sube_0)$, the \emph{affine ideally neutral space} associated to~$g$,
\end{itemize}
where $\phi$~is any canonizing map of~$g$. We shall justify in a moment (Proposition~\ref{uniquely_determined}.(i)) that these definitions do not depend on the choice of~$\phi$.
\end{definition}

Here is the idea behind this definition. Suppose first that the representation is non-swinging (so that $X_0$~is actually generic), and that the Jordan projection of~$g$ is sufficiently close to~$X_0$. Then it actually has the same type as~$X_0$, \ie we have $\Omega^\subg_{\jordan(g)} = \Omega^\subg_{X_0}$ and similarly for $\Omega^\sube$, $\Omega^\subl$. Whenever that happens, the ideal dynamical spaces of~$g$ coincide with its actual dynamical spaces, as defined in section~4.3 of~\cite{Smi16}.

If $X_0$ is not generic, this is no longer true as such. We still want to assume that $\jordan(g)$ is sufficiently close to~$X_0$; but now, this can at best ensure that $g$ is compatible with~$X_0$. In that case, we only get that:
\begin{itemize}
\item the moduli of the eigenvalues of~$\restr{g}{V^\subgg_g}$ are ``much larger'' than~$1$;
\item the moduli of the eigenvalues of~$\restr{g}{V^\subll_g}$ are ``much smaller'' than~$1$;
\item the moduli of the eigenvalues of~$\restr{g}{A^\subap_g}$ might now differ from~$1$, but somehow remain ``not too far'' from~$1$.
\end{itemize}

Let us finally check that this definition makes sense, and prove a few extra properties along the way.

\begin{proposition}~
\label{uniquely_determined}
\begin{hypothenum}
\item The definitions above do not depend on the choice of~$\phi$;
\item the datum of~$A^\subgap_g$ uniquely determines the spaces $V^\subgap_g$ and~$V^\subgg_g$;
\item the datum of~$A^\sublap_g$ uniquely determines the spaces $V^\sublap_g$ and~$V^\subll_g$;
\item the data of both $A^\subgap_g$ and~$A^\sublap_g$ uniquely determine all eight ideal dynamical spaces.
\end{hypothenum}
\end{proposition}
The spaces $A^\subgap_g$ and~$A^\sublap_g$ will play a crucial role, as they are the affine analogs of the attracting and repelling flags $y^{X_0, +}_g$ and~$y^{X_0, -}_g$ defined in section~\ref{sec:X-regular}: see Remark~\ref{affine_flag_varieties} below for an explanation.
\begin{proof}
An immediate corollary of Proposition~\ref{stabg=stabge} is that we have
\begin{align}
\label{eq:affine_stabilizers}
\Stab_{G \ltimes V}(A^\subge_0) &= P^+_{X_0} \ltimes V^\subge_0
& \Stab_{G \ltimes V}(V^\subge_0) &= P^+_{X_0} \ltimes V
& \Stab_{G \ltimes V}(V^\subg_0) &= P^+_{X_0} \ltimes V \nonumber \\
\Stab_{G \ltimes V}(A^\suble_0) &= P^-_{X_0} \ltimes V^\suble_0
& \Stab_{G \ltimes V}(V^\suble_0) &= P^-_{X_0} \ltimes V
& \Stab_{G \ltimes V}(V^\subl_0) &= P^-_{X_0} \ltimes V \nonumber \\
\Stab_{G \ltimes V}(A^\sube_0) &= L_{X_0} \ltimes V^\sube_0
& \Stab_{G \ltimes V}(V^\sube_0) &= L_{X_0} \ltimes V,
& &
\intertext{with the following relationships:}
P^+_{X_0} &\ltimes V^\subge_0
\mathrlap{\qquad\quad\;\subset} & P^+_{X_0} &\ltimes V
\mathrlap{\qquad\quad\;=} & P^+_{X_0} &\ltimes V \nonumber \\
P^-_{X_0} &\ltimes V^\suble_0
\mathrlap{\qquad\quad\;\subset} & P^-_{X_0} &\ltimes V
\mathrlap{\qquad\quad\;=} & P^-_{X_0} &\ltimes V \nonumber \\
L_{X_0} &\ltimes V^\sube_0
\mathrlap{\qquad\quad\;\subset} & L_{X_0} &\ltimes V. \nonumber
\end{align}
The first two lines immediately imply (ii) and (iii).

For points (i) and (iv), note that all eight groups contain $L_{X_0} \ltimes V^\sube_0$. Point (i) follows by Proposition~\ref{canonizing}. Point (iv) follows using the identity
\[(P^+_{X_0} \ltimes V^\subge_0) \cap (P^-_{X_0} \ltimes V^\suble_0) = L_{X_0} \ltimes V^\sube_0. \qedhere\]
\end{proof}

\subsection{Quasi-translations}
\label{sec:quasi-translations}

Let us now investigate the action of a $\rho$-regular map~$g \in G \ltimes V$ on its affine ideally neutral space~$A^\subap_g$. The goal of this subsection is to prove that it is ``almost'' a translation (Proposition~\ref{V=_translation}).

We fix on~$V$ a Euclidean form~$B$ satisfying the conditions of Lemma~\ref{K-invariant} for the representation~$\rho$.

\begin{definition}
\label{quasi-translation_def}
We call \emph{quasi-translation} any affine automorphism of $A^\sube_0$ induced by an element of $L \ltimes V^\sube_0$.
\end{definition}

Let us explain in what sense quasi-translations are indeed ``almost'' translations.
\begin{proposition}
\label{quasi_translation_charact}
Let $V^\transl_0$ be the set of fixed points of~$L$ in~$V^\sube_0$:
\[V^\transl_0 := \setsuch{v \in V^\sube_0}{\forall l \in L,\; l v = v}.\]
Let $V^\extra_0$ be the $B$-orthogonal complement of~$V^\transl_0$ in~$V^\sube_0$. Then any quasi-translation is an element of
\[\Big( \GL(V^\extra_0) \ltimes V^\extra_0 \Big) \times V^\transl_0.\]
\end{proposition}
In other words, quasi-translations are affine automorphisms of $V^\sube_0$ that preserve the directions of $V^\extra_0$ and $V^\transl_0$ and act only by translation on the $V^\transl_0$ component; the superscripts $\transl$ and~$\extra$ respectively stand for ``translation'' and ``affine''. You may think of a quasi-translation as a kind of ``screw displacement''. The reason we consider them to be ``almost translations'' is that we are going to focus on their action on~$V^\transl_0$, and try to get $V^\extra_0$ out of the picture.
\begin{remark}
If the representation~$\rho$ satisfies the hypotheses of the Main Theorem, then hypothesis~\ref{itm:main_condition}\ref{itm:fixed_by_l} ensures that the space~$V^\transl_0$ is actually nonzero.
\end{remark}
\begin{proof}
We need to show that every element of~$L$ fixes~$V^\transl_0$ (pointwise) and leaves~$V^\extra_0$ invariant (globally). The former is true by definition of~$V^\transl_0$. For the latter, let us prove it separately for elements of~$M$ and elements of~$A$:
\begin{itemize}
\item By hypothesis, elements of~$M$ preserve the form~$B$; since they leave invariant the space~$V^\transl_0$, they also leave invariant its $B$-orthogonal complement~$V^\extra_0$.
\item Let us introduce the notation
\begin{equation}
V^\subenz_0 := \bigoplus_{\lambda \in \Omega^\sube_{X_0} \setminus \{0\}} V^\lambda
\end{equation}
(the symbol ``$\Bumpeq$'' is intended to represent the idea of ``avoiding zero''), so that the space~$V^\sube_0$ decomposes into the orthogonal sum 
\begin{equation}
\label{eq:V_eq_decomp}
V^\sube_0 = V^0 \oplus V^\subenz_0.
\end{equation}
Then since $V^\transl \subset V^0$ (obviously), similarly, we have an orthogonal sum
\begin{equation}
\label{eq:V_a_decomp}
V^\extra_0 = (V^\extra_0 \cap V^0) \oplus V^\subenz_0.\
\end{equation}
Now clearly $V^\subenz_0$, being a sum of restricted weight spaces, is invariant by~$A$. Moreover, every element of~$A$ actually fixes every element of~$V^0$, in particular leaves invariant the subspace $V^\extra_0 \cap V^0$. The conclusion follows. \qedhere
\end{itemize}
\end{proof}
\begin{remark}
\label{quasi-translation-generalization}
Note that in contrast to the non-swinging case (Proposition~4.19 in~\cite{Smi16}), quasi-translations no longer have to act by isometries on~$V^\extra_0$, as this space now comprises the possibly nontrivial space~$V^\subenz_0$ where $A$ acts nontrivially.

This phenomenon is the reason why the reasoning we used in~\cite{Smi16} to prove additivity of Margulis invariants could not be reused here without major restructuring.
\end{remark}

We now claim that any $\rho$-regular map acts on its affine ideally neutral space by quasi-translations:
\begin{proposition}
\label{V=_translation}
Let $g \in G \ltimes V$ be a $\rho$-regular map, and let~$\phi \in G \ltimes V$ be any canonizing map for~$g$. Then the restriction of the conjugate $\phi g \phi^{-1}$ to $A^\sube_0$ is a quasi-translation.
\end{proposition}
Let us actually formulate an even more general result, which will have another application in the next subsection:
\begin{lemma}
\label{quasi-translation}
Any map $f \in G \ltimes V$ stabilizing both $A^\subge_0$ and $A^\suble_0$ acts on $A^\sube_0$ by quasi-translation.
\end{lemma}
\begin{proof}~
\begin{itemize}
\item We begin by showing that any element of $\mathfrak{l}_{X_0} = \mathfrak{p}_{X_0}^+ \cap \mathfrak{p}_{X_0}^-$ acts on $V^\sube_0$ in the same way as some element of $\mathfrak{l}$. Recall that by definition
\[\mathfrak{l}_{X_0} = \mathfrak{l} \oplus \bigoplus_{\alpha(X_0) = 0} \mathfrak{g}^\alpha\]
and
\[V^\sube_0 = \bigoplus_{\lambda(X_0) = 0} V^\lambda.\]
Thus we want to show that for every restricted root~$\alpha$ and restricted weight~$\lambda$ such that $\alpha(X_0) = \lambda(X_0) = 0$, we have $\mathfrak{g}^\alpha \cdot V^\lambda = 0$. It is sufficient to show that in such a case, the sum~$\lambda+\alpha$ is no longer a restricted weight. But otherwise, both $\lambda$ and~$\lambda+\alpha$ would be elements of~$\Omega^\sube_{X_0}$, hence they would both be fixed by~$w_0$ since $X_0$~is generically symmetric. This would mean that $\alpha$~is also fixed by~$w_0$, which is impossible.

\item We now conclude in the same fashion as in the proof of Lemma~4.21 in~\cite{Smi16}, passing from the two Lie algebras $\mathfrak{p}_{X_0}^\pm$ first to the identity components~$P_{X_0, e}^\pm$ of~$P_{X_0}^\pm$, then to the whole groups~$P_{X_0}^\pm$, then (using Proposition~\ref{stabg=stabge}) to the stabilizers of~$A^\subge_0$ and of~$A^\suble_0$. \qedhere
\end{itemize}
\end{proof}\begin{proof}[Proof of Proposition~\ref{V=_translation}]
The proposition follows immediately by taking $f = \phi g \phi^{-1}$. Indeed, by Proposition~\ref{reference_spaces_invariant}, the ``canonized'' map~$\phi g \phi^{-1}$ stabilizes~$A^\subge_0$ and~$A^\suble_0$.
\end{proof}
See Example~4.22 in~\cite{Smi16} for examples of what quasi-translations become in some particular cases (all of them non-swinging).


\subsection{Canonical identifications and the Margulis invariant}
\label{sec:canonical}
The main goal of this subsection is to associate to every $\rho$-regular map~$g \in G \ltimes V$ a vector in~$V^\transl_0$, called its ``Margulis invariant'' (see Definition~\ref{margulis_invariant}). The two Propositions and the Lemma that lead up to this definition are important as well, and will be often used subsequently.

Proposition~\ref{uniquely_determined}~(iv) has shown us that the ``geometry'' of a $\rho$-regular map~$g$ (namely the position of its ideal dynamical spaces) is entirely determined by the pair of spaces
\[(A^\subgap_g, A^\sublap_g) = \phi(A^\subge_0, A^\suble_0).\]
In fact, such pairs of spaces play a crucial role. Let us begin with a definition; its connection with the observation we just made will become clear after Proposition~\ref{pair_transitivity}.
\begin{definition}~
\begin{itemize}
\item We define a \emph{parabolic space} to be any subspace of $V$ that is the image of either~$V^\subge_0$ or~$V^\suble_0$ (no matter which one, since $X_0$~is symmetric) by some element of $G$.
\item We define an \emph{affine parabolic space} to be any subspace of $A$ that is the image of $A^\subge_0$ by some element of $G \ltimes V$. Equivalently, a subspace $A^\subgap \subset A$ is an affine parabolic space if and only if it is not contained in $V$ and its linear part $V^\subgap = A^\subgap \cap V$ is a parabolic space.
\item We say that two parabolic spaces (or two affine parabolic spaces) are \emph{transverse} if their intersection has the lowest possible dimension, or equivalently if their sum is the whole space~$V$ (or~$A$).
\end{itemize}
\end{definition}
See Example~4.26 in~\cite{Smi16}.
\begin{proposition}
\label{pair_transitivity}
A pair of parabolic spaces (resp. of affine parabolic spaces) is transverse if and only if it may be sent to~$(V^\subge_0, V^\suble_0)$ (resp. to~$(A^\subge_0, A^\suble_0)$) by some element of~$G$ (resp. of~$G \ltimes V$).
\end{proposition}
In particular, for any $\rho$-regular map~$g \in G \ltimes V$, the pair~$(A^\subgap_g, A^\sublap_g)$ is by definition a transverse pair of affine parabolic spaces.

This proposition, as well as its proof, is very similar to Claim~2.8 in~\cite{Smi14} and to Proposition~4.27 in~\cite{Smi16}.
\begin{proof}
Let us prove the linear version; the affine version follows immediately. Let $(V_1, V_2)$ be any pair of parabolic spaces. By definition, for $i = 1, 2$, we may write $V_i = \phi_i(V^\subge_0)$ for some $\phi_i \in G$. Let us apply the Bruhat decomposition to the map $\phi_1^{-1}\phi_2$: we may write
\begin{equation}
\phi_1^{-1}\phi_2 = p_1wp_2,
\end{equation}
where $p_1, p_2$ belong to the minimal parabolic subgroup $P^+$, and $w$ is an element of the restricted Weyl group $W$ (or, technically, some representative thereof). Let $\phi := \phi_1p_1 = \phi_2p_2^{-1}w^{-1}$; since $P^+$ stabilizes $V^\subge_0$, we have
\begin{equation}
V_1 = \phi(V^\subge_0) \;\text{ and }\; V_2 = \phi(w V^\subge_0).
\end{equation}
Thus we have
\begin{align}
\text{$V_1$ and~$V_2$ are transverse}
  &\iff \text{$w V^\subge_0$ is transverse to $V^\subge_0$} \nonumber \\
  &\iff \Omega^\subge_{X_0} \cup w\Omega^\subge_{X_0} = \Omega \nonumber \\
  &\iff w\Omega^\subl_{X_0} \subset \Omega^\subge_{X_0} \nonumber \\
  &\iff w w_0\Omega^\subg_{X_0} \subset \Omega^\subge_{X_0}.
\end{align}
By Lemma~\ref{Weyl_semi_stabilizer}, this implies that $w w_0$~stabilizes~$\Omega^\subg_{X_0}$. Hence it stabilizes its complement~$\Omega^\suble_{X_0}$, which means that $w V^\subge_0 = V^\suble_0$. Thus we have $V_1 = \phi(V^\subge_0)$ and~$V_2 = \phi(V^\suble_0)$ as required.

Conversely, if those equalities hold, then $V_1$ and~$V_2$ are obviously transverse.
\end{proof}

\begin{remark}~
\label{affine_flag_varieties}
\begin{itemize}
\item It follows from Proposition~\ref{stabg=stabge} that the set of all parabolic spaces can be identified with the flag variety~$G/P^+_{X_0}$, by identifying every parabolic space~$\phi(V^\subge_0)$ with the coset~$\phi P^+_{X_0}$. For every $X_0$-regular element $g \in G$, this identification then matches the ideally noncontracting space~$V^\subgap_g$ with the attracting flag~$y^{X_0, +}_g$.
\item Composing with the bijection~$\iota_X$ defined in Proposition~\ref{intrinsic_inverse}, we may also identify this set with the \emph{opposite} flag variety~$G/P^-_{X_0}$. For $X_0$-regular $g \in G$, this matches the ideally \emph{nonexpanding} space~$V^\sublap_g$ with the \emph{repelling} flag~$y^{X_0, -}_g$.
\item Using the Bruhat decomposition of~$P_{X_0}$ (see Corollary~\ref{parabolic_Bruhat_decomposition}), we may then show that two parabolic spaces $V_1 = \phi_1(V^\subge_0)$ and~$V_2 = \phi_2(V^\subge_0) = \phi_2 \circ w_0(V^\suble_0)$ are transverse if and only if the corresponding pair of cosets
\[(\phi_1 P^+_{X_0},\; \phi_2 w_0 P^-_{X_0})\]
is transverse in the sense of Definition~\ref{transverse_purely_linear}.
\item Similarly, it follows from~\eqref{eq:affine_stabilizers} that we can in principle identify the set of all \emph{affine} parabolic spaces with the ``affine flag variety''~$G \ltimes V/P^+_{X_0} \ltimes V^\subge_0$. This would however require very cumbersome notations.

In the linear case, it was natural to do it in order to make things representation-independent. In the affine case however, there is a ``privileged'' representation anyway (namely~$\rho$). We decided that translating everything into that ``abstract'' language was not worth the trouble, but the reader may go through that exercise if they wish so.
\end{itemize}
\end{remark}

Consider a transverse pair of affine parabolic spaces. Their intersection may be seen as a sort of ``abstract affine neutral space''. We now introduce a family of ``canonical identifications'' between those spaces. These identifications have however an inherent ambiguity: they are only defined up to quasi-translation.
\begin{proposition}
\label{canonical_identification}
Let $(A_1, A_2)$ be a pair of transverse affine parabolic spaces. Then any map $\phi \in G \ltimes V$ such that $\phi(A_1, A_2) = (A^\subge_0, A^\suble_0)$ gives, by restriction, an identification of the intersection $A_1 \cap A_2$ with $A^\sube_0$, which is unique up to quasi-translation.
\end{proposition}
Here by $\phi(A_1, A_2)$ we mean the pair $(\phi(A_1), \phi(A_2))$. This generalizes Corollary~2.14 in~\cite{Smi14} and Proposition~4.29 in~\cite{Smi16}.
\begin{remark}
Note that if $A_1 \cap A_2$ is obtained in another way as an intersection of two affine parabolic spaces, the identification with~$A^\sube_0$ will, in general, no longer be the same, not even up to quasi-translation: there could also be an element of the Weyl group involved.
\end{remark}
\begin{proof}
The existence of such a map~$\phi$ follows from Proposition~\ref{pair_transitivity}. Now let $\phi$ and~$\phi'$ be two such maps, and let $f$~be the map such that
\begin{equation}
\phi' = f \circ \phi
\end{equation}
(\ie $f := \phi' \circ \phi^{-1}$). Then by construction $f$~stabilizes both $A^\subge_0$ and~$A^\suble_0$. It follows from Lemma~\ref{quasi-translation} that the restriction of~$f$ to~$A^\sube_0$ is a quasi-translation.
\end{proof}
Let us now explain why we call these identifications ``canonical''. The following lemma, while seemingly technical, is actually crucial: it tells us that the identifications defined in Proposition~\ref{canonical_identification} commute (up to quasi-translation) with the projections that naturally arise if we change one of the parabolic subspaces in the pair while fixing the other.
\begin{lemma}
\label{projections_commute}
Take any affine parabolic space~$A_1$.

Let $A_2$ and~$A'_2$ be any two affine parabolic spaces both transverse to~$A_1$.

Let~$\phi$ (resp.~$\phi'$) be an element of~$G \ltimes V$ that sends the pair~$(A_1, A_2)$ (resp.~$(A_1, A'_2)$) to~$(A^\subge_0, A^\suble_0)$; these two maps exist by Proposition~\ref{pair_transitivity}.

Let~$W_1$ be the inverse image of~$V^\subg_0$ by any map~$\phi$ such that~$A_1 = \phi^{-1}(A^\subge_0)$ (this image is unique by Proposition~\ref{stabg=stabge}).

Let
\[\psi: A_1 \longlongrightarrow A_1 \cap A'_2\]
be the projection parallel to~$W_1$.

Then the map~$\overline{\psi}$ defined by the commutative diagram
\[
\begin{tikzcd}
  A^\sube_0
    \arrow{rr}{\overline{\psi}}
&
& A^\sube_0\\
\\
  A_1 \cap A_2
    \arrow{rr}{\psi}
    \arrow{uu}{\phi}
&
& A_1 \cap A'_2
    \arrow{uu}{\phi'}
\end{tikzcd}
\]
is a quasi-translation.
\end{lemma}
The space~$W_1$ is, in a sense, the ``abstract linear expanding space'' corresponding to the ``abstract affine noncontracting space''~$A_1$: more precisely, for any $\rho$-regular map~$g \in G \ltimes V$ such that $A^\subgap_g = A_1$, we have $V^\subgg_g = W_1$.

The projection $\psi$ is well-defined because $A^\subge_0 = V^\subg_0 \oplus A^\sube_0 = V^\subg_0 \oplus (A^\subge_0 \cap A^\suble_0)$, and so $A_1 = \phi'^{-1}(A^\subge_0) = W_1 \oplus (A_1 \cap A'_2)$.

This statement generalizes Lemma~2.18 in~\cite{Smi14} and Lemma~4.30 in~\cite{Smi16}.
\begin{proof}
The proof is exactly the same as the proof of Lemma~4.30 in~\cite{Smi16}.
\end{proof}

Now let $g$~be a $\rho$-regular map. We already know that it acts on its affine ideally neutral space by quasi-translation; now the canonical identifications we have just introduced allow us to compare the actions of different elements on their respective affine ideally neutral spaces, as if they were both acting on the same space~$A^\sube_0$. However there is a catch: since the identifications are only canonical up to quasi-translation, we lose information about what happens in~$V^\extra_0$; only the translation part along~$V^\transl_0$ remains.

Formally, we make the following definition. Let $\pi_\transl$~denote the projection from~$V^\sube_0$ onto~$V^\transl_0$ parallel to~$V^\extra_0$.
\begin{definition}
\label{margulis_invariant}
Let $g \in G \ltimes V$ be a $\rho$-regular map. Take any point $x$ in the affine space $A^{\subap}_{g} \cap V_{\Aff}$ and any map $\phi \in G$ that canonizes $\ell(g)$. Then we define the \emph{Margulis invariant} of~$g$ to be the vector
\[M(g) := \pi_\transl(\phi(g(x)-x)) \in V^\transl_0.\]
\end{definition}
This vector does not depend on the choice of~$x$ or~$\phi$: indeed, composing~$\phi$ with a quasi-translation does not change the $V^\transl_0$-component of the image. See Proposition~2.16 in~\cite{Smi14} for a detailed proof of this claim (for $V = \mathfrak{g}$).

Informally, the Margulis invariant gives the ``translation part'' of the asymptotic dynamics of an $\mathbb{R}$-regular element $g \in G \ltimes V$ (the ``linear part'' being given by $\jordan(g)$, just as in the linear case). As such, it plays a central role in this paper.

\subsection{Quantitative properties}
\label{sec:regular_metric}

In this subsection, we define and study the affine versions of non-degeneracy and contraction strength. We more or less follow the second half of Subsection~5.1 in~\cite{Smi16}, or of Subsection~2.6 in~\cite{Smi14}.

We endow the extended affine space~$A$ with a Euclidean norm (written simply $\| \bullet \|$) whose restriction to~$V$ coincides with the norm~$B$ defined in Lemma~\ref{K-invariant} and that makes~$p_0$ a unit vector orthogonal to~$V$. Then the subspaces $V^\subg_0$, $V^\subl_0$, $V^\subenz_0$, $V^\extra_0 \cap V^0$, $V^\transl_0$ and~$\mathbb{R} p_0$ are pairwise orthogonal. 

\begin{definition}[Affine version of non-degeneracy]
\label{regular_definition}
Take a pair of affine parabolic spaces $(A_1, A_2)$. An \emph{optimal canonizing map} for this pair is a map $\phi \in G \ltimes V$ satisfying
\[\phi(A_1, A_2) = (A^\subge_0, A^\suble_0)\]
and minimizing the quantity $\max \left( \|\phi\|, \|\phi^{-1}\| \right)$. By Proposition~\ref{pair_transitivity} and a compactness argument, such a map exists if and only if $A_1$ and $A_2$ are transverse.

We define an \emph{optimal canonizing map} for a $\rho$-regular map $g \in G \ltimes V$ to be an optimal canonizing map for the pair $(A^\subgap_g, A^\sublap_g)$.

Let $C \geq 1$. We say that a pair of affine parabolic spaces $(A_1, A_2)$ (resp. a $\rho$-regular map~$g$) is \emph{$C$-non-degenerate} if it has a canonizing map $\phi$ such that
\[\left \|\phi \right\| \leq C \quad\textnormal{and}\quad \left \|\phi^{-1} \right\| \leq C.\]

Now take $g_1$, $g_2$ two $\rho$-regular maps in $G \ltimes V$. We say that the pair $(g_1, g_2)$ is \emph{$C$-non-degenerate} if every one of the four possible pairs $(A^{\subgap}_{g_i}, A^{\sublap}_{g_j})$ is $C$-non-degenerate.
\end{definition}

The point of this definition is that there are a lot of calculations in which, when we treat a $C$-non-degenerate pair of spaces as if they were perpendicular, we err by no more than a (multiplicative) constant depending on $C$.

\begin{remark}
\label{unified_treatment}
The set of transverse pairs of extended affine spaces is characterized by two open conditions: there is of course transversality of the spaces, but also the requirement that each space not be contained in $V$. What we mean here by ``degeneracy'' is failure of one of these two conditions. Thus the property of a pair $(A_1, A_2)$ being $C$-non-degenerate actually encompasses two properties.

First, it implies that the spaces $A_1$ and $A_2$ are transverse in a quantitative way. More precisely, this means that some continuous function that would vanish if the spaces were not transverse is bounded below. An example of such a function is the smallest non identically vanishing of the ``principal angles'' defined in the proof of Lemma~\ref{regular_to_proximal}~(iv).

Second, it implies that both $A_1$ and $A_2$ are ``not too close'' to the space $V$ (in the same sense). In purely affine terms, this means that the affine spaces $A_1 \cap V_{\Aff}$ and $A_2 \cap V_{\Aff}$ contain points that are not too far from the origin.

Both conditions are necessary, and appeared in the previous literature (such as \cite{Mar87} and~\cite{AMS02}). However, they were initially treated separately. The idea of encompassing both in the same concept of ``$C$-non-degeneracy'' seems to have been first introduced in the author's earlier paper~\cite{Smi14}.
\end{remark}

%
%

\begin{definition}[Affine version of contraction strength]
\label{s_definition}
Let $s > 0$. For a $\rho$-regular map $g \in G \ltimes V$, we say that $g$ is \emph{$s$-contracting along~$X_0$} if we have:
\begin{equation}
\forall (x, y) \in V^{\subll}_{g} \times A^{\subgap}_{g}, \quad
  \frac{\|g(x)\|}{\|x\|} \leq s\frac{\|g(y)\|}{\|y\|}.
\end{equation}
(Note that by definition the spaces $V^{\subll}_{g}$ and $A^{\subgap}_{g}$ always have the same dimensions as $V^\subl_0$ and $A^\subge_0$ respectively, hence they are nonzero.)

We define the \emph{affine contraction strength along~$X_0$} of $g$ to be the smallest number $s_{X_0}(g)$ such that $g$ is $s_{X_0}(g)$-contracting along~$X_0$. In other words, we have
\begin{equation}
s_{X_0}(g) =
\left\| \restr{g}     {V^{\subll}_{g}}  \right\|
\left\| \restr{g^{-1}}{A^{\subgap}_{g}} \right\|.
\end{equation}
This should not be confused with the linear contraction strength $\vec{s}$ defined in Definition~\ref{vec_s_definition}.
\end{definition}

This notion is closely related to the notion of asymptotic contraction:

\begin{proposition}~
\label{contraction_strength_limit}
\begin{hypothenum}
\item \label{itm:one_to_asymp_contr} If a map~$g \in G \ltimes V$ is $\rho$-regular and $1$-contracting along $X_0$ (\ie such that ${s_{X_0}(g) < 1}$), then it is also asymptotically contracting along~$X_0$.
\item \label{itm:asymptotic_explanation} If a map~$g \in G \ltimes V$ is $\rho$-regular and asymptotically contracting along~$X_0$, then we have
\[\lim_{N \to \infty} s_{X_0}(g^N) = 0.\]
\end{hypothenum}
\end{proposition}
\begin{proof}
Let~$g \in G \ltimes V$ be $\rho$-regular. We claim that it is asymptotically contracting along~$X_0$ if and only if it satisfies the inequality
\begin{equation}
r\left( \restr{g}     {V^{\subll}_{g}}  \right)
r\left( \restr{g^{-1}}{A^{\subgap}_{g}} \right) < 1
\end{equation}
(where $r(f)$ denotes the spectral radius of~$f$). Indeed, as a porism of Proposition~\ref{eigenvalues_and_singular_values_characterization}~(i), we obtain that the spectrum of~$\restr{g}{V^{\subll}_{g}}$ is
\[\setsuch{e^{\lambda(\jordan(g))}}{\lambda \in \Omega^\subl_{X_0}},\]
and the spectrum of~$\restr{g}{A^{\subgap}_{g}}$ is
\[\setsuch{e^{\lambda(\jordan(g))}}{\lambda \in \Omega^\subge_{X_0}} \cup \{1\}\]
(the ``$1$'' accounts for the affine extension). By Assumption~\ref{zero_is_a_weight}, the eigenvalue~$1$ is already contained in the spectrum of the linear part, so we may actually ignore the ``$\cup \{1\}$'' part. The claim follows.

The conclusion then follows from the well-known facts that for every linear map~$f$, we have
\begin{equation}
r(f) \leq \|f\|
\end{equation}
(this gives (i)), and
\begin{equation}
\label{eq:Gelfand}
\log \|f^N\| = N \log r(f) + \underset{N \to \infty}{\bigo}(\log N)
\end{equation}
(also known as Gelfand's formula; this gives (ii)).
\end{proof}

\subsection{Comparison of affine and linear properties}
\label{sec:affine_and_linear}

The goal of this subsection is to prove Proposition~\ref{affine_to_intrinsic_linear}, that, for any $\rho$-regular element~$g \in G \ltimes V$, relates the quantitative properties we just introduced to the corresponding properties of its linear part~$\ell(g)$.

This is given by Lemma~2.25 in~\cite{Smi14} in the case of the adjoint representation, and by Lemma~5.8 in~\cite{Smi16} (which is a straightforward generalization) in the non-swinging case. In the general case however, only points (i) and~(ii) generalize in the obvious way. The statement~(iii) holds only in a weaker form: we now need to consider the \emph{linear} contraction strength of~$\ell(g)$, but the \emph{affine} contraction strength of~$g$. This is basically what forced us to develop the ``purely linear theory'' (section~\ref{sec:linear_regular_maps}) in a systematic way, rather than presenting it as a particular case of the ``affine theory'' as we did in the previous papers.

In order to be able to compare the affine contraction strength with the linear one, we begin by expressing the former in terms of the Cartan projection. The following is a generalization of Lemma~6.6 in~\cite{Smi16}, formulated in a slightly more general way (the original statement was essentially the same inequality without the absolute value).

\begin{proposition}
\label{contraction_strength_and_cartan_projection}
For every $C \geq 1$, there is a constant~$k_{\ref{contraction_strength_and_cartan_projection}}(C)$ with the following property. Let $g \in G$ be a $C$-non-degenerate $\rho$-regular map. Then we have
\[\left| - \log s_{X_0}(g) - \left( \min_{\lambda \in \Omega^\subge_{X_0}} \lambda(\cartan(g)) \;-\;
\max_{\lambda \in \Omega^\subl_{X_0}} \lambda(\cartan(g)) \right) \right|
  \;\;\leq\;\;
k_{\ref{contraction_strength_and_cartan_projection}}(C).\]
(Recall that $\Omega^\subge_{X_0}$~is the set of restricted weights that take nonnegative values on~$X_0$, and~$\Omega^\subl_{X_0}$ is its complement in~$\Omega$.)
\end{proposition}
To make sense of this estimate, keep in mind that the quantity $- \log s_{X_0}(g)$ is typically positive. Also note that the ``minimum'' term is certainly nonpositive, as $0 \in \Omega^\subge_{X_0}$.
\begin{proof}
First of all let $\phi$ be an optimal canonizing map for~$g$, and let $g' = \phi g \phi^{-1}$. Then it is easy to see that we have
\begin{equation}
\label{eq:WLOG_canonical_form}
s_{X_0}(g') \asymp_C s_{X_0}(g),
\end{equation} 
and we have already seen~\eqref{eq:cartan_g_g'} that the difference $\cartan(g') - \cartan(g)$ is bounded by a constant that depends only on~$C$. Hence we lose no generality in replacing~$g$ by~$g'$. Clearly, it is enough to show that for~$g'$, which is in canonical form, we have the equality:
\begin{equation}
\label{eq:sg'_majoration}
\min_{\lambda \in \Omega^\subge_{X_0}} \lambda(\cartan(g')) -
\max_{\lambda \in \Omega^\subl_{X_0}} \lambda(\cartan(g'))
  \;=\;
- \log s_{X_0}(g').
\end{equation}

This is the straightforward generalization of~(6.4) in~\cite{Smi16}, and is proved in exactly the same fashion, \emph{mutatis mutandis}.
\end{proof}

As a stepping stone for Proposition~\ref{affine_to_intrinsic_linear}, we first need the ``na{\"i}ve'' extension of point~(ii) of Lemma~2.25 in~\cite{Smi14} or Lemma~5.8 in~\cite{Smi16}, giving a bound on the \emph{affine} contraction strength of the \emph{linear} part of~$g$ (seen as an element of~$G \ltimes V$ by the usual embedding).

\begin{lemma}
\label{affine_to_vector_weak}
For any $\rho$-regular map $g \in G \ltimes V$, we have $s_{X_0}(\ell(g)) \leq s_{X_0}(g)$.
\end{lemma}
The proof is very similar to the one given in~\cite{Smi14}.
\begin{proof} 
We have
\begin{align*}
s_{X_0}(\ell(g))
  &= \left\| \restr{\ell(g)}{V^{\subll}_{\ell(g)}}  \right\|
     \left\| \restr{\ell(g)^{-1}}{A^{\subgap}_{\ell(g)}} \right\| \\
  &= \left\| \restr{g}{V^{\subll}_{g}}  \right\|
     \max\left( \left\| \restr{g^{-1}}{V^{\subgap}_{g}} \right\|, 1 \right) \\
  &\leq \left\| \restr{g}{V^{\subll}_{g}}  \right\|
        \max\left( \left\| \restr{g^{-1}}{A^{\subgap}_{g}} \right\|, 1 \right) \\
  &= s_{X_0}(g).
\end{align*}
To justify the last equality, note that if $\phi$ is a canonizing map for~$g$, then the space $A^{\subgap}_{g} = \phi^{-1}(A^\subge_0)$ contains the subspace $\phi^{-1}(V^0)$, which is nonzero by Assumption~\ref{zero_is_a_weight}, is clearly $g$-invariant, and all eigenvalues of $g^{-1}$ restricted to that subspace have modulus~1; hence $\left\| \restr{g^{-1}}{A^{\subgap}_{g}} \right\|  \geq \left\| \restr{g^{-1}}{\phi^{-1}(V^0)} \right\| \geq 1$.
\end{proof}

We may now state and prove the appropriate generalization of Lemma~2.25 in~\cite{Smi14} and Lemma~5.8 in~\cite{Smi16}, linking in all three points the purely linear properties with the affine properties. Recall that $\vec{s}_{X_0}$ stands for the linear $X$-contraction strength (with $X = X_0$) as defined in Definition~\ref{vec_s_definition}, as opposed to $s_{X_0}$, which stands for the affine contraction strength.

\begin{proposition}
\label{affine_to_intrinsic_linear}
For every $C \geq 1$, there is a positive constant $s_{\ref{affine_to_intrinsic_linear}}(C)$ with the following property. Let $(g, h)$ be a $C$-non-degenerate pair of $\rho$-regular elements of~$G \ltimes V$. In this case:
\begin{hypothenum}
\item the pair $(\ell(g), \ell(h))$ is $C'$-non-degenerate (in the sense of Definition~\ref{linear-C-non-deg-definition}), for some constant~$C'$ that depends only on~$C$;
\item we have
\[\vec{s}_{X_0}(\ell(g)) \lesssim_C s_{X_0}(g);\]
\item moreover, if we assume that $s_{X_0}(g^{-1}) \leq s_{\ref{affine_to_intrinsic_linear}}(C)$, then we actually have
\[\vec{s}_{X_0}(\ell(g)) \left\| \restr{g^{-1}}{A^\subap_g} \right\| \lesssim_C s_{X_0}(g).\]
\end{hypothenum}
\end{proposition}
\begin{proof}~
\begin{hypothenum}
\item By Remark~\ref{choice_of_nu_does_not_matter}, we lose no generality in specifying some particular form for the map~$\nu$ introduced in Definition~\ref{nu_definition}. Let us define~$\nu$ in a way which is consistent with the definition of $C$-non-degeneracy in the affine case, namely using a particular case of the formula~\eqref{eq:nu_example}:
\begin{equation}
\label{eq:nu_compatible_definition}
\nu(\phi) := \max(\|\rho(\phi)\|,\; \|\rho(\phi)^{-1}\|),
\end{equation}
where $\rho$ is our working representation, and $\|\bullet\|$ is the Euclidean norm we introduced in the beginning of section~\ref{sec:regular_metric} (or more precisely its restriction to~$V$, which is just the Euclidean norm introduced in Lemma~\ref{K-invariant}). Clearly if~$\phi$ is a canonizing map for~$g$, then $\ell(\phi)$ is a canonizing map for~$\ell(g)$. Since obviously for any $\phi \in G \ltimes V$, we have $\|\ell(\phi)\| \leq \|\phi\|$, the conclusion follows.
\item First note that by Lemma~\ref{affine_to_vector_weak}, we have
\begin{equation}
\label{eq:slg_sg}
s_{X_0}(\ell(g)) \leq s_{X_0}(g).
\end{equation}
Now apply Proposition~\ref{contraction_strength_and_cartan_projection}: passing to the exponential, we have
\begin{equation}
\label{eq:slg_Ctg}
s_{X_0}(\ell(g)) \asymp_C
\exp\left( \max_{\lambda \in \Omega^\subl_{X_0}} \lambda(\cartan(\ell(g))) \right)
\exp\left( -\min_{\lambda \in \Omega^\subge_{X_0}} \lambda(\cartan(\ell(g))) \right).
\end{equation}
Now the key point is Lemma~\ref{de_part_et_d_autre}~\ref{itm:weak}: for every $i \in \Pi \setminus \Pi_{X_0}$, we may find $\lambda^\subl_i \in \Omega^\subl_{X_0}$ and $\lambda^\subge_i \in \Omega^\subge_{X_0}$ whose difference is~$\alpha_i$. It immediately follows that
\begin{equation}
\exp\left( \max_{\lambda \in \Omega^\subl_{X_0}} \lambda(\cartan(\ell(g))) \right)
\exp\left( -\min_{\lambda \in \Omega^\subge_{X_0}} \lambda(\cartan(\ell(g))) \right)
\geq \vec{s}_{X_0}(\ell(g)).
\end{equation}

\item We proceed in two steps: we establish first \eqref{eq:sg_omegasubl_normagsubap} (which is straightforward), then \eqref{eq:sxlg_omegasubl} (which relies on the strong version~\ref{itm:strong} of Lemma~\ref{de_part_et_d_autre} --- in fact, this is the only place where the strong version is needed). 
\begin{itemize}
\item We have, by definition:
\[s_{X_0}(g) =
    \left\| \restr{g}{V^{\subll}_{g}}  \right\|
    \left\| \restr{g^{-1}}{A^{\subgap}_{g}} \right\|.\]
Let $\phi$ be an optimal canonizing map for $g$. Since $g$ is $C$-non-degenerate (and the images $\phi(A^{\subap}_{g})$ and $\phi(V^{\subgg}_{g})$, being respectively equal to $A^\sube_0$ and $V^\subg_0$, are orthogonal), it follows that
\[s_{X_0}(g) \asymp_C
    \left\| \restr{g}     {V^{\subll}_{g}}  \right\|
    \max \left(
    \left\| \restr{g^{-1}}{A^{\subap}_{g}}  \right\|,
    \left\| \restr{g^{-1}}{V^{\subgg}_{g}}  \right\|
    \right).\]
Clearly we have $\left\| \restr{g^{-1}}{A^{\subap}_{g}} \right\| \geq \left\| \restr{g^{-1}}{\phi^{-1}(V^0)} \right\| = 1$. On the other hand, since $s_{X_0}(g^{-1}) \leq 1$, we have $\left\| \restr{g^{-1}}{V^{\subgg}_{g}}  \right\| \leq 1$. It follows that
\begin{equation}
\label{eq:sg_simplified}
s_{X_0}(g) \asymp_C
    \left\| \restr{g}     {V^{\subll}_{g}} \right\|
    \left\| \restr{g^{-1}}{A^{\subap}_{g}} \right\|.
\end{equation}
By methods similar to the proof of Proposition~\ref{contraction_strength_and_cartan_projection} (essentially by Proposition~\ref{eigenvalues_and_singular_values_characterization}), we may rewrite this as
\begin{equation}
\label{eq:sg_omegasubl_normagsubap}
s_{X_0}(g) \asymp_C
    \exp\left( \max_{\lambda \in \Omega^\subl_{X_0}} \lambda(\cartan(\ell(g))) \right)
    \left\| \restr{g^{-1}}{A^{\subap}_{g}} \right\|.
\end{equation}

\item Note that, combining Proposition~\ref{contraction_strength_and_cartan_projection} with Lemma~\ref{affine_to_vector_weak}, we have
\begin{align}
\exp\left( \max_{\lambda \in \Omega^\suble_{X_0}} \lambda(\cartan(\ell(g))) \right)
\exp\left( -\min_{\lambda \in \Omega^\subg_{X_0}} \lambda(\cartan(\ell(g))) \right)
&\lesssim_C
s_{X_0}(\ell(g)^{-1}) \nonumber \\
&\leq
s_{X_0}(g^{-1}).
\end{align}

If we take $s_{\ref{affine_to_intrinsic_linear}}(C)$ equal to the inverse of the implicit constant in that inequality, we may assume that
\begin{equation}
\forall \lambda^\suble \in \Omega^\suble_{X_0},\;
\forall \lambda^\subg \in \Omega^\subg_{X_0},\quad
\lambda^\suble(\cartan(\ell(g))) \leq \lambda^\subg(\cartan(\ell(g)));
\end{equation}
in particular, since $0 \in \Omega^\suble_{X_0}$, we then have
\begin{equation}
\forall \lambda \in \Omega^\subg_{X_0},\quad
\lambda(\cartan(\ell(g))) \geq 0.
\end{equation}
Obviously this remains true for $\lambda \in \Omega^\subg_{X_0} \cup \{0\}$. Now take some $i \in \Pi \setminus \Pi_{X_0}$; then we have
\begin{align}
-\alpha_i(\cartan(\ell(g)))
  &= \lambda^\subl_i(\cartan(\ell(g))) - \lambda^\subge_i(\cartan(\ell(g))) \nonumber \\
  &\leq \lambda^\subl_i(\cartan(\ell(g))),
\end{align}
where $\lambda^\subge_i \in \Omega^\subg_{X_0} \cup \{0\}$ and~$\lambda^\subl_i \in \Omega^\subl_{X_0}$ are the restricted weights introduced in Lemma~\ref{de_part_et_d_autre}~\ref{itm:strong}. This implies that
\begin{align}
\label{eq:sxlg_omegasubl}
\vec{s}_{X_0}(\ell(g)) &:= \exp\left( \max_{\alpha \in \Pi \setminus \Pi_{X_0}} -\alpha(\cartan(\ell(g))) \right) \nonumber \\
  &\leq \exp\left( \max_{\lambda \in \Omega^\subl_{X_0}} \lambda(\cartan(\ell(g))) \right).
\end{align}
\end{itemize}
Combining the two estimates, the conclusion follows. \qedhere
\end{hypothenum}
\end{proof}

\section{Products of $\rho$-regular maps}
\label{sec:quantitative_properties_of_products}


The goal of this section is to prove Proposition~\ref{regular_product}, which is the ``main part'' of the affine version of Schema~\ref{proposition_template}. The general strategy is the same as in Section~7 in~\cite{Smi16} or in Section~3.2 in~\cite{Smi14}: we reduce the problem to Proposition~\ref{proximal_product}, by considering the action of~$G \ltimes V$ on a suitable exterior power $\ext^{p} A$ (rather than on the spaces~$V_i$ as in Section~\ref{sec:linear_regular_product}).

We start by proving the following result, whose role in~\cite{Smi16} was played by Proposition~6.17. Even though the new version involves slightly different inequalities, the proof is quite similar.

\begin{proposition}
\label{regular_product_qualitative}
For every $C \geq 1$, there is a positive constant $s_{\ref{regular_product_qualitative}}(C) \leq 1$ with the following property. Take any $C$-non-degenerate pair $(g, h)$ of $\rho$-regular maps in~$G \ltimes V$ such that $s_{X_0}(g) \leq s_{\ref{regular_product_qualitative}}(C)$ and $s_{X_0}(h) \leq s_{\ref{regular_product_qualitative}}(C)$. Then $gh$ is asymptotically contracting along~$X_0$.
\end{proposition}
\begin{proof}
The main idea is as follows: the hypotheses can be reformulated as constraints on~$\cartan(g)$ and~$\cartan(h)$, while the desired conclusion is a condition on~$\jordan(gh)$. We then use Corollary~\ref{jordan_additivity_reformulation} to link the latter with the former.

Let $C \geq 1$, and let $(g, h)$ be a $C$-non-degenerate pair of $\rho$-regular maps in $G \ltimes V$, such that
\[s_{X_0}(g) \leq s_{\ref{regular_product_qualitative}}(C) \quad\text{ and }\quad s_{X_0}(h) \leq s_{\ref{regular_product_qualitative}}(C)\]
for some constant $s_{\ref{regular_product_qualitative}}(C)$ to be specified later.

The first thing to note is that since the property of being $\rho$-regular depends only on the linear part, Lemma~\ref{affine_to_vector_weak} and Proposition~\ref{affine_to_intrinsic_linear}~(i) reduce the problem to the case where $g, h \in G$.

Let $\cartan'(g, h)$ be the vector defined in Corollary~\ref{jordan_additivity_reformulation}. Then we deduce from~\eqref{eq:ct'gh_ctg_cth} that for every pair of restricted weights $\lambda^\subge \in \Omega^\subge_{X_0}$ and~$\lambda^\subl \in \Omega^\subl_{X_0}$, we have:
\begin{align}
\label{eq:ljgh-ljg-ljh_estimate}
|(\lambda^\subge - \lambda^\subl)(\cartan'(g, h) - \cartan(g) - \cartan(h))|
  &\;\leq\; \|\lambda^\subge - \lambda^\subl\|\|\cartan'(g, h) - \cartan(g) - \cartan(h)\| \nonumber \\
  &\;\leq\; 2 \left( \max_{\lambda \in \Omega} \|\lambda\| \right) k_{\ref{jordan_additivity_reformulation}}(C).
\end{align}

On the other hand, Proposition~\ref{contraction_strength_and_cartan_projection} gives us
\begin{align}
\max_{\lambda \in \Omega^\subl_{X_0}} \lambda(\cartan(g))
- \min_{\lambda \in \Omega^\subge_{X_0}} \lambda(\cartan(g))
  &\;\leq\; \log s_{X_0}(g) + k_{\ref{contraction_strength_and_cartan_projection}}(C).
\end{align}

Taking $s_{\ref{regular_product_qualitative}}(C)$ small enough, we may assume that
\begin{equation}
\label{eq:ljg_estimate}
\forall \lambda^\subge \in \Omega^\subge_{X_0},\; \forall \lambda^\subl \in \Omega^\subl_{X_0},\quad
(\lambda^\subge - \lambda^\subl)(\cartan(g)) > \left( \max_{\lambda \in \Omega} \|\lambda\| \right) k_{\ref{jordan_additivity_reformulation}}(C).
\end{equation}
Of course a similar estimate then holds for~$h$:
\begin{equation}
\label{eq:ljh_estimate}
\forall \lambda^\subge \in \Omega^\subge_{X_0},\; \forall \lambda^\subl \in \Omega^\subl_{X_0},\quad
(\lambda^\subge - \lambda^\subl)(\cartan(h)) > \left( \max_{\lambda \in \Omega} \|\lambda\| \right) k_{\ref{jordan_additivity_reformulation}}(C).
\end{equation}

Adding together the three estimates \eqref{eq:ljgh-ljg-ljh_estimate}, \eqref{eq:ljg_estimate} and \eqref{eq:ljh_estimate}, we find that for every such pair, we have
\begin{equation}
\label{eq:lambda_cartan_gh_less0}
(\lambda^\subge - \lambda^\subl)(\cartan'(g, h)) > 0.
\end{equation}

On the other hand, we have~\eqref{eq:cartan_in_jordan_convex_hull} which says that
\[\jordan(gh) \in \Conv(W_{X_0} \cdot \cartan'(g, h)).\]
Now take any~$w \in W_{X_0}$. From Proposition~\ref{stabg=stabge}, it then follows that $w^{-1}$~stabilizes both~$\Omega^\subge_{X_0}$ and~$\Omega^\subl_{X_0}$; hence we still have
\begin{align}
(\lambda^\subge - \lambda^\subl)(w(\cartan'(g, h)))
&= (w^{-1}(\lambda^\subge) - w^{-1}(\lambda^\subl))(\cartan'(g, h)) \nonumber \\
&> 0.
\end{align}
Thus the difference $\lambda^\subge - \lambda^\subl$ takes positive values on every point of the orbit~$W_{X_0} \cdot \cartan'(g, h)$; hence it also takes positive values on every point of its convex hull. In particular, we have
\begin{equation}
(\lambda^\subge - \lambda^\subl)(\jordan(gh)) > 0.
\end{equation}
We conclude that $gh$ is indeed asymptotically contracting along~$X_0$.
\end{proof}

%
%

We now establish the correspondence between affine and proximal properties. We introduce the integers:
\begin{align}
p &:= \dim A^\subge_0 = \dim V^\subge_0 + 1; \nonumber \\
q &:= \dim V^\subl_0; \\
d &:= \dim A = \dim V + 1 = q+p. \nonumber
\end{align}
For every $g \in G \ltimes V$, we may define its exterior power $\ext^p g: \ext^p A \to \ext^p A$. The Euclidean structure of $A$ induces in a canonical way a Euclidean structure on $\ext^p A$.

\begin{lemma} \mbox{ }
\label{regular_to_proximal}
\begin{hypothenum}
\item Let $g \in G \ltimes V$ be a map asymptotically contracting along~$X_0$. Then $\ext^{p} g$ is proximal, and the attracting (resp. repelling) space of $\ext^{p} g$ depends on nothing but $A^{\subgap}_{g}$ (resp.~$V^{\subll}_{g}$):
\[\begin{cases}
E^s_{\ext^{p} g} = \ext^{p} A^{\subgap}_{g} \\
E^u_{\ext^{p} g} = \setsuch{x \in \ext^{p} A}
                                        {x \wedge \ext^{q} V^{\subll}_{g} = 0}.
\end{cases}\]
\item For every $C \geq 1$, whenever $(g_1, g_2)$ is a $C$-non-degenerate pair of $\rho$-regular maps that are also asymptotically contracting along~$X_0$, $(\ext^p g_1, \ext^p g_2)$ is a $C^p$-non-degenerate pair of proximal maps.
\item For every $C \geq 1$, there is a constant~$s_{\ref{regular_to_proximal}}(C) < 1$ with the following property. For every $C$-non-degenerate $\rho$-regular map $g \in G \ltimes V$ that is also asymptotically contracting along~$X_0$, we have
\begin{equation}
s_{X_0}(g) \lesssim_C \tilde{s}(\ext^{p} g).
\end{equation}
If in addition $s_{X_0}(g) \leq s_{\ref{regular_to_proximal}}(C)$, we have
\begin{equation}
s_{X_0}(g) \asymp_C \tilde{s}(\ext^{p} g).
\end{equation}
(Recall that $\tilde{s}$ and~$s_{X_0}$ stand respectively for the proximal and affine contraction strengths: see Definitions~\ref{s_definition} and~\ref{s_tilde_definition}.)
\item For any two $p$-dimensional subspaces $A_1$ and $A_2$ of $A$, we have
\[\alpha^\mathrm{Haus}(A_1, A_2)
\;\asymp\; \alpha \left( \ext^{p} A_1,\; \ext^{p} A_2 \right).\]
\end{hypothenum}
\end{lemma}

This is similar to Lemma~3.8 in~\cite{Smi14} and to Lemma~7.2 in~\cite{Smi16}; but now, there is the additional complication of needing to distinguish between $\rho$-regularity and asymptotic contraction.

The proofs of points~(ii) and~(iii) however still remain very similar to the corresponding proofs in~\cite{Smi14}. We chose to reproduce them here, in particular in order to correct a small mistake: in~\cite{Smi14} we erroneously claimed that we could take $s_{\ref{regular_to_proximal}}(C) = 1$, which stemmed from a confusion between $g$ and its canonized version~$g'$.

\begin{proof}~
\begin{hypothenum}
\item Let $g \in G \ltimes V$ be a map asymptotically contracting along~$X_0$. From Proposition~\ref{eigenvalues_and_singular_values_characterization} (as already noted in the proof of Proposition~\ref{contraction_strength_limit}), it follows that
\begin{align}
r \left( \restr{g}{V^\subgg_g} \right)
  &< r \left( \restr{g^{-1}}{A^\subgap_g} \right)^{-1},
\end{align}
\ie every eigenvalue of~$\restr{g}{V^\subgg_g}$ is smaller in modulus than every eigenvalue of~$\restr{g}{A^\subgap_g}$. Let $\lambda_{1}, \ldots, \lambda_{d}$ be the eigenvalues of~$g$ (acting on $A$), counted with multiplicity and ordered by nondecreasing modulus; we then have
\begin{equation}
|\lambda_{q}| = r \left( \restr{g}{V^\subgg_g} \right) < r \left( \restr{g^{-1}}{A^\subgap_g} \right)^{-1} = |\lambda_{q+1}|.
\end{equation}
On the other hand, we know that the eigenvalues of~$\ext^{p} g$ counted with multiplicity are exactly the products of the form $\lambda_{i_1}\cdots\lambda_{i_{p}}$, where $1 \leq i_1 < \cdots < i_{p} \leq d$. As the two largest of them (by modulus) are $\lambda_{q+1} \cdots \lambda_{d}$ and $\lambda_{q}\lambda_{q+2} \cdots \lambda_{d}$, it follows that $\ext^{p} g$ is proximal.

%
%
As for the expression of $E^s$ and $E^u$, it follows immediately by considering a basis that trigonalizes~$g$.

\item Take any pair $(i, j) \in \{1, 2\}^2$. Let $\phi$ be an optimal canonizing map for the pair $(A^\subgap_{g_i}, A^\sublap_{g_j})$. Then we have $\phi(A^\subgap_{g_i}) = A^\subge_0$ and (by Proposition~\ref{uniquely_determined}) $\phi(V^\subll_{g_j}) = V^\subl_0$. In the Euclidean structure we have chosen, $A^\subge_0$ is orthogonal to $V^\subl_0$; hence $\ext^p A^\subge_0$ is orthogonal to the hyperplane~$\setsuch{x \in \ext^{p} A}{x \wedge \ext^{q} V^\subl_0 = 0}$. By the previous point, it follows that $\ext^p \phi$ is a canonizing map for the pair $(E^s_{\ext^p g_i}, E^u_{\ext^p g_j})$. As $\|\ext^p \phi\| \leq \|\phi\|^p$ and similarly for $\phi^{-1}$, the conclusion follows.

\item Let $C \geq 1$, and let $g \in G \ltimes V$ be a $C$-non-degenerate $\rho$-regular map that is also asymptotically contracting along~$X_0$. Let $\phi$ be an optimal canonizing map for $g$, and let $g' = \phi g \phi^{-1}$. Then it is clear that $s_{X_0}(g') \asymp_C s_{X_0}(g)$ and $\tilde{s}(\ext^p g') \asymp_C \tilde{s}(\ext^p g)$; so it is sufficient to prove the statement for~$g'$.

Let $s_1 \leq \ldots \leq s_{p}$ (resp. $s'_1 \leq \ldots \leq s'_q$) be the singular values of $g'$ restricted to~$A^{\subgap}_{g'} = A^\subge_0$ (resp. to~$V^{\subll}_{g} = V^\subl_0$), so that $\left\| \restr{g'^{-1}}{A^\subge_0} \right\| = s_1^{-1}$ and $\left\| \restr{g'}{V^{\subl}_0} \right\| = s'_q$. Since the spaces $A^{\subge}_0$ and $V^{\subl}_0$ are stable by~$g'$ and orthogonal, we get that the singular values of $g'$ on the whole space $A$ are
\[s'_1, \ldots, s'_q, s_1, \ldots, s_{p}\]
(note however that unless~$s_{X_0}(g') \leq 1$, this list may fail to be sorted in nondecreasing order.) On the other hand, we know that the singular values of $\ext^{p} g'$ are products of $p$ distinct singular values of $g'$. Since $E^s_{\ext^{p} g'}$ is orthogonal to $E^u_{\ext^{p} g'}$, we may once again analyze the singular values separately for each subspace. We know that the singular value corresponding to $E^s$ is equal to $s_1 \cdots s_{p}$; we deduce that $\left\| \restr{\ext^{p} g}{E^u} \right\|$ is equal to the maximum of the remaining singular values. In particular it is larger than or equal to $s'_q \cdot s_2 \cdots s_{p}$. On the other hand, if $\lambda$ is the largest eigenvalue of $\ext^{p} g'$, then we have
\[|\lambda| = \left| \lambda_{q+1} \cdots \lambda_{d} \right|
            = \left| \det (\restr{g'}{A^{\subge}_0}) \right|
            = s_1 \cdots s_p\]
(where $\lambda_1, \ldots, \lambda_{d}$ are the eigenvalues of~$g'$ or equivalently of~$g$, sorted by nondecreasing modulus). The second equality holds because $g$, hence~$g'$, is asymptotically contracting along~$X_0$, so that its \emph{eigen}values \emph{are} sorted in the ``correct'' order. It follows that:
\begin{equation}
\label{eq:s'_lower_bound}
\tilde{s}(\ext^{p} g')
   = \frac{\left\| \restr{\ext^{p} g'}{E^u_{\ext^{p} g'}} \right\|}
          {|\lambda|}
   \geq \frac{s'_q \cdot s_2 \cdots s_{p}}{s_1 \cdots s_{p}}
   = s'_q s_1^{-1}
   = s_{X_0}(g'),
\end{equation}
which is the first estimate we were looking for.

Now if we take $s_{\ref{regular_to_proximal}}(C)$ small enough, we may suppose that $s_{X_0}(g') \leq 1$. Then we have $s'_q \leq s_1$, which means that the singular values of $\ext^{p} g'$ are indeed sorted in the ``correct'' order. Hence $s'_q \cdot s_2 \cdots s_{p}$ is actually the largest singular value of $\restr{\ext^{p} g'}{E^u}$, and the inequality becomes an equality: $\tilde{s}(\ext^{p} g') = s_{X_0}(g')$. The second estimate follows.

\item See Lemma~3.8~(iv) in~\cite{Smi14}. \qedhere
\end{hypothenum}
\end{proof}

We also need the following technical lemma, which generalizes Lemma~3.9 in~\cite{Smi14} and Lemma~7.3 in~\cite{Smi16}:

\begin{lemma}
\label{continuity_of_non_degeneracy}
There is a constant $\eps > 0$ with the following property. Let $A_1, A_2$ be any two affine parabolic spaces such that
\[\begin{cases}
\alpha^\mathrm{Haus}(A_1, A^\subge_0) \leq \eps \\
\alpha^\mathrm{Haus}(A_2, A^\suble_0) \leq \eps.
\end{cases}\]
Then they form a $2$-non-degenerate pair.
\end{lemma}
(Of course the constant~2 is arbitrary; we could replace it by any number larger than~1.)
\begin{proof}
The proof is exactly the same as the proof of Lemma~3.9 in~\cite{Smi14}, \emph{mutatis mutandis}.
\end{proof}

\begin{proposition}
\label{regular_product}
For every $C \geq 1$, there is a positive constant $s_{\ref{regular_product}}(C) < 1$ with the following property. Take any $C$-non-degenerate pair $(g, h)$ of $\rho$-regular maps in~$G \ltimes V$; suppose that we have $s_{X_0}(g^{\pm 1}) \leq s_{\ref{regular_product}}(C)$ and $s_{X_0}(h^{\pm 1}) \leq s_{\ref{regular_product}}(C)$. Then $gh$ is still $\rho$-regular, $2C$-non-degenerate, and we have:
\begin{hypothenum}
\item $\begin{cases}
\alpha^\mathrm{Haus} \left(V^{\subgap}_{gh},\; V^{\subgap}_{g} \right) \lesssim_C \vec{s}_{X_0}(\ell(g)) \vspace{1mm} \\
\alpha^\mathrm{Haus} \left(V^{\sublap}_{gh},\; V^{\sublap}_{h} \right) \lesssim_C \vec{s}_{X_0}(\ell(h))
\end{cases}$;
\item $\begin{cases}
\alpha^\mathrm{Haus} \left(A^{\subgap}_{gh},\; A^{\subgap}_{g} \right) \lesssim_C s_{X_0}(g) \vspace{1mm} \\
\alpha^\mathrm{Haus} \left(A^{\sublap}_{gh},\; A^{\sublap}_{h} \right) \lesssim_C s_{X_0}(h^{-1})
\end{cases}$;
\item $s_{X_0}(gh) \lesssim_C s_{X_0}(g)s_{X_0}(h)$.
\end{hypothenum}
\end{proposition}
Points~(ii) and~(iii) are a generalization of Proposition~3.6 in~\cite{Smi14} and Proposition~7.4 in~\cite{Smi16}; the proof is very similar. Together they give the ``main part'' of the affine version of Schema~\ref{proposition_template}.

As for point~(i), it generalizes Corollary~3.7 in~\cite{Smi14} and Corollary~7.5 in~\cite{Smi16}. But its statement is now stronger, as it involves the \emph{linear} contraction strength $\vec{s}_{X_0}$. As such, it can no longer be obtained as a corollary of the affine version; instead it must be proved independently, using Proposition~\ref{intrinsic_regular_product}.

\begin{remark}
Note that point~(ii) involves~$h^{-1}$, but in point~(i) we have simply written~$h$ instead. In fact in the linear case, the distinction between~$h$ and~$h^{-1}$ becomes irrelevant, as they both have the same \emph{linear} contraction strength $\vec{s}_{X_0}$ by Proposition~\ref{intrinsic_inverse}~\ref{itm:w0_contraction}. However their \emph{affine} contraction strengths $s_{X_0}$ \emph{can} be different.
\end{remark}

\begin{proof}[Proof of Proposition~\ref{regular_product}]
Let us fix some constant $s_{\ref{regular_product}}(C)$, small enough to satisfy all the constraints that will appear in the course of the proof. Let $(g, h)$ be a pair of maps satisfying the hypotheses.

First of all note that if we assume that~$s_{\ref{regular_product}}(C) \leq 1$, then Proposition~\ref{contraction_strength_limit}~(i) ensures that $g$, $h$, $g^{-1}$ and~$h^{-1}$ are all asymptotically contracting along~$X_0$.

\begin{itemize}
\item Let us prove~(i). By Proposition~\ref{asymptotically_contracting_to_regular}~(i), it follows that $g$ and~$h$, hence $\ell(g)$ and~$\ell(h)$, are $X_0$-regular. Proposition~\ref{affine_to_intrinsic_linear}~(i) and~(ii) then implies that the pair $(\ell(g), \ell(h))$ satisfies the hypotheses of Proposition~\ref{intrinsic_regular_product}. Hence we have
\begin{equation}\begin{cases}
\delta \left(y^{X_0, +}_{\ell(gh)},\; y^{X_0, +}_{\ell(g)} \right) \lesssim_C \vec{s}_{X_0}(\ell(g)); \vspace{1mm} \\
\delta \left(y^{X_0, -}_{\ell(gh)},\; y^{X_0, -}_{\ell(h)} \right) \lesssim_C \vec{s}_{X_0}(\ell(h)).
\end{cases}\end{equation}
Now remember (Remark~\ref{affine_flag_varieties}) that the attracting (resp. repelling) flag of a map $g \in G$ carries the same information as its (linear) ideally expanding (resp. contracting) space. Even more precisely, we may deduce from Proposition~\ref{stabg=stabge} that the orbital map from~$G$ to the orbit of~$V^\subge_0$ in the $(p-1)$-dimensional Grassmanian of~$V$ descends to a smooth embedding of the flag variety $G/P^+_{X_0}$. Since the flag variety is compact, the embedding is in particular Lipschitz-continuous. The desired inequalities follow.

\item If we take $s_{\ref{regular_product}}(C) \leq s_{\ref{regular_product_qualitative}}(C)$, then Proposition~\ref{regular_product_qualitative} tells us that $gh$~is asymptotically contracting along~$X_0$. We may also apply Proposition~\ref{regular_product_qualitative} to the pair $(h^{-1}, g^{-1})$: so $(gh)^{-1}$ is also asymptotically contracting along~$X_0$. In other terms, $gh$~is actually compatible with~$X_0$ (by Proposition~\ref{asymp_cont_vs_compat}). By Proposition~\ref{asymptotically_contracting_to_regular}~(ii), we deduce that $gh$~is $\rho$-regular.

The remainder of the proof works exactly like the proof of Proposition~3.6 in~\cite{Smi14} or of Proposition~7.4 in~\cite{Smi16}, namely by applying Proposition~\ref{proximal_product} to the maps $\gamma_1 = \ext^{p} g$ and~$\gamma_2 = \ext^{p} h$. Let us check that $\gamma_1$ and~$\gamma_2$ satisfy the required hypotheses:
\begin{itemize}
\item By Lemma~\ref{regular_to_proximal}~(i), $\gamma_1$ and $\gamma_2$ are proximal.
\item By Lemma~\ref{regular_to_proximal}~(ii), the pair $(\gamma_1, \gamma_2)$ is $C^p$-non-degenerate.
\item If we choose $s_{\ref{regular_product}}(C) \leq s_{\ref{regular_to_proximal}}(C)$, it follows by Lemma~\ref{regular_to_proximal}~(iii) that $\tilde{s}(\gamma_1) \lesssim_C s_{X_0}(g)$ and $\tilde{s}(\gamma_2) \lesssim_C s_{X_0}(h)$. If we choose $s_{\ref{regular_product}}(C)$ sufficiently small, then $\gamma_1$ and $\gamma_2$ are $\tilde{s}_{\ref{proximal_product}}(C^p)$-contracting, \ie sufficiently contracting to apply Proposition~\ref{proximal_product}.
\end{itemize}
Thus we may apply Proposition~\ref{proximal_product}. It remains to deduce from its conclusions the conclusions of Proposition~\ref{regular_product}.
\begin{itemize}
\item From Proposition~\ref{proximal_product}~(i), using Lemma~\ref{regular_to_proximal} (i), (iii) and~(iv), we get
\[\alpha^\mathrm{Haus} \left(A^{\subgap}_{gh},\; A^{\subgap}_{g} \right)
  \lesssim_C s_{X_0}(g),\]
which shows the first line of Proposition~\ref{regular_product}~(ii).
\item By applying Proposition~\ref{proximal_product} to $\gamma_2^{-1} \gamma_1^{-1}$ instead of $\gamma_1 \gamma_2$, we get in the same way the second line of Proposition~\ref{regular_product}~(ii).
\item Let $\phi$ be an optimal canonizing map for the pair $(A^{\subgap}_{g}, A^{\sublap}_{h})$. By hypothesis, we have $\left \|\phi^{\pm 1} \right\| \leq C$. But if we take $s_{\ref{regular_product}}(C)$ sufficiently small, the two inequalities that we have just shown, together with Lemma~\ref{continuity_of_non_degeneracy}, allow us to find a map $\phi'$ with $\|\phi'\| \leq 2$, $\|{\phi'}^{-1}\| \leq 2$ and
\[\phi' \circ \phi (A^{\subgap}_{gh}, A^{\sublap}_{gh}) = (A^\subge_0, A^\suble_0).\]
It follows that the composition map $gh$ is $2C$-non-degenerate.
\item The last inequality, namely Proposition~\ref{regular_product}~(iii), now is deduced from Proposition~\ref{proximal_product}~(ii) by using Lemma~\ref{regular_to_proximal}~(iii). \qedhere
\end{itemize}
\end{itemize}
\end{proof}

\section{Additivity of Margulis invariants}
\label{sec:additivity}

The goal of this section is to prove Propositions \ref{invariant_inverse} and~\ref{invariant_additivity_only}, which explain how the Margulis invariant behaves under group operations (respectively inverse and composition). The latter gives the ``asymptotic dynamics'' part in the affine version of Schema~\ref{proposition_template}. They generalize the respective parts (i) and~(ii) of Proposition~8.1 in~\cite{Smi16}, or of Proposition~4.1 in~\cite{Smi14}. These two propositions are the key ingredients in the proof of the Main Theorem.

The two subsections are dedicated to these two propositions respectively.

\subsection{Margulis invariants of inverses}
\label{sec:invariant_inverse}

The following proposition generalizes Proposition~4.1~(i) in~\cite{Smi14} and Proposition~8.1~(i) in~\cite{Smi16}. The proof is similar, and fairly straightforward. Compare it also with the results of Section~\ref{sec:group_inverse}.

\begin{proposition}
\label{invariant_inverse}
For every $\rho$-regular map~$g \in G \ltimes V$, we have
\[M(g^{-1}) = -w_0(M(g)).\]
\end{proposition}
\begin{remark}
\label{Weyl_action_works}
Note that $M(g)$ is (by definition) an element of the space~$V^\transl_0$, which (again by definition) is the set of fixed points of~$L = Z_G(A)$. From this, it is straightforward to deduce that $V^\transl_0$ is invariant by~$N_G(A)$. Hence $w_0$ induces a linear involution on~$V^\transl_0$, which does not depend on the choice of a representative of~$w_0$ in~$G$.
\end{remark}
\begin{proof}
First of all note that if~$\phi$ is a canonizing map for~$g$, then $w_0 \phi$ is a canonizing map for~$g^{-1}$. Indeed, assume that $g = \tau_v \exp(\jordan(g)) g_e g_u$ is in canonical form; then we have $w_0(-\jordan(g)) \in \mathfrak{a}^+$ (since by definition $w_0$ exchanges the dominant Weyl chamber~$\mathfrak{a}^+$ and its negative~$-\mathfrak{a}^+$), and $w_0(-v) \in V^\sube_0$ (since $X_0$ is symmetric, the action of~$w_0$ preserves~$V^\sube_0$).

It remains to show that $w_0$, or more precisely any representative~$\tilde{w}_0 \in N_G(A)$ of~$w_0$, commutes with $\pi_{\transl}$. We use the well-known fact that the group~$W$, that we defined as the quotient~$N_G(A)/Z_G(A)$, is also equal to the quotient~$N_K(A)/Z_K(A)$ (see \cite{Kna96}, formulas (7.84a) and~(7.84b)); hence
\begin{equation}
\tilde{w}_0 \in N_G(A) = W Z_G(A) = W Z_K(A) A = N_K(A) A \subset K A.
\end{equation}
Now recall (\eqref{eq:V_eq_decomp}, \eqref{eq:V_a_decomp}) that we have an orthogonal decomposition
\[V^\sube_0 = 
\lefteqn{\overbrace{\phantom{
  V^\subenz_0 \oplus (V^\extra_0 \cap V^0)
}}^{\displaystyle V^\extra_0}}
  V^\subenz_0 \oplus \underbrace{
  (V^\extra_0 \cap V^0) \oplus V^\transl_0
}_{\displaystyle V^0}.\]
Let us show that all three components are (globally) invariant by~$\tilde{w}_0$:
\begin{itemize}
\item For $V^\subenz_0$ this is obvious (since $\Omega^{w_0} \setminus \{0\}$ is invariant by~$w_0$).
\item For $V^\transl_0$ this follows from Remark~\ref{Weyl_action_works}.
\item This is obviously the case for~$V^0$. Now by definition the group~$A$ acts trivially on~$V^0$, and by construction $K$~acts on~$V^0$ by orthogonal transformations (indeed the Euclidean structure was chosen in accordance with Lemma~\ref{K-invariant}); hence~$V^\extra_0 \cap V^0$, which is the orthogonal complement of~$V^\transl_0$ in~$V^0$, is also invariant by~$\tilde{w}_0$.
\end{itemize}

The desired formula now immediately follows from the definition of the Margulis invariant.
\end{proof}

\subsection{Margulis invariants of products}
\label{sec:invariant_additivity_only}

The following proposition generalizes Proposition~4.1~(ii) in~\cite{Smi14} and Proposition~8.1~(ii) in~\cite{Smi16}.

\begin{proposition}
\label{invariant_additivity_only}
For every $C \geq 1$, there are positive constants $s_{\ref{invariant_additivity_only}}(C) \leq 1$ and $k_{\ref{invariant_additivity_only}}(C)$ with the following property. Let $g, h \in G \ltimes V$ be a $C$-non-degenerate pair of $\rho$-regular maps, with $g^{\pm 1}$ and $h^{\pm 1}$ all $s_{\ref{invariant_additivity_only}}(C)$-contracting along~$X_0$. Then $gh$ is $\rho$-regular, and we have:
\[\|M(gh) - M(g) - M(h)\| \leq k_{\ref{invariant_additivity_only}}(C).\]
\end{proposition}

\begin{proof}[Outline of proof]
Since the proof is extremely technical, let us start with an informal (non-rigorous) outline. 

Recall that $M(gh)$ is defined by considering the quasi-translation induced by~$gh$ on its affine ideally neutral space~$A_{gh}^\subap$ (canonically identified with $A_0^\sube$), and taking its translation part along~$V_0^t$. So our goal is to undersand the action of~$gh$ on~$A_{gh}^\subap$.

The obvious first step is to decompose $gh$ into its factors. So we need to understand, on the one hand, the map induced by~$h$ from~$A_{gh}^\subap$ to its image (which is $A_{hg}^\subap$, since $hg$~is the conjugate of~$gh$ by~$h$); and on the other hand, the map induced by~$g$ from~$A_{hg}^\subap$ back to~$A_{gh}^\subap$.

To do this, let us schematically represent all the relevant subspaces: see Figure~\ref{fig:Margulis_schematic}. This is somewhat similar in spirit to Figure~2 from~\cite{AMS11} (and possibly other figures from that paper), but much more schematic.

\begin{figure}[h]
\[
\includegraphics{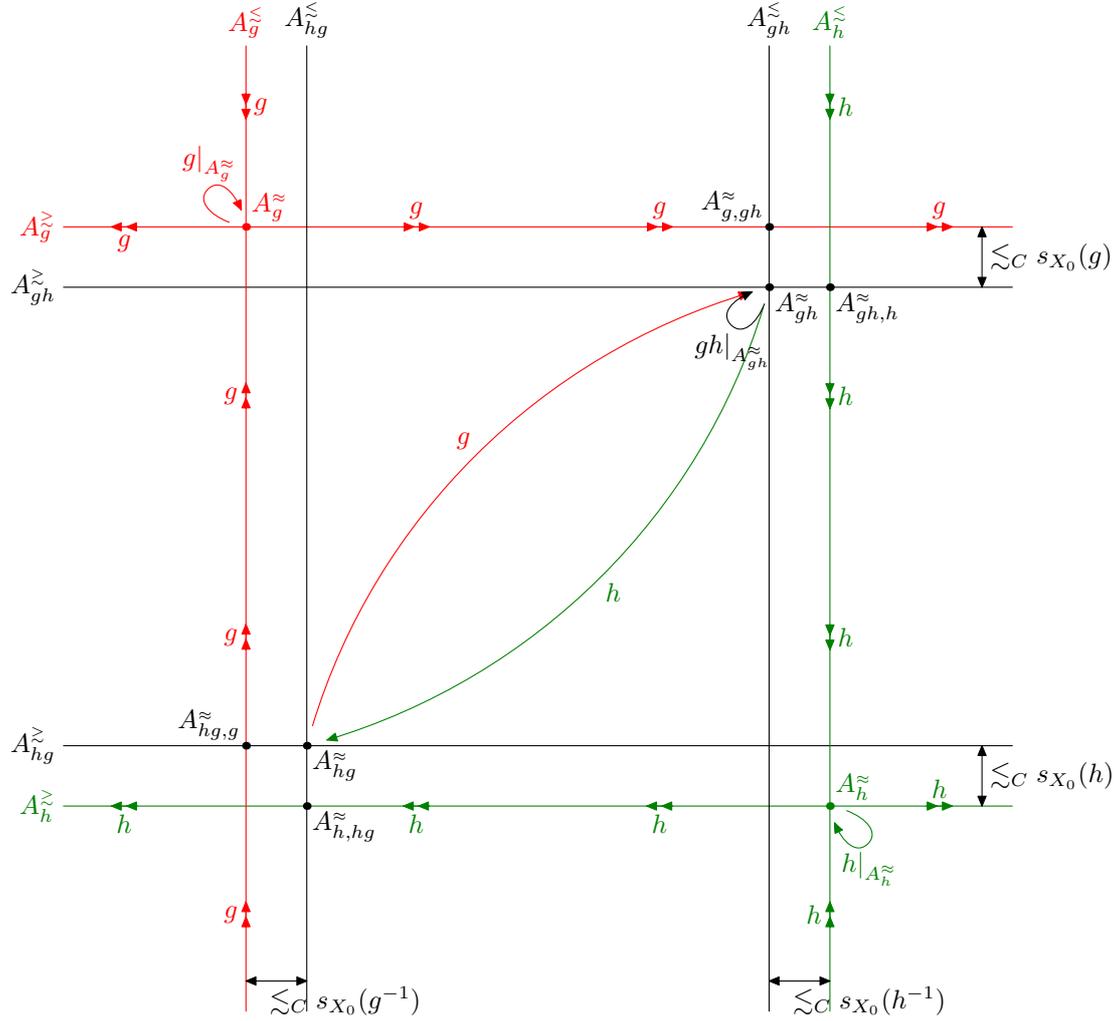}
\]
\caption{Relative positions of subspaces involved in the proof of Proposition~\ref{invariant_additivity_only}. This picture should be understood to represent some affine cross-section of the full picture. Lines represent (the corresponding cross-sections of) affine parabolic spaces (such as $A_g^\subgap$), and dots represent (the corresponding cross-sections of) intersections of transverse pairs of them (such as $A_g^\subap$).  We use the notational convention $A^{\subap}_{u, v} := A^{\subgap}_{u} \cap A^{\sublap}_{v}$ for any two $\rho$-regular maps $u$ and~$v$. Double-headed arrows represent contraction and expansion by the maps $g$ and~$h$. \\
We encourage the reader to watch this picture on a screen, so as to benefit from helpful (even if dispensable) color cues.}
\label{fig:Margulis_schematic}
\end{figure}
\afterpage{\clearpage} 

Note that (as indicated on the diagram) the four affine parabolic spaces corresponding to~$gh$ and~$hg$ lie very close to some of the affine parabolic spaces corresponding to~$g$ and~$h$: this is a consequence of Proposition~\ref{regular_product}~(ii).

Now the fact that $g$~is ``strongly contracting'' actually implies that it strongly contracts everything along the $V_g^\subll$ directions towards~$A_g^\subap$, but strongly expands everything along the $V_g^\subgg$ directions away from~$A_g^\subap$ (indicated by double-headed arrows on Figure~\ref{fig:Margulis_schematic}). Taking into account the relative positions of various spaces, we see that the map induced by~$g$ from~$A_{hg}^\subap$ to~$A_{gh}^\subap$ can then be approximated by the composition of the following five maps:
\begin{itemize}
\item some projection from~$A_{hg}^\subap$ to the space~$A_{hg, g}^\subap$ lying nearby;
\item the projection from~$A_{hg, g}^\subap$ to~$A_g^\subap$ parallel to~$V_g^\subll$;
\item the quasi-translation induced by~$g$ on~$A_g^\subap$;
\item the projection from~$A_{g}^\subap$ to~$A_{g, gh}^\subap$ parallel to~$V_g^\subgg$;
\item some projection from~$A_{g, gh}^\subap$ to the space~$A_{gh}^\subap$ lying nearby.
\end{itemize}
Finally we take advantage of the canonical identifications introduced in Section~\ref{sec:canonical} to identify all of these $A^\subap$ spaces with~$A_0^\sube$. Then by Lemma~\ref{projections_commute}, all the projections from one space to the other collapse to quasi-translations, that can be shown to be bounded (taking advantage of the non-degerenacy of the geometry). There remains only the quasi-translation corresponding to the action of~$g$ on~$A_g^\subgap$, whose projection to~$A_0^t$ is a translation by precisely~$M(g)$.

Similarly the map induced by~$h$ from~$A_{gh}^\subap$ to~$A_{hg}^\subap$ is close to a map whose conjugate by the canonical identifications is a quasi-translation of~$A_0^\sube$, whose projection to~$A_0^t$ is a translation by~$M(h)$. The conclusion then follows.

The proof splits into two parts:
\begin{itemize}
\item an ``algebraic'' part comprising Lemma~\ref{MgghgMg}, where we show that the map descibed above as an approximation for~$g: A_{hg}^\subap \to A_{gh}^\subap$ has a ``Margulis invariant'' close to that of~$g$ (this is basically a straightforward application of Lemma~\ref{projections_commute}); 
\item and an ``analytic'' part comprising Lemma~\ref{P1_almost_Id}, its Corollary~\ref{P1_almost_Id_cor} and Lemma~\ref{MghgMgghg}, where we estimate exactly how closely this map approximates $g: A_{hg}^\subap \to A_{gh}^\subap$. \noqed\qedhere 
\end{itemize}
\end{proof}

\begin{proof}[Comparison with earlier paper]
The same idea was used to prove Proposition~8.1~(ii) in~\cite{Smi16} (and Proposition~4.1~(ii) in~\cite{Smi14}). However we had to completely reorganize the proof, because of the following issue. In the proof of Lemma~8.6 in~\cite{Smi16}, a key point is that since the factors that are introduced to construct Diagram~4, namely $\ell(\overline{g_{gh}})$ and~$\ell(\overline{g_{g, gh}})$, are linear parts of quasi-translations, they automatically have bounded norm. But in the general case, this last deduction fails: see Remark~\ref{quasi-translation-generalization}.

This actually led us to develop a better (albeit still technical) proof: the separation between the ``algebraic'' and the ``analytic'' part was much less obvious in the previous papers. The new proof is also more symmetric: we got rid of the lopsided Diagram~4 from~\cite{Smi16}, and of the confusing series $\overline{P_1}$, $\overline{P_2}$, $\overline{P'_2}$ and~$\overline{P''_1}$. \noqed
\end{proof}

\begin{proof}
We start by doing some setup, culminating in the crucial decomposition~\eqref{eq:Mgh_first_decomposition} which reduces the problem to proving Lemmas \ref{MgghgMg} and~\ref{MghgMgghg}. The hard work will then be done while proving these lemmas.

Let $C \geq 1$. We choose some constant $s_{\ref{invariant_additivity_only}}(C) \leq 1$, small enough to satisfy all the constraints that will appear in the course of the proof. For the remainder of this section, we fix $g, h \in G \ltimes V$ a $C$-non-degenerate pair of $\rho$-regular maps such that $g^{\pm 1}$ and $h^{\pm 1}$ are $s_{\ref{invariant_additivity_only}}(C)$-contracting along~$X_0$.

The following remark will be used throughout this proof. \noqed
\end{proof}
\begin{remark}
\label{ggh_gh_2C}
We may suppose that the pairs $(A^{\subgap}_{gh}, A^{\sublap}_{gh})$, $(A^{\subgap}_{hg}, A^{\sublap}_{hg})$, $(A^{\subgap}_{g}, A^{\sublap}_{gh})$ and $(A^{\subgap}_{hg}, A^{\sublap}_{g})$ are all $2C$-non-degenerate.
Indeed, recall that (by Proposition~\ref{regular_product}), we have
\[\begin{cases}
  \alpha^\mathrm{Haus} \left(A^{\subgap}_{gh},\; A^{\subgap}_{g} \right) \lesssim_C s_{X_0}(g) \vspace{1mm} \\
  \alpha^\mathrm{Haus} \left(A^{\sublap}_{gh},\; A^{\sublap}_{h} \right) \lesssim_C s_{X_0}(h^{-1})
\end{cases}\]
and similar inequalities with $g$ and $h$ interchanged. On the other hand, by hypothesis, $(A^{\subgap}_{g}, A^{\sublap}_{h})$ is $C$-non-degenerate. If we choose $s_{\ref{invariant_additivity_only}}(C)$ sufficiently small, these four statements then follow from Lemma~\ref{continuity_of_non_degeneracy}.
\end{remark}
\begin{proof}[Proof of Proposition~\ref{invariant_additivity_only}, continued]
If we take $s_{\ref{invariant_additivity_only}}(C) \leq s_{\ref{regular_product}}(C)$, then Proposition~\ref{regular_product} ensures that $gh$ is $\rho$-regular.

To estimate $M(gh)$, we decompose the induced map~$gh: A^{\subap}_{gh} \to A^{\subap}_{gh}$ into a product of several maps, as announced in the proof outline. We refer the reader to Figure~\ref{fig:Margulis_schematic} for a better intuition of this decomposition.

\begin{itemize}
\item We begin by decomposing the product $gh$ into its factors. We have the commutative diagram
\begin{equation}
\begin{tikzcd}
  A^{\subap}_{gh}
& 
& 
& A^{\subap}_{hg}
    \arrow{lll}{g}
& 
& 
& A^{\subap}_{gh}
    \arrow{lll}{h}
    \arrow[bend right]{llllll}[swap]{gh}
\end{tikzcd}
\end{equation}
Indeed, since $hg$ is the conjugate of $gh$ by $h$ and vice-versa, we have $h(A^{\subap}_{gh}) = A^{\subap}_{hg}$ and $g(A^{\subap}_{hg}) = A^{\subap}_{gh}$.
\item Next we factor the map $g: A^{\subap}_{hg} \to A^{\subap}_{gh}$ through the map $g: A^{\subap}_{g} \to A^{\subap}_{g}$, which is better known to us. We have the commutative diagram
\begin{equation}
\begin{tikzcd}
  A^{\subap}_{gh}
    \arrow{rdd}[swap]{\pi_g}
&
&
& A^{\subap}_{hg}
    \arrow{lll}{g}
    \arrow{ldd}{\pi_g}
\\
\\
& A^{\subap}_{g}
& A^{\subap}_{g}
    \arrow{l}{g}
&
\end{tikzcd}
\end{equation}
where $\pi_g$ is the projection onto $A^{\subap}_{g}$ parallel to $V^{\subgg}_{g} \oplus V^{\subll}_{g}$. (It commutes with $g$ because $A^{\subap}_{g}$, $V^{\subgg}_{g}$ and $V^{\subll}_{g}$ are all invariant by $g$.)
\item Now we decompose again every diagonal arrow from the last diagram into two factors. We call $P_1$ (resp. $P_2$) the projection onto $A^{\subap}_{g, gh}$ (resp. $A^{\subap}_{hg, g}$), still parallel to~$V^{\subgg}_{g} \oplus V^{\subll}_{g}$. (Recall the notational convention $A^{\subap}_{u, v} := A^{\subgap}_{u} \cap A^{\sublap}_{v}$ from Figure~\ref{fig:Margulis_schematic}.)

To justify this definition, we must check that $A^{\subap}_{g, gh}$ (and similarly~$A^{\subap}_{hg, g}$) is supplementary to $V^{\subgg}_{g} \oplus V^{\subll}_{g}$. Indeed, by Remark~\ref{ggh_gh_2C}, $A^{\sublap}_{gh}$ is transverse to $A^{\subgap}_{g}$, hence (by Proposition~\ref{pair_transitivity} and Proposition~\ref{uniquely_determined}~(ii)) supplementary to $V^{\subgg}_{g}$; thus $A^{\subgap}_{g} = V^{\subgg}_{g} \oplus A^{\subap}_{g, gh}$ and $A = V^{\subll}_{g} \oplus A^{\subgap}_{g} = V^{\subll}_{g} \oplus V^{\subgg}_{g} \oplus A^{\subap}_{g, gh}$. Then we have the commutative diagrams
\begin{subequations}
\begin{equation}
\begin{tikzcd}
  A^{\subap}_{gh}
    \arrow[bend right]{rr}[swap]{\pi_g}
    \arrow{r}{P_1}
& A^{\subap}_{g, gh}
    \arrow{r}{\pi_g}
& A^{\subap}_{g}
\end{tikzcd}
\end{equation}
and
\begin{equation}
\begin{tikzcd}
  A^{\subap}_{hg}
    \arrow[bend right]{rr}[swap]{\pi_g}
    \arrow{r}{P_2}
& A^{\subap}_{hg, g}
    \arrow{r}{\pi_g}
& A^{\subap}_{g}
\end{tikzcd}
\end{equation}
\end{subequations}
\item Finally (and this is new in comparison to~\cite{Smi16}), we would like to replace $P_1$ and~$P_2$ by some ``better-behaved'' projections.
\begin{itemize}
\item We define $\pi_{g, gh}$ to be the projection onto~$A^\subap_{g, gh}$ parallel to~$V^\subgg_g \oplus V^\subll_{gh}$. Obviously it induces a bijection between $A^\subap_{gh}$ and~$A^\subap_{g, gh}$.
\item We define $\pi_{hg, g}$ to be the projection onto~$A^\subap_{hg, g}$ parallel to~$V^\subgg_{hg} \oplus V^\subll_g$. Obviously it induces a bijection between $A^\subap_{hg}$ and~$A^\subap_{hg, g}$.
\end{itemize}
The reason they are ``better-behaved'' is that by Lemma~\ref{projections_commute}, they actually commute with canonical identifications (see Remark~\ref{clean_proof} below for more details). We then make the decompositions
\begin{subequations}
\label{eq:Pi_decomposition}
\begin{equation}
\label{eq:P1_decomposition}
P_1 = (P_1 \circ \pi_{g, gh}^{-1}) \circ \pi_{g, gh}
\end{equation}
and
\begin{equation}
\label{eq:P2_decomposition}
P_2 = (P_2 \circ \pi_{hg, g}^{-1}) \circ \pi_{hg, g}.
\end{equation}
\end{subequations}
\end{itemize}

The last three steps can be repeated with $h$ instead of $g$. The way to adapt the second step is straightforward; for the third step, we factor $\pi_h: A^{\subap}_{hg} \to A^{\subap}_{h}$ through $A^{\subap}_{h, hg}$ and $\pi_h: A^{\subap}_{gh} \to A^{\subap}_{h}$ through $A^{\subap}_{gh, h}$; for the fourth step, we project respectively along~$V^\subgg_h \oplus V^{\subll}_{hg}$ and along~$V^{\subgg}_{gh} \oplus V^\subll_h$.

Combining these four decompositions, we get the lower half of Diagram~\ref{fig:largediagcomm1}. (Compare it with Figure~\ref{fig:Margulis_schematic} for a more geometric picture.) We left out the expansion of $h$; we leave drawing the full diagram for especially brave readers.

Let us now interpret all of these maps as endomorphisms of $A^\sube_0$. To do this, we choose some optimal canonizing maps
\[\phi_{g},\; \phi_{gh},\; \phi_{hg},\; \phi_{g, gh},\; \phi_{hg, g}\]
respectively of $g$, of $gh$, of $hg$, of the pair $(A^{\subgap}_{g}, A^{\sublap}_{gh})$ and of the pair $(A^{\subgap}_{hg}, A^{\sublap}_{g})$. This allows us to define $\overline{g_{gh}}$, $\overline{h_{gh}}$, $\overline{g_{g, gh}}$, $\overline{g_{\subap}}$, $\overline{\Delta_1}$, $\overline{\Delta_2}$, $\overline{\chi_1}$, $\overline{\chi_2}$, $\overline{\psi_1}$, $\overline{\psi_2}$ to be the maps that make the whole Diagram~\ref{fig:largediagcomm1} commutative.

\begin{figure}[!p]
\renewcommand{\figurename}{Diagram}
\[
\begin{tikzcd}
  A^\sube_0
    \arrow{rd}[swap]{\overline{\Delta_1}}
&
&
&
&
&
&
&
&
  A^\sube_0
    \arrow{ld}{\overline{\Delta_2}}
    \arrow{llllllll}{\overline{g_{gh}}}
& A^\sube_0
    \arrow{l}{\overline{h_{gh}}}
\\
&
  A^\sube_0
    \arrow{rdd}[swap]{\overline{\chi_1}}
&
&
&
&
&
& A^\sube_0
    \arrow{ldd}{\overline{\chi_2}}
    \arrow{llllll}{\overline{g_{g, gh}}}
\\
\\
&
&
  A^\sube_0
    \arrow{rdd}[swap]{\overline{\psi_1}}
&
&
&
& A^\sube_0
    \arrow{ldd}{\overline{\psi_2}}
&
&
\\
\\
&
&
& A^\sube_0
&
& A^\sube_0
	\arrow{ll}{\overline{g_{\subap}}}
&
&
&
\\
  A^{\subap}_{gh}
    \arrow[bend right]{rrddd}[swap]{P_1}
    \arrow[dashed]{uuuuuu}[pos=0.04]{\phi_{gh}}
&
&
&
&
&
&
&
& A^{\subap}_{hg}
    \arrow[bend left]{llddd}{P_2}
    \arrow{llllllll}{g}
    \arrow[dashed]{uuuuuu}[pos=0.04]{\phi_{hg}}
& A^{\subap}_{gh}
    \arrow{l}{h}
    \arrow[dashed]{uuuuuu}[pos=0.04]{\phi_{gh}}
\\
& A^{\subap}_{gh}
    \arrow{rdd}{\pi_{g, gh}}
    \arrow[dashed]{uuuuuu}[pos=0.26]{\phi_{gh}}
&
&
&
&
&
& A^{\subap}_{hg}
    \arrow{ldd}[swap]{\pi_{hg, g}}
    \arrow[dashed]{uuuuuu}[pos=0.26]{\phi_{hg}}
\\
\\
&
&
  A^{\subap}_{g, gh}
    \arrow{rdd}{\pi_{g}}
    \arrow[dashed]{uuuuuu}[pos=0.60]{\phi_{g, gh}}
&
&
&
& A^{\subap}_{hg, g}
    \arrow{ldd}[swap]{\pi_{g}}
    \arrow[dashed]{uuuuuu}[pos=0.60]{\phi_{hg, g}}
&
&
\\
\\
&
&
& A^{\subap}_{g}
    \arrow[dashed]{uuuuuu}[pos=0.93]{\phi_{g}}
&
& A^{\subap}_{g}
	\arrow{ll}{g}
    \arrow[dashed]{uuuuuu}[pos=0.93]{\phi_{g}}
&
&
&
\end{tikzcd}
\]
\caption{}
\label{fig:largediagcomm1}
\end{figure}
\afterpage{\clearpage} 

Now let us define
\begin{equation}
\begin{cases}
M_{gh}(g) := \pi_{\transl}(\overline{g_{gh}}(x) - x) \\
M_{gh}(h) := \pi_{\transl}(\overline{h_{gh}}(x) - x)
\end{cases}
\end{equation}
for any $x \in V^\sube_{\Aff, 0}$, where $V^\sube_{\Aff, 0} := A^\sube_0 \cap V_{\Aff}$ is the affine space parallel to $V^\sube_0$ and passing through the origin~$p_0$. Since $gh$ is the conjugate of $hg$ by $g$ and vice-versa, the elements of $G \ltimes V$ (defined in an obvious way) whose restrictions to $A^\sube_0$ are $\overline{g_{gh}}$ and $\overline{h_{gh}}$ stabilize the spaces $A^\subge_0$ and $A^\suble_0$. By Lemma~\ref{quasi-translation}, $\overline{g_{gh}}$ and $\overline{h_{gh}}$ are thus quasi-translations. It follows that these values $M_{gh}(g)$ and $M_{gh}(h)$ do not depend on the choice of $x$. Compare this to the definition of a Margulis invariant (Definition~\ref{margulis_invariant}): we have $M(gh) = \pi_{\transl}(\overline{g_{gh}} \circ \overline{h_{gh}}(x) - x)$ for any $x \in V^\sube_{\Aff, 0}$. It immediately follows that
\begin{equation}
\label{eq:Mgh_first_decomposition}
M(gh) = M_{gh}(g) + M_{gh}(h).
\end{equation}
(Note that while the vectors $M_{gh}(g)$ and $M_{gh}(h)$ are elements of $V^\transl_0$, the maps $\overline{g_{gh}}$ and $\overline{h_{gh}}$ are extended affine isometries acting on the whole subspace $A^\sube_0$.) Thus it is enough to show that
\begin{equation}
\|M_{gh}(g) - M(g)\| \lesssim_C 1 \quad\text{ and }\quad \|M_{gh}(h) - M(h)\| \lesssim_C 1.
\end{equation}

We will prove the estimate for~$g$ in two steps: it will follow by combining Lemmas \ref{MgghgMg} and~\ref{MghgMgghg} below. The proof of the estimate for~$h$ is analogous.
\end{proof}
\begin{remark}
In contrast to actual Margulis invariants, the values $M_{gh}(g)$ and $M_{gh}(h)$ \emph{do} depend on our choice of canonizing maps. Choosing other canonizing maps would force us to subtract some constant from the former and add it to the latter.
\end{remark}

\begin{remark}
\label{clean_proof}
The fourth decomposition step above, namely \eqref{eq:Pi_decomposition}, is what makes the whole proof much cleaner. In~\cite{Smi16}, we had a map called~$\overline{P_1}$, which was ``almost a quasi-translation'' (Lemma~8.9 in~\cite{Smi16}). In the current proof, this map decomposes into two pieces
\begin{equation}
\label{eq:P1_new_decomposition}
\overline{P_1} = \overline{\chi_1} \circ \overline{\Delta_1},
\end{equation}
that are much easier to deal with: $\overline{\chi_1}$ is now a \emph{bounded} quasi-translation just like $\overline{\psi_1}$, and thus falls in with the ``algebraic'' part; while~$\overline{\Delta_1}$ is now almost the \emph{identity} (Lemma~\ref{P1_almost_Id}), and thus stays in the ``analytic'' part.
\end{remark}

\begin{lemma}
\label{MgghgMg}
Take any $x \in V^\sube_{\Aff, 0}$. Then the vector
\[M_{g, gh}(g) := \pi_{\transl}(\overline{g_{g, gh}}(x) - x)\]
does not, in fact, depend on the choice of~$x$, and satisfies
\[\|M_{g, gh}(g) - M(g)\| \lesssim_C 1.\]
\end{lemma}
\begin{proof}
We start (and this is a key part of this subsection) by invoking Lemma~\ref{projections_commute} on the four maps $\overline{\chi_1}^{-1}$, $\overline{\psi_1}^{-1}$, $\overline{\psi_2}$ and~$\overline{\chi_2}$. The lemma tells us that all four of them are quasi-translations. Hence the map~$\overline{g_{g, gh}} = \overline{\chi_1}^{-1} \circ \overline{\psi_1}^{-1} \circ \overline{g_{\subap}} \circ \overline{\psi_2} \circ \overline{\chi_2}$ is also a quasi-translation; this proves that $M_{g, gh}$ is well-defined.

Let us now show that the norms of these four maps are bounded by a constant that depends only on~$C$. It is not hard to check that the second part of the lemma follows from this. 
\begin{itemize}
\item Let us start with~$\overline{\chi_2}$. By definition, we have:
\begin{align}
\overline{\chi_2} &= \phi_{hg, g} \circ \pi_{hg, g} \circ \phi_{hg}^{-1} \nonumber \\
                  &= \left( \phi_{hg, g} \circ \pi_{hg, g} \circ \phi_{hg, g}^{-1} \right) \circ
                     \left( \phi_{hg, g} \circ \phi_{hg}^{-1} \right) \nonumber \\
                  &= \pi_0 \circ \left( \phi_{hg, g} \circ \phi_{hg}^{-1} \right),
\end{align}
where $\pi_0$ is the projection onto~$A^\sube_0$ parallel to~$V^\subg_0 \oplus V^\subl_0$. Now $\pi_0$ is actually an orthogonal projection, hence it has norm~$1$; and $\phi_{hg, g} \circ \phi_{hg}^{-1}$~is bounded by Remark~\ref{ggh_gh_2C}. 
\item Similarly, we have
\begin{equation}
\overline{\psi_2} = \pi_0 \circ \left( \phi_g \circ \phi_{hg, g}^{-1} \right),
\end{equation}
which is bounded.
\item To deal with $\overline{\psi_1}^{-1}$, note that $\pi_g: A^\subap_{g, gh} \to A^\subap_g$ and $\pi_{g, gh}: A^\subap_g \to A^\subap_{g, gh}$ are inverse to each other. We deduce that
\begin{align}
\overline{\psi_1}^{-1} &= \phi_{g, gh} \circ \pi_{g, gh} \circ \phi_g^{-1} \nonumber \\
                       &= \left( \phi_{g, gh} \circ \pi_{g, gh} \circ \phi_{g, gh}^{-1} \right) \circ
                          \left( \phi_{g, gh} \circ \phi_g^{-1} \right) \nonumber \\
                       &= \pi_0 \circ \left( \phi_{g, gh} \circ \phi_g^{-1} \right),
\end{align}
and we conclude as previously.
\item Similarly, we have
\begin{equation}
\overline{\chi_1}^{-1} = \pi_0 \circ \left( \phi_{gh} \circ \phi_{g, gh}^{-1} \right),
\end{equation}
and we conclude in the same way. \qedhere
\end{itemize}
\end{proof}

In order to prove Lemma~\ref{MghgMgghg}, we first need to prove that the maps $\overline{\Delta_1}$ and $\overline{\Delta_2}$ are, in an appropriate sense, close to the identity.

\begin{lemma}
\label{P1_almost_Id}
The estimates
\begin{hypothenum}
\item $\|\overline{\Delta_1}^{-1} - \Id\| \lesssim_C \alpha^\mathrm{Haus}\left( A^\subgap_{gh},\; A^\subgap_g \right)$;
\item $\|\overline{\Delta_2}^{-1} - \Id\| \lesssim_C \alpha^\mathrm{Haus}\left( A^\sublap_{hg},\; A^\sublap_g \right)$,
\end{hypothenum}
hold, as soon as the respective right-hand sides are smaller than some constant depending only on~$C$.
\end{lemma}

\begin{corollary}
\label{P1_almost_Id_cor}
We have:
\begin{hypothenum}
\item $\|\overline{\Delta_1}^{-1} - \Id\| \lesssim_C s_{X_0}(g)$;
\item $\|\overline{\Delta_2}^{-1} - \Id\| \lesssim_C s_{X_0}(g^{-1})$;
\item $\|\ell(\overline{\Delta_1})^{-1} - \Id\| \lesssim_C \vec{s}_{X_0}(\ell(g))$;
\item $\|\ell(\overline{\Delta_2})^{-1} - \Id\| \lesssim_C \vec{s}_{X_0}(\ell(g))$.
\end{hypothenum}
\end{corollary}
In light of Remark~\ref{clean_proof}, this corollary can be seen as a simpler version of Lemma~8.9 in~\cite{Smi16}. Indeed the old~$\overline{P_1}$ corresponds by~\eqref{eq:P1_new_decomposition} to the new $\overline{\chi_1} \circ \overline{\Delta_1}$, while the old~$\overline{P''_1}$ corresponds to~$\overline{\chi_1}$ alone.
\begin{proof}
Points (i) and~(ii) immediately follow from the Lemma combined with Proposition~\ref{regular_product}~(ii), provided that we take $s_{\ref{invariant_additivity_only}}(C)$ small enough.

For points (iii) and~(iv), simply apply the Lemma to the pair $(\ell(g), \ell(h))$. (It is easy to check that this pair still satisfies the hypotheses of Proposition~\ref{invariant_additivity_only}.) We then get:
\[\begin{cases}
\|\ell(\overline{\Delta_1})^{-1} - \Id\| \lesssim_C \alpha^\mathrm{Haus}\left( V^\subgap_{gh},\; V^\subgap_g \right) \lesssim_C \vec{s}_{X_0}(\ell(g)); \\
\|\ell(\overline{\Delta_2})^{-1} - \Id\| \lesssim_C \alpha^\mathrm{Haus}\left( V^\sublap_{hg},\; V^\sublap_g \right) \lesssim_C \vec{s}_{X_0}(\ell(g)).
\end{cases}\]
In each case, the second inequality follows from Proposition~\ref{regular_product}~(i). The first inequality is an application of the Lemma, which is licit (provided $s_{\ref{invariant_additivity_only}}(C)$~is small enough), since Proposition~\ref{affine_to_intrinsic_linear}~(ii) propagates the required upper bound to~$\vec{s}_{X_0}(\ell(g))$.
\end{proof}

\begin{proof}[Proof of Lemma~\ref{P1_almost_Id}]
This proof might seem slightly technical, but there is an easy intuition behind it. Essentially, the idea is that if you jump back and forth between two spaces that almost coincide, going both times in directions whose angle with the two spaces is not too shallow, then you can't end up very far from your starting point.
\begin{hypothenum}
\item By Remark~\ref{ggh_gh_2C}, we know that $\|\phi_{gh}^{\pm 1}\| \leq 2C$. Conjugating everything by this map, it is thus sufficient to show that for every $x \in A^\subap_{gh}$, we have
\begin{equation}
\label{eq:p1pigghx_close_to_x}
\left\| P_1^{-1}(\pi_{g, gh}(x)) - x \right\| \lesssim_C \alpha^\mathrm{Haus}\left( A^\subgap_{gh},\; A^\subgap_g \right)\|x\|
\end{equation}
(where by $P_1^{-1}$ we mean the inverse of the bijection $P_1: A^\subap_{gh} \to A^\subap_{g, gh}$).

\begin{itemize}
\item Let us first estimate the quantity $\|\pi_{g, gh}(x) - x\|$. To begin with, let us push everything forward by the map $\phi_{g, gh}$; writing $y = \phi_{g, gh}(x)$, we have:
\begin{align*}
\frac{\|\phi_{g, gh}(\pi_{g, gh}(x) - x)\|}{\|\phi_{g, gh}(x)\|}
&= \frac{\|\pi_0(y) - y\|}{\|y\|} \\
&= \sin \alpha(y, A^\sube_0) \\
&\leq \alpha(y, A^\sube_0).
\end{align*}
By Remark~\ref{ggh_gh_2C}, we know that $\|\phi_{g, gh}^{\pm 1}\| \leq 2C$; hence we may pull everything back by~$\phi_{g, gh}$ again:
\begin{align}
\frac{\|\pi_{g, gh}(x) - x\|}{\|x\|}
&\lesssim_C \alpha(x, A^\subap_{g, gh}) \nonumber \\
&\leq\phantom{_C} \alpha\left( A^\subap_{gh},\; A^\subgap_g \right) \nonumber \\
&\leq\phantom{_C} \alpha^\mathrm{Haus}\left( A^\subgap_{gh},\; A^\subgap_g \right).
\end{align}
We conclude that
\begin{equation}
\label{eq:pigghx_close_to_x}
\|\pi_{g, gh}(x) - x\| \lesssim_C \alpha^\mathrm{Haus}\left( A^\subgap_{gh},\; A^\subgap_g \right)\|x\|.
\end{equation}

\item Let us now estimate the quantity $\left\| P_1^{-1}(\pi_{g, gh}(x)) - \pi_{g, gh}(x) \right\|$. We introduce the notation~$z = \pi_{g, gh}(x)$, and we define $\xi$ to be the unique linear automorphism of~$A$ satisfying
\begin{equation}
\xi(x) =
\begin{cases}
\phi_{g, gh}(x) &\text{if } x \in A^\subgap_g \\
\phi_g(x) &\text{if } x \in V^\subll_g
\end{cases}
\end{equation}
(so that $\xi^{-1}(V^\subg_0) = V^\subgg_g$ and~$\xi^{-1}(V^\subl_0) = V^\subll_g$ but~$\xi^{-1}(A^\sube_0) = A^\subap_{g, gh}$). From the inequalities $\|\phi_g^{\pm 1}\| \leq C$ and~$\|\phi_{g, gh}^{\pm 1}\| \leq 2C$, it is easy to deduce that both norms $\|\xi\|$ and $\|\xi^{-1}\|$ are bounded by a constant that depends only on~$C$.

Now we obviously have $\xi \circ P_1 \circ \xi^{-1} = \pi_0$, hence
\[\frac{\|\xi(P_1^{-1}(z) - z)\|}{\|\xi(z)\|} = \tan \alpha(\xi(P_1^{-1}(z)), A^\sube_0).\]
Now we have
\begin{align*}
\alpha(\xi(P_1^{-1}(z)), A^\sube_0)
&\lesssim_C \alpha(P_1^{-1}(z), A^\subap_{g, gh}) &\text{since $\xi$ and $\xi^{-1}$ are bounded} \\
&\lesssim_C \alpha^\mathrm{Haus}\left( A^\subgap_{gh},\; A^\subgap_g \right) &\text{since $P_1^{-1}(z) \in A^\subap_{gh}$.}
\end{align*}
If the right-hand side is small enough, we may assume that $\alpha(\xi(z), A^\sube_0)$ is smaller than some fixed constant~$\alpha_0$ (say $\alpha_0 = \frac{\pi}{4}$). For every $\alpha \leq \alpha_0$, we obviously have
\[\tan \alpha \lesssim_{\alpha_0} \alpha.\]
It follows that
\begin{equation}
\frac{\|\xi(P_1^{-1}(z) - z)\|}{\|\xi(z)\|} \lesssim_C \alpha^\mathrm{Haus}\left( A^\subgap_{gh},\; A^\subgap_g \right).
\end{equation}
Now since $\xi$ and $\xi^{-1}$ are bounded, we deduce that
\[\|P_1^{-1}(z) - z\| \lesssim_C \alpha^\mathrm{Haus}\left( A^\subgap_{gh},\; A^\subgap_g \right)\|z\|.\]
It remains to estimate $\|z\|$ in terms of $\|x\|$. We have:
\[\|z\| \;\leq\; \|x\| + \|\pi_{g, gh}(x) - x\| \;\lesssim_C\; \|x\| + \alpha^\mathrm{Haus}\left( A^\subgap_{gh},\; A^\subgap_g \right)\|x\|\]
by \eqref{eq:pigghx_close_to_x}. Taking~$\alpha^\mathrm{Haus}\left( A^\subgap_{gh},\; A^\subgap_g \right)$ small enough, we may assume that $\|z\| \leq 2\|x\|$. We conclude that
\begin{equation}
\label{eq:p1z_close_to_z}
\left\| P_1^{-1}(\pi_{g, gh}(x)) - \pi_{g, gh}(x) \right\| \lesssim_C \alpha^\mathrm{Haus}\left( A^\subgap_{gh},\; A^\subgap_g \right)\|x\|.
\end{equation}
\end{itemize}
Adding together \eqref{eq:pigghx_close_to_x} and~\eqref{eq:p1z_close_to_z}, the desired inequality~\eqref{eq:p1pigghx_close_to_x} follows.

\item The proof is completely analogous. \qedhere
\end{hypothenum}
\end{proof}

\begin{lemma}
\label{MghgMgghg}
We have
\[\|M_{gh}(g) - M_{g, gh}(g)\| \lesssim_C s_{X_0}(g).\]
\end{lemma}
The proof bears some similarity to the second half of the proof of Lemma~8.6 in~\cite{Smi16}.
\begin{proof}
We shall decompose the difference $M_{gh}(g) - M_{g, gh}(g)$ into two terms~\eqref{eq:Mgh_difference_decomposition}, and then estimate each term separately.

Recall that
\[M_{gh}(g) = \pi_{\transl}(\overline{g_{gh}}(x) - x),\]
where $x$ can be any element of the affine space~$V^\sube_{\Aff, 0}$. Let us take
\[x = \overline{\Delta_2}^{-1} \left( \overline{g_{g, gh}}^{-1}(p_0) \right).\]
We remind that $p_0$ is the point chosen as the origin of the affine space~$V_{\Aff}$ (see Subsection~\ref{sec:affine}), seen in the extended affine space~$A$ as a unit vector orthogonal to~$V$.

Since $\overline{g_{gh}} = \overline{\Delta_1}^{-1} \circ \overline{g_{g, gh}} \circ \overline{\Delta_2}$, we may then write
\begin{align*}
M_{gh}(g) &= \pi_{\transl}\left(\overline{\Delta_2}(x) - x\right) \;+\;
   \pi_{\transl}\left(\overline{g_{g, gh}}\left(\overline{\Delta_2}(x)\right) - \overline{\Delta_2}(x)\right) \;+\;
   \pi_{\transl}\left(\overline{\Delta_1}^{-1}\left(\overline{g_{g, gh}}\left(\overline{\Delta_2}(x)\right)\right) - \overline{g_{g, gh}}\left(\overline{\Delta_2}(x)\right)\right) \\
   &= \underbrace{\pi_{\transl}\left( \overline{g_{g, gh}}^{-1}(p_0) - \overline{\Delta_2}^{-1} \left( \overline{g_{g, gh}}^{-1}(p_0)\right) \right)}_{\displaystyle =: (I)} \;+\;
   \underbrace{\pi_{\transl}\left(p_0 - \overline{g_{g, gh}}^{-1}(p_0)\right)}_{\displaystyle =: (II)} \;+\;
   \underbrace{\pi_{\transl}\left(\overline{\Delta_1}^{-1}(p_0) - p_0\right)}_{\displaystyle =: (III)}.
\end{align*}
The middle term~$(II)$ is then by definition equal to~$M_{g, gh}(g)$, so that
\begin{equation}
\label{eq:Mgh_difference_decomposition}
M_{gh}(g) - M_{g, gh}(g) = (I) + (III).
\end{equation}
Now we have
\begin{align}
\label{eq:I_estimation}
\|(III)\| &\leq \|\overline{\Delta_1}^{-1}(p_0) - p_0\| \nonumber\\
          &\leq \|\overline{\Delta_1}^{-1} - \Id\|\|p_0\| \nonumber\\
          &=    \|\overline{\Delta_1}^{-1} - \Id\| \nonumber\\
          &\lesssim_C s_{X_0}(g)
                 &\text{ by Corollary~\ref{P1_almost_Id_cor}~(i).}
\end{align}
It remains to estimate~$(I)$. We set
\[y := \overline{g_{g, gh}}^{-1}(p_0).\]
Let us calculate the norm of this vector:
\begin{align}
\label{eq:y_estimation}
\|y\| &\leq\phantom{_C}\! \left\|\overline{g_{g, gh}}^{-1}\right\| \nonumber\\
      &=\phantom{_C}\!    \left\|\overline{\chi_2}^{-1} \circ \overline{\psi_2}^{-1} \circ \overline{g_\subap}^{-1} \circ \overline{\psi_1} \circ \overline{\chi_1}\right\| \nonumber\\
      &\lesssim_C \left\|\overline{g_\subap}^{-1}\right\| \nonumber\\
      &=\phantom{_C}\!    \left\|\phi_g \circ \restr{g^{-1}}{A^\subap_g} \circ \phi_g^{-1}\right\| \nonumber\\
      &\lesssim_C \left\|\restr{g^{-1}}{A^\subap_g}\right\|.
\end{align}
(To justify the third line, remember that we have seen in the proof of Lemma~\ref{MgghgMg} that the four maps $\overline{\chi_1}^{-1}$, $\overline{\psi_1}^{-1}$, $\overline{\psi_2}$ and~$\overline{\chi_2}$ are all bounded.) Now we have:
\begin{align}
\label{eq:III_estimation}
\|(I)\| &= \left\|\pi_t\left(y - \overline{\Delta_2}^{-1}(y)\right)\right\| \nonumber\\
        &\leq \left\|\overline{\Delta_2}^{-1}(y) - y\right\| \nonumber\\
        &= \mathrlap{\left\|\left(\ell(\overline{\Delta_2}^{-1}) - \Id\right)(y - p_0) +
                  \left(\overline{\Delta_2}^{-1} - \Id\right)(p_0)\right\|}
             &\text{ by definition of } p_0 \nonumber\\
        &\leq \left\|\ell(\overline{\Delta_2}^{-1}) - \Id\right\|\|y\| + \left\|\overline{\Delta_2}^{-1} - \Id\right\|\!\!\!\!\!\!\!\!\!\!\!\!\!\!\!\!\!\!\!\!\!
             &\text{ as } \|y\|^2 = \|y - p_0\|^2 + 1 \nonumber\\
        &\lesssim_C \vec{s}_{X_0}(\ell(g))\left\|\restr{g^{-1}}{A^\subap_g}\right\| + s_{X_0}(g)
             &\text{ by Corollary~\ref{P1_almost_Id_cor} (ii) and~(iv) and~\eqref{eq:y_estimation}} \nonumber\\
        &\lesssim_C s_{X_0}(g)
             &\text{ by Proposition~\ref{affine_to_intrinsic_linear}~(iii).}
\end{align}
Joining together \eqref{eq:I_estimation} and~\eqref{eq:III_estimation}, the conclusion follows.
\end{proof}

\section{Margulis invariants of words}
\label{sec:induction}

We have already studied how contraction strengths (Proposition~\ref{regular_product}) and Margulis invariants (Proposition~\ref{invariant_additivity_only}) behave when we take the product of two mutually $C$-non-degenerate, sufficiently contracting $\rho$-regular maps. The goal of this section is to generalize these results to words of arbitrary length on a given set of generators. This is a straightforward generalization of Section~9 in~\cite{Smi16} and of Section~5 in~\cite{Smi14}.

\begin{definition}
\label{cyclically_reduced_definition}
Take $k$ generators $g_1, \ldots, g_k$. Consider a word $g = g_{i_1}^{\sigma_1} \cdots g_{i_l}^{\sigma_l}$ with length~$l \geq 1$ on these generators and their inverses (for every $m = 1, \ldots, l$, we have $1 \leq i_m \leq k$ and~$\sigma_m = \pm 1$). We say that $g$ is \emph{reduced} if for every $m = 1, \ldots, l-1$, we have $(i_{m+1}, \sigma_{m+1}) \neq (i_m, -\sigma_m)$. We say that $g$ is \emph{cyclically reduced} if it is reduced and also satisfies $(i_1, \sigma_1) \neq (i_l, -\sigma_l)$.
\end{definition}

\begin{proposition}
\label{Schottky_group}
For every $C \geq 1$, there is a positive constant $s_{\ref{Schottky_group}}(C) \leq 1$ with the following property. Take any family of maps $g_1, \ldots, g_k \in G \ltimes V$ satisfying the following hypotheses:
\begin{enumerate}[label=\textnormal{(H\arabic*)}]
\item \label{itm:all_are_X0} Every $g_i$ is $\rho$-regular.
\item \label{itm:pairwise_C_non_deg} Any pair taken among the maps $\{g_1, \ldots, g_k, g_1^{-1}, \ldots, g_k^{-1}\}$ is $C$-non-degenerate, except of course if it has the form $(g_i, g_i^{-1})$ for some $i$.
\item \label{itm:all_are_contracting} For every $i$, we have $s_{X_0}(g_i) \leq s_{\ref{Schottky_group}}(C)$ and $s_{X_0}(g_i^{-1}) \leq s_{\ref{Schottky_group}}(C)$.
\end{enumerate}
Then every nonempty cyclically reduced word $g = g_{i_1}^{\sigma_1} \cdots g_{i_l}^{\sigma_l}$ (with $1 \leq i_m \leq k$, $\sigma_m = \pm 1$ for every $m$) has the following properties.
\begin{hypothenum}
\addtolength{\itemsep}{.5\baselineskip}
\item The map~$g$ is $\rho$-regular.
\item $\begin{cases}
       \alpha^\mathrm{Haus} \left(A^{\subgap}_{g},\;
                                  A^{\subgap}_{g_{i_1}^{\sigma_1}} \right)
          \lesssim_C 2 \left( 1 - 2^{-(l-1)} \right) s_{\ref{Schottky_group}}(C) \vspace{1mm} \\
       \alpha^\mathrm{Haus} \left(A^{\sublap}_{g},\;
                                  A^{\sublap}_{g_{i_l}^{\sigma_l}} \right)
          \lesssim_C 2 \left( 1 - 2^{-(l-1)} \right) s_{\ref{Schottky_group}}(C).
       \end{cases}$
\item $s_{X_0}(g) \leq 2^{-(l-1)} s_{\ref{Schottky_group}}(C)$.
\item $\displaystyle \left\| M(g) - \sum_{m=1}^{l} M(g_{i_m}^{\sigma_m}) \right\| \leq (l-1)k_{\ref{invariant_additivity_only}}(2C)$.
\item If $h = g_{i'_1}^{\sigma'_1} \cdots g_{i'_{l'}}^{\sigma'_{l'}}$ is another nonempty cyclically reduced word such that $gh$ is also cyclically reduced, the pair $(g, h)$ is $2C$-non-degenerate.
\end{hypothenum}
\end{proposition}

In the sequel, we will only make use of the properties (i), (iv) and the particular case $h = g$ of (v) (let us call these three properties the ``weak'' version). However the natural way to prove this proposition is to work by induction, with the induction step relying on all five properties (the ``full'' version).

This proof is exactly analogous to the proof of Proposition~5.2 in~\cite{Smi14} (whose formulation technically corresponds to the weak version only; but the proof is actually done by proving the full version and discarding the unnecessary statements). For ease of reading, we reproduce it here nevertheless.

\begin{proof}
We proceed by induction on $\max(l, l')$. Clearly all five statements are true for $l = l' = 1$.

Now let $l \geq 2$, and suppose that statements (i) through (v) are true for all cyclically reduced words of length $m$ with $1 \leq m \leq l-1$. Take any cyclically reduced word $g = g_{i_1}^{\sigma_1} \cdots g_{i_l}^{\sigma_l}$. Then we claim that it is possible to decompose $g$ into two cyclically reduced subwords
\[g' := g_{i_1}^{\sigma_1} \cdots g_{i_m}^{\sigma_m}
\quad \text{and} \quad
  g'' := g_{i_{m+1}}^{\sigma_{m+1}} \cdots g_{i_l}^{\sigma_l},\]
both nonempty (that is, $0 < m < l$).

For example, we may take~$m$ to be the smallest (positive) index such that
\[(i_{m+1}, \sigma_{m+1}) \neq (i_l, -\sigma_l).\]
Such an index always exists, and is at most equal to~$l-1$. Then the first subword is actually of the form
\[g' = g_{i_1}^{\sigma_1} \left( g_{i_l}^{-\sigma_l} \right)^{m-1},\]
and any word of this form is automatically cyclically reduced as soon as it is reduced. As for $g''$, it is cyclically reduced by construction.

By induction hypotheses (i) and (v), $g'$ and $g''$ are $\rho$-regular and form a $2C$-non-degenerate pair; by induction hypothesis (iii), we have $s_{X_0}(g') \leq 2^{-(m-1)} s_{\ref{Schottky_group}}(C) \leq s_{\ref{Schottky_group}}(C)$ and we may suppose that $s_{\ref{Schottky_group}}(C) \leq s_{\ref{regular_product}}(2C)$ (similarly for $g'^{-1}$, $g''$, $g''^{-1}$). Thus the pair $(g', g'')$ satisfies Proposition~\ref{regular_product} (with constant $2C$). Let us show that $g$~satisfies the properties (i) through (v).
\begin{itemize}
\item The property (i) (that $g$ is $\rho$-regular) is a direct consequence of Proposition~\ref{regular_product}.

\item Let us check the property (iii). From Proposition~\ref{regular_product} (iii), it follows that $s_{X_0}(g) \lesssim_C s_{X_0}(g')s_{X_0}(g'')$; we then have, by induction hypothesis (iii):
\begin{align*}
s_{X_0}(g) &\lesssim_C \left( 2^{-(m-1)} s_{\ref{Schottky_group}}(C) \right)\left( 2^{-(l-m-1)} s_{\ref{Schottky_group}}(C) \right) \\
     &=\;        s_{\ref{Schottky_group}}(C) \left( 2^{-(l-2)} s_{\ref{Schottky_group}}(C) \right).\intertext{
If we take $s_{\ref{Schottky_group}}(C)$ sufficiently small, we get
}s_{X_0}(g) &\leq 2^{-(l-1)} s_{\ref{Schottky_group}}(C).\end{align*}

\item Now we check (ii); it is enough to check the first inequality (the second one follows by substituting $g^{-1}$). Note that $\alpha^\mathrm{Haus}$ is a metric on the set of all vector subspaces of~$A$, so we have
\begin{align*}
  \alpha^\mathrm{Haus} \left(A^{\subgap}_{g},\;
                             A^{\subgap}_{g_{i_1}^{\sigma_1}} \right)
&\leq\quad
  \alpha^\mathrm{Haus} \left(A^{\subgap}_{g},\;
                             A^{\subgap}_{g'} \right)
+ \alpha^\mathrm{Haus} \left(A^{\subgap}_{g'},\;
                             A^{\subgap}_{g_{i_1}^{\sigma_1}} \right).\intertext{
Estimating the first term by Proposition~\ref{regular_product} (ii) and the second term by induction hypothesis (ii), we get:
} \alpha^\mathrm{Haus} \left(A^{\subgap}_{g},\;
                             A^{\subgap}_{g_{i_1}^{\sigma_1}} \right)
&\lesssim_C\, s_{X_0}(g') + 2 \left( 1 - 2^{-(m-1)} \right) s_{\ref{Schottky_group}}(C).\intertext{
Now by induction hypothesis (iii) we have $s_{X_0}(g') \leq 2^{-(m-1)} s_{\ref{Schottky_group}}(C)$, hence
} \alpha^\mathrm{Haus} \left(A^{\subgap}_{g},\;
                             A^{\subgap}_{g_{i_1}^{\sigma_1}} \right)
&\lesssim_C\, 2^{-(m-1)} s_{\ref{Schottky_group}}(C) + 2 \left( 1 - 2^{-(m-1)} \right) s_{\ref{Schottky_group}}(C)\\
&=\quad    2 \left( 1 - 2^{-m} \right) s_{\ref{Schottky_group}}(C) &\\
&\leq\quad 2 \left( 1 - 2^{-(l-1)} \right) s_{\ref{Schottky_group}}(C),
\end{align*}
since $m \leq l-1$. (Here the implicit multiplicative constant is the same as in Proposition~\ref{regular_product} (ii), and does not change after the induction step.)

\item Next we check (iv). By induction hypothesis (iv), we have
\[\begin{cases}
\displaystyle \left\| M(g') - \sum_{p=1}^{m} M(g_{i_p}^{\sigma_p}) \right\| \leq (m-1)k_{\ref{invariant_additivity_only}}(2C) \vspace{2mm} \\
\displaystyle \left\| M(g'') - \sum_{p=m+1}^{l} M(g_{i_p}^{\sigma_p}) \right\| \leq (l-m-1)k_{\ref{invariant_additivity_only}}(2C).
\end{cases}\]
If we take $s_{\ref{Schottky_group}}(C) \leq s_{\ref{invariant_additivity_only}}(2C)$, then $g'$ and $g''$ satisfy Proposition~\ref{invariant_additivity_only}, hence
\[\left\| M(g) - M(g') - M(g'') \right\| \leq k_{\ref{invariant_additivity_only}}(2C).\]
Adding these three inequalities together, we get the desired conclusion.

\item It remains to check (v): let $h = g_{i'_1}^{\sigma'_1} \cdots g_{i'_{l'}}^{\sigma'_{l'}}$ be another cyclically reduced word (with $1 \leq l' \leq l$) such that $gh$ is also cyclically reduced. We need to check that the four pairs $(A^\subgap_g, A^\sublap_g)$, $(A^\subgap_g, A^\sublap_h)$, $(A^\subgap_h, A^\sublap_g)$ and $(A^\subgap_h, A^\sublap_h)$ are $2C$-non-degenerate. This follows by Lemma~\ref{continuity_of_non_degeneracy} from the property (ii) (applied to both $g$ and $h$) and from the hypothesis \ref{itm:pairwise_C_non_deg}, provided we take $s_{\ref{Schottky_group}}(C)$ small enough. \qedhere
\end{itemize}
\end{proof}

\section{Construction of the group}
\label{sec:construction}

Here we prove the Main Theorem. The reasoning is similar to that of Section~6 in~\cite{Smi14}, and almost identical to that of Section~10 in~\cite{Smi16}. The main difference is the substitution of~$\mathfrak{a}'_{\rho, X_0}$ instead of~$\mathfrak{a}_{\rho, X_0}$, which in particular requires us to invoke Proposition~\ref{asymptotically_contracting_to_regular} (which had no equivalent in~\cite{Smi16}) in the final proof. Also, since we have now developed the purely linear theory in a systematic way (Section~\ref{sec:linear_regular_maps}), the relationship between linear properties and affine properties becomes clearer (in particular in the second bullet point of the final proof).

Let us recall the outline of the proof. We begin by showing (Lemma~\ref{properly_discontinuous}) that if we take a group generated by a $C$-non-degenerate family of sufficiently contracting $\rho$-regular maps that have suitable Margulis invariants, it satisfies all of the conclusions of the Main Theorem, except Zariski-density. We then exhibit such a group that is also Zariski-dense (and thus prove the Main Theorem).

The idea is to ensure that the Margulis invariants of all elements of the group lie almost on the same half-line. Obviously if $-w_0$ maps every element of~$V^\transl_0$ to its opposite, Proposition~\ref{invariant_inverse} makes this impossible. So we now exclude this case:
\begin{assumption}
\label{inverse_ok}
The representation~$\rho$ is such that the action of~$w_0$ on~$V^\transl_0$ is not trivial.
\end{assumption}
This is precisely condition~\ref{itm:main_condition} from the Main Theorem. More precisely, $V^\transl_0$ is the set of all vectors that satisfy~\ref{itm:main_condition}\ref{itm:fixed_by_l}, and what we say here is that some of them also satisfy~\ref{itm:main_condition}\ref{itm:not_fixed_by_w0}. See Example~10.2 in~\cite{Smi16} for examples of representations that do or do not satisfy this condition.

Thanks to Assumption~\ref{inverse_ok}, we choose once and for all some nonzero vector $M_C \in V^\transl_0$ that is a fixed point of $-w_0$ (which is possible since $w_0$~is an involution). This requirement still leaves us free to prescribe the norm of this vector; let us additionally assume that $\|M_C\| = 2k_{\ref{invariant_additivity_only}}(2C)$.
\begin{lemma}
\label{properly_discontinuous}
Take any family $g_1, \ldots, g_k \in G \ltimes V$ satisfying the hypotheses \ref{itm:all_are_X0}, \ref{itm:pairwise_C_non_deg} and \ref{itm:all_are_contracting} from Proposition~\ref{Schottky_group}, and also the additional condition
\begin{enumerate}[label=\textnormal{(H\arabic*)},start=4]
\item \label{itm:prescribed_marg_inv} For every $i$, $M(g_i) = M_C$.
\end{enumerate}
Then these maps generate a free group acting properly discontinuously on the affine space~$V_{\Aff}$.
\end{lemma}
\begin{proof}
The proof is exactly the same as the proof of Lemma~6.1 in~\cite{Smi14}, \emph{mutatis mutandis}. The role of Proposition~5.2 in~\cite{Smi14} is now played by points (i), (iv) and the case $h = g$ in (v) of Proposition~\ref{Schottky_group}. The constant denoted as $\mu(C)$ (or~$\mu(2C)$) in the earlier paper corresponds to what we now call $k_{\ref{invariant_additivity_only}}(C)$ (or respectively~$k_{\ref{invariant_additivity_only}}(2C)$). The (orthogonal) projection
\[\hat{\pi}_{\mathfrak{z}}: \hat{\mathfrak{g}} \to \mathfrak{z} \oplus \mathbb{R}_0
\quad\text{ parallel to } \mathfrak{d} \oplus \mathfrak{n}^+ \oplus \mathfrak{n}^-\]
(here $\mathbb{R}_0$ was my old notation for $\mathbb{R} p_0$) now becomes the (orthogonal) projection
\[\hat{\pi}_{\transl}: A \to V^\transl_0 \oplus \mathbb{R} p_0
\quad\text{ parallel to } V^\extra_0 \oplus V^\subg_0 \oplus V^\subl_0. \qedhere\]
\end{proof}

We may now finally prove the Main Theorem. We follow the same strategy as in the proof of the Main Theorems of~\cite{Smi14} and of~\cite{Smi16} (with a few additional tweaks).

\begin{proof}[Proof of Main Theorem]
Let $\rho$ be any representation that satisfies the hypotheses of the Main Theorem. Then it satisfies Assumption~\ref{inverse_ok} (which is just a reformulation of these hypotheses) and Assumption~\ref{zero_is_a_weight} (which follows from hypothesis~\ref{itm:main_condition}\ref{itm:fixed_by_l}), so all the work that we have done so far can be applied to $\rho$.

We will construct a positive constant $C \geq 1$ and a family of maps $g_1, \ldots, g_k \in G \ltimes V$ (with $k \geq 2$) that satisfy the conditions \ref{itm:all_are_X0} through \ref{itm:prescribed_marg_inv} and whose linear parts generate a Zariski-dense subgroup of $G$, then we will apply Lemma~\ref{properly_discontinuous}. We will start by constructing a family $\gamma_1, \ldots, \gamma_k$ that satisfies a different set of conditions, then modify it step by step to make it meet all the requirements.
\begin{itemize}
\needspace{\baselineskip}
\item To construct this ``preliminary'' family, we use a result of Benoist: we apply Lemma~4.3.a in~\cite{Ben97} to
\begin{itemize}
\item $\Gamma = G$;
\item $t = k+1$;
\item $\Omega_1 = \cdots = \Omega_k = \mathfrak{a}'_{\rho, X_0} \cap \mathfrak{a}^{++}$, as defined in (6.13*). Here Proposition~\ref{paper_non_vacuous} finally comes in handy: it ensures that this intersection is indeed a nonempty open convex cone.
\end{itemize}
This gives us, for any $k \geq 2$, a family of maps $\gamma_1, \ldots, \gamma_k \in G$ (which we shall see as elements of~$G \ltimes V$, by identifying~$G$ with the stabilizer of~$p_0$), such that:
\begin{hypothenum}
\item Every $\gamma_i$ is compatible with~$X_0$. In particular, by Proposition~\ref{asymptotically_contracting_to_regular}~(ii), it is $\rho$-regular (this is \ref{itm:all_are_X0}).
\item For any two indices $i$, $i'$ and signs $\sigma$, $\sigma'$ such that $(i', \sigma') \neq (i, -\sigma)$, the pair
\[(y_{\gamma_i^\sigma}^{X_0, +},\; y_{\gamma_{i'}^{\sigma'}}^{X_0, -}) \in G/P^+_{X_0} \times G/P^-_{X_0}\]
is transverse.
\item Any single $\gamma_i$ generates a Zariski-connected group.
\item All of the $\gamma_i$ generate together a Zariski-dense subgroup of $G$.
\end{hypothenum}
Since in our case $t$~is finite, Benoist's item~(v) is not relevant to us.

A comment about item~(i): Benoist's theorem only works with $\mathbb{R}$-regular elements; so it forces every $\gamma_i$~to be not only $\rho$-regular, but actually $\mathbb{R}$-regular (of which we make no use).

A comment about item~(ii): since we have taken Benoist's $\Gamma$ to be the whole group~$G$, we have $\theta = \Pi$, so that $Y_\Gamma$ is the full flag variety~$G/P^+$. So Benoist's theorem actually gives us the stronger property that the pair
\[(y_{\gamma_i^\sigma}^{X, +},\; y_{\gamma_{i'}^{\sigma'}}^{X, -}) \in G/P^+ \times G/P^-\]
(where $X$ is some element of the open Weyl chamber~$\mathfrak{a}^{++}$) is transverse in the \emph{full} flag variety. Once again, we only need the weaker version.

\item Now we choose the value of~$C$, so as to ensure condition~\ref{itm:pairwise_C_non_deg}. Using Remark~\ref{affine_flag_varieties}, condition (ii) above may be restated in the following way: for any two indices $i$, $i'$ and signs $\sigma$, $\sigma'$ such that $(i', \sigma') \neq (i, -\sigma)$, the pair of parabolic spaces
\[(V^\subgap_{\gamma_i^\sigma}, V^\sublap_{\gamma_{i'}^{\sigma'}})\]
is transverse. Clearly, every pair of transverse spaces is $C$-non-degenerate for some finite $C$. Since we have a finite number of such pairs, we can simply take as~$C$ the largest of these values.
\item Next we replace every $\gamma_i$ by some power $\gamma_i^N$, so as to ensure~\ref{itm:all_are_contracting}; in fact, we will even require that for every $i$, we have $s_{X_0}(\gamma_i^{\pm N}) \leq s_{\textnormal{Main}}(C)$ for an even smaller constant $s_{\textnormal{Main}}(C)$, to be specified in the next step. By Proposition~\ref{contraction_strength_limit}~(ii), we can indeed ensure these inequalities by choosing $N$ large enough. On the other hand, from condition~(iii) (Zariski-connectedness), it follows that any algebraic group containing some power $\gamma_i^N$ of some generator must actually contain the generator~$\gamma_i$ itself; so this substitution preserves condition (iv) (Zariski-density). Clearly, conditions (i), (ii) and (iii) are then preserved as well.
\item Finally, to satisfy \ref{itm:prescribed_marg_inv}, we replace the maps $\gamma_i$ by the maps
\begin{equation}
g_i := \tau_{\phi_i^{-1}(M_C)} \circ \gamma_i
\end{equation}
(for $1 \leq i \leq k$), where $\phi_i$ is a canonizing map for $\gamma_i$.

We need to check that this does not break the first three conditions. Indeed, for every $i$, we have $\gamma_i = \ell(g_i)$; even better, since the affine map~$\phi_i g_i \phi_i^{-1}$ has by construction canonical form, $g_i$ has the same geometry as $\gamma_i$ (meaning that $A^{\subgap}_{g_i} = A^{\subgap}_{\gamma_i} = V^{\subgap}_{\gamma_i} \oplus \mathbb{R} p_0$ and $A^{\sublap}_{g_i} = A^{\sublap}_{\gamma_i} = V^{\sublap}_{\gamma_i} \oplus \mathbb{R} p_0$). Hence the $g_i$ still satisfy the hypotheses \ref{itm:all_are_X0} and \ref{itm:pairwise_C_non_deg}, but now we have $M(g_i) = M_C$ (this is \ref{itm:prescribed_marg_inv}). As for contraction strength along~$X_0$, we have:
\begin{align}
s_{X_0}(g_i) &= 
\left\| \restr{g_i}     {V^{\subll}_{g_i}}  \right\|
\left\| \restr{g_i^{-1}}{A^{\subgap}_{g_i}} \right\| \nonumber \\
       &=
\left\| \restr{\gamma_i}     {V^{\subll}_{\gamma_i}}  \right\|
\left\| \restr{(\tau_{\phi_i^{-1}(M_C)} \circ \gamma_i)^{-1}}{A^{\subgap}_{g_i}} \right\| \nonumber \\
       &\leq s_{X_0}(\gamma_i) \left\| \phi_i^{-1} \circ \tau_{M_C}^{-1} \circ \phi_i \right\| \nonumber \\
       &\lesssim_C s_{\textnormal{Main}}(C)\|\tau_{M_C}^{-1}\|,
\end{align}
and similarly for $g_i^{-1}$. Recall that $\|M_C\| = 2k_{\ref{invariant_additivity_only}}(2C)$, hence the quantity
\[\|\tau_{M_C}\| = \|\tau_{M_C}^{-1}\| = \|\tau_{-M_C}\|\]
depends only on $C$: in fact it is equal to the norm of the 2-by-2 matrix $\bigl( \begin{smallmatrix} 1 & \|M_C\| \\ 0 & 1 \end{smallmatrix} \bigr)$. It follows that if we choose
\begin{equation}
s_{\textnormal{Main}}(C) \leq s_{\ref{Schottky_group}}(C) \left\| \begin{matrix} 1 & 2k_{\ref{invariant_additivity_only}}(2C) \\ 0 & 1 \end{matrix} \right\|^{-1},
\end{equation}
then the hypothesis \ref{itm:all_are_contracting} is satisfied.

We conclude that the group generated by the elements $g_1, \ldots, g_k$ acts properly discontinuously (by Lemma~\ref{properly_discontinuous}), is free (by the same result), nonabelian (since $k \geq 2$), and has linear part Zariski-dense in $G$. \qedhere
\end{itemize}
\end{proof}

\bibliographystyle{alpha}
\bibliography{/home/ilia/Documents/Travaux_mathematiques/mybibliography.bib}

\end{document}